%% file: Total_variation_burgers_v3.tex
\documentclass[10pt,a4paper,final]{amsart}
\usepackage{geometry, empheq}
\usepackage[english]{babel}
\usepackage{array}
\newcolumntype{C}[1]{>{\centering\let\newline\\\arraybackslash\hspace{0pt}}m{#1}}
\usepackage{amsmath}
\usepackage{showkeys}
\usepackage{amssymb}
\usepackage{amsthm}
\usepackage[demo]{graphicx}
\usepackage{color}
\usepackage{varioref}
\usepackage{nicefrac}
\usepackage{caption}
\usepackage{mathtools}
\usepackage{subcaption}
\usepackage{float}
\usepackage[normalem]{ulem}
\usepackage{multirow}
\usepackage{multicol}
\usepackage{amssymb}
\usepackage[scaled]{beramono}
\usepackage[captionpos=b,framed,numbered,autolinebreaks,useliterate]{mcode}
\usepackage{mathdots}% http://ctan.org/pkg/mathdots
\usepackage{upquote}
\usepackage{caption}
\usepackage{color,xcolor}
\usepackage{listings}
\usepackage{algorithmicx}
\usepackage{algorithm}
\usepackage{algpseudocode}
\usepackage{algpascal}
\usepackage{enumitem}
\usepackage{mathtools}
\usepackage{pgfplots}
\pgfplotsset{width=6.5cm,compat=newest}
\usepackage{pgfplotstable}
\usepackage{tikz}
\pgfplotsset{compat=1.6}
\pgfplotsset{soldot1/.style={color=red,only marks,mark=*}} \pgfplotsset{holdot1/.style={color=red,fill=white,only marks,mark=*}}
\pgfplotsset{soldot/.style={color=blue,only marks,mark=*}} \pgfplotsset{holdot/.style={color=blue,fill=white,only marks,mark=*}}
\newtheorem{lemma}{{Lemma}}
\newtheorem{theorem}{{Theorem}}
\newtheorem{definition}{{Definition}}
\newtheorem{proposition}{{Proposition}}

\newtheorem{remark}{{Remark}}

\title[TGV-regularization for variational data assimilation]{Total generalized variation regularization in variational data assimilation for Burgers' equation}
\author{J.C. De los Reyes$^\ddag$ and E. Loayza$^{\dag,\ddag}$ }
\address{$^\ddag$Research Center on Mathematical Modelling (MODEMAT), Escuela Polit\'ecnica Nacional, Quito, Ecuador}
\address{$^\dag$ Faculty of Mathematics, Technical University of Chemnitz, Chemnitz, Germany}
\begin{document}
\smallskip
\begin{abstract}
We propose a second-order total generalized variation (TGV) regularization for the reconstruction of the initial condition in variational data assimilation problems. After showing the equivalence between TGV regularization and the Bayesian method for the MAP estimator, we focus on the detailed study of the inviscid Burgers' data assimilation problem. Due to the difficult structure of the governing hyperbolic conservation law, we consider a discretize--then--optimize approach and derive first-order optimality conditions for the problem. For the numerical solution, we propose a globalized reduced Newton-type method and prove convergence of the algorithm to stationary points. The paper finishes with some numerical experiments where among others, the performance of TGV--regularization compared to the TV--regularization is tested.
\end{abstract}
\maketitle

%%%%%%%%%%%%%%%%%%%%%%%%%%%%%%
%%%%%%%%%%%%%%%%%%%%%%%%%%%%%%

\section{Introduction}
Data assimilation can be described as the process in which one aims to find an approximation of the initial condition of a dynamical system using state observations and background information, in order to obtain improved numerical forecasts of the system. The resulting inverse problems are typically ill--posed, since one seeks a solution using only partial observations of the state. As a remedy, regularization approaches such as filtering, a posteriori regularization, variational regularization, etc., are usually considered \cite{lewisdataassim,siltanenbook,vogelinversebook}.

In this paper we focus our attention on the variational data assimilation framework, whose main idea consists in solving an optimization problem that fits the observations, on the one hand, and uses background information of previous forecasts, on the other. The background term usually plays the role of a regularizer, if one compares with an inverse problems methodology. Clasically, variational data assimilation problems arise from a robust Bayesian estimation of the initial condition \cite{kalnay2003atmospheric,lewisdataassim}. If the observations are taken at one solely time instant, one obtains a so-called 3D-VAR problem, while if a whole time window is considered, the so called 4D-VAR problem is obtained.

A generic mathematical formulation of the 4D-VAR data assimilation problem consists in solving the least squares problem:
\begin{equation*}\label{eq:00}
\left\{
\begin{array}{ll}
\displaystyle \min_{u}& ~\dfrac{1}{2} \displaystyle\sum_{i=1}^N (z_i-\mathcal{H}_i(y_i))^T R_i^{-1}(z_i-\mathcal{H}_i(y_i))+ \dfrac{1}{2}(u-u^b)^TB^{-1}(u-u^b) \\
\hbox{subject to: }& \\
& y_{i}=u,\qquad\qquad\qquad i=1,\\
& y_{i+1}=\mathcal{M}_{i+1}(y_i),\quad\forall i=\{1,\ldots,N\},
\end{array}
\right.
\end{equation*}
where $N$ is the number of time instants at which the observations are taken and each vector $z_i$ consists of the observations spread over the spatial mesh at the \emph{i}-th time instant. Furthermore, $B$ is the covariance matrix of the background error, $R_i$ the covariance matrix for the observational error, the operator  $\mathcal{H}_i(\cdot)$ is the observation operator that maps the model state to the observation space for each instant of time and $\mathcal{M}_i(\cdot)$ is the nonlinear state operator for each instant.

Although 4D-VAR has been succesfully tested and proved to be of operational utility in numerical weather prediction, there are some shortcomings of the standard approach, related to the fact that no continuity of the solutions can in general be expected. In fact, the reconstruction of discontinuities, in form of so-called sharp fronts, are of key importance, and unsuccesfully reconstructed with the standard methodology in Numerical Weather Prediction. With this in mind, alternative regularizations such as Total Variation (TV) have been previously considered in \cite{BuddFreitag2011,freitag2010resolution}, where a comparison with respect to more standard regularizations has been carried out with promising results.

Total variation, however, has its downside too, since it performs well mainly when the reconstructed solution has a piecewise constant structure, something that cannot in general be expected in fluid dynamics. To overcome the piecewise constant reconstruction effect, higher order regularizers have been considered in recent years, with the total generalized variation (TGV) as one of their main representatives \cite{bredies2011inverse,bredies2010total,knoll2011second,bredies2013spatially}.
The second-order total generalized variation regularization depends on two positive parameters $\alpha>0$ and $\beta>0$ and has the following structure:
\begin{equation}\label{eq:tgv reg}
  \mathcal{R}(u)=\min_{w\in\mathbb{R}^{n-1}}\alpha\sum_{i=1}^{n-1} |D_{i}u-w_i|+\beta\sum_{i=1}^{n-1}| E_{i}w|,\qquad\qquad u\in \mathbb{R}^n,
\end{equation}
where the matrix $D$ corresponds to the discrete gradient and $E$ to the symmetrized discrete gradient.

In this paper we investigate the use of second-order total generalized variation regularization for the solution of data assimilation problems constrained by the inviscid Burgers equation. This equation is a first model for turbulence and a standard toy problem for the dynamics of the atmosphere. Due to its hyperbolic conservation law structure, we adopt a discretize-then-optimize approach for the solution of the problem, to avoid complications related to the well-posedness of the equation in function spaces \cite{pfaff2015optimal}. A similar approach has been previously pursued in \cite{apte2010variational}.

Specifically, we consider the problem
\begin{equation}\label{eq:02}
\left\{
\begin{array}{ll}
\min ~\dfrac{1}{2} \displaystyle\sum_{i=1}^N (z_i-\mathcal{H}_i(y_i))^T R_i^{-1}(z_i-\mathcal{H}_i(y_i)) +\dfrac{1}{2}(u-u^b)^TB^{-1}(u-u^b) + \mathcal{R}(u)\\
\hbox{subject to the discretized Burgers equation,}%~\eqref{eq:burgers}}
\end{array}
\right.
\end{equation}
where $\mathcal{R}(\cdot)$ is the regularization term given in \eqref{eq:tgv reg}. For problem \eqref{eq:02} we analyze the well-posedness of the discretized equation, existence of optimal solutions, well-posedness of the adjoint equation and the characterization of local minima by means of a first order optimality system of Karush-Kuhn-Tucker type.

The most challenging aspect concerns the numerical solution of the problem. In order to devise a fast superlinear convergent algorithm, we consider a Huber regularization of the total generalized variation and start by verifying the consistency of the approximation. Each regularized problem is solved by means of a reduced Newton method based on a modified primal-dual formulation of the optimality system. The modification concerns the second-order iteration matrix, which in its original form does not guarantee descent in each iteration. The global convergence of the resulting algorithm is demonstrated.

The organization of the paper is as follows. Section 2 is devoted to discuss the equivalence of the TGV regularized problem with the Bayesian method for the MAP estimation. In section 3 we analyze the discretization of the Burgers equation as well as the corresponding data assimilation problem. The next section is devoted to the presentation of a globalized Newton-type method considering dual variables and projections, and its convergence study is carried out.  In the last section, we present  numerical experiments that reveal the improvement in the recovery of the solutions when we use the TGV--regularization and test the performance of the proposed algorithm.

%%%%%%%%%%%%%%%%%%%%%%%%%%%%%%
%%%%%%%%%%%%%%%%%%%%%%%%%%%%%%

\section{Maximum A Posteriori estimation }
This section aims to show that problem~\eqref{eq:02} has a statistical interpretation and can be derived using a Bayesian approach to obtain the Maximum A Posteriori (MAP) estimator. With this purpose, let us start by writing the TGV regularized data assimilation problem in generic form as:
\[
J(x)=\dfrac{1}{2}(\mathcal{E}(x)-\bar{z})^TG^{-1}(\mathcal{E}(x)-\bar{z})+\sum_{i=1}^{2n-1}|\mathbb{D}_{i}x|
\]
with $x=(u,w)$, $\bar{z}=(z,u^b)$, $\mathcal{E}(x)=\mathcal{E}(u,w)=(\mathcal{H}(\mathcal{M}(u)),u)$ and the matrices
\[
G^{-1}=\left[\begin{array}{cc}
R^{-1}&0\\0&B^{-1}
\end{array}\right]\quad \text{and}\quad
\mathbb{D}=
\left[\begin{array}{cc}
\alpha D&-\alpha \mathbb{I} \\
0& \beta E \end{array}\right].
\]

As mentioned in the introduction, we expect the solutions of the data assimilation problem to allow discontinuities. This information is in fact what we know \emph{a priori} about the solution, and we use it for the MAP estimation. The next lines were developed following the ideas in~\cite{lee2013bayesian}.

In the Bayesian method, unlike the maximum likelihood estimation, we assume that the unknown is a random variable and we know a priori its probability density function $p(x)$ and the conditional probability density function (c.p.d.f.) of the observations $\bar{z}$ given $x$, denoted by $p(\bar{z}|x)$.  The basic idea of this approach is to combine the information from both probability density functions using Bayes' formula in order to obtain an \emph{a posteriori} probability. In this case, it is assumed that the observations and background information have normally distributed errors with covariance matrices $R$ and $B$, respectively. Therefore, the conditional probability density function (c.p.d.f.) is given by the following gaussian distribution:
\[
p(\bar{z}|x)=\dfrac{1}{(2\pi)^{\nicefrac{n+m}{2}} |G|^{\nicefrac{1}{2}}}\exp\left\{-\dfrac{1}{2}(\bar{z}-\mathcal{E}(x))^TG^{-1}(\bar{z}-\mathcal{E}(x))\right\}.
\]
Now, we need to look for a probability distribution that guarantees us to capture the expected main features of the solutions. Since these solutions are expected to be discontinuous, many previous works suggest to use the Laplace distribution~\cite{BuddFreitag2011,lee2013bayesian}. Specifically, it is assumed that the solution is a realization of a random process defined by the prior probability distribution function given  by the Laplace one for the matrix $\mathbb{D}x$ with parameters $\theta=1$ and $\mu=0$:
\[
p(x)=\dfrac{1}{2\theta}\exp\left\{\dfrac{1}{\theta} \sum_{i=1}^{2n-1} |\mathbb{D}_{i}x-\mu|\right\}.
\]
Using Bayes' formula we get that the \emph{a posteriori} distribution of $x$ given the observations $\bar{z}$ is
\[
p(x|\bar{z})=\dfrac{p(\bar{z}|x)p(x)}{p(\bar{z})}.
\]
Since we aim to find the best ``Bayesian estimator" using the maximum a posteriori approach, we maximize the previously obtained \emph{a posteriori} distribution as follows
\begin{equation}
  \max_{x}~p(x|\bar{z})=\max_{x}~\dfrac{p(\bar{z}|x)p(x)}{p(\bar{z})}.
\end{equation}

The distribution $p(\bar{z})$ is the marginal probability distribution and it does not depend on $x$. Thus, we optimize
\begin{equation} \label{eq: margin prob}
  \max_{x}~p(\bar{z}|x)p(x),
\end{equation}

Since the logarithmic function is monotonically increasing and continuous, solving \eqref{eq: margin prob} is equivalent to maximize the natural logarithm of the objective function, that is
\[
\max_{x}~\ln(p(\bar{z}|x))+\ln(p(x)).
\]
Replacing the definition of the probability density functions mentioned above, we get
\[
\max_{x}~-\left\{\dfrac{1}{2}(\bar{z}-\mathcal{E}(x))^TG^{-1}(\bar{z}-\mathcal{E}(x))+\sum_{i=1}^{2n-1}|\mathbb{D}_{i}x|\right\}
\]
and, consequently, the equivalence with the 4D--VAR problem with TGV--regularization.
%that is the objective function of the problem will be given by the following expression:
%\begin{align*}
%J(\mathbf{y},u,w)&{}=\displaystyle\sum_{i=1}^N (z_i-\mathcal{H}_i(y_i))^T R_i^{-1}(z_i-\mathcal{H}_i(y_i))+ \dfrac{1}{2}(u-u^b)^T B^{-1}(u-u^b)\\
%&{}\displaystyle + \alpha\sum_{i=1}^{n-1}|D_iu-w_i|+\beta \sum_{i=1}^{n-1}|E_i w|
%\end{align*}
%and this is the objective function that we are going to use along the paper.

%%%%%%%%%%%%%%%%%%%%%%%%%%%%%%
%%%%%%%%%%%%%%%%%%%%%%%%%%%%%%

\section{Discrete--then--optimize approach}
This section is devoted to the analysis of the data assimilation problem for the Burgers equation. We begin with the discretization of the state equation using a standard finite difference scheme for the spatial discretization and a semi--implicit Euler for the time discretization.

\subsection{Discretization of the Burgers equation}
 The Burgers equation was developed in 1939 as a simplified version of the Navier-Stokes equation in order to better understand turbulence. The equation has a nonlinear convective term that makes it difficult to solve. The inviscid one-dimensional equation is given by the following expression:
\begin{equation}\label{eq:burgers}
  \begin{cases}
    \dfrac{\partial y}{\partial t} + y \dfrac{\partial y}{\partial x} =f & \hbox{ in }Q=\Omega\times (0,T), \\
    y=0 & \hbox{ on } \Gamma\times (0,T),\nonumber \\
    y(x,0)=u(x) & \hbox{ in } \Omega. \nonumber
  \end{cases}
\end{equation}

In order to solve the latter both a spatial and a time discretization schemes have to be utilized.
For the spatial discretization several methodologies, like finite differences, finite elements or finite volumes, have been classically considered. Here we focus on a finite differences scheme with a homogeneous partition of the domain $\Omega=(0,1)$, whose discretization points are given by
\[
x_i=ih, \qquad \forall i=1,\ldots, n
\]
with $h=\nicefrac{1}{n+1}$ and $n$ the amount of spatial discretization points.
For the approximation of the first derivative of $y_i$ we use a standard forward finite difference scheme, with the corresponding matrix denoted by $U$.

% The main feature of the solutions that we expect for this kind of problems is the non--continuity. For that reason, it is important

For the time discretization, in order to guarantee convergence towards a (possibly) discontinuous solution, we focus on semi-implicit schemes. Moreover, in order to avoid the solution of the non--linearity in each time step, we consider a semi--implicit Euler scheme given by:
\[
\dfrac{y^{j+1}-y^j}{\Delta t} + \hbox{diag}(y^j)U(y^{j+1})=f^{j+1},
\]
%{\color{red}no debería ser $diag(y^j)U y^{j+1}$ en lugar de $y^jU(y^{j+1})$}

where $\Delta t=1/(N_t+1)$ and $N_t$ is the number of time discretization points. Choosing this kind of discretization the problem reduces to the solution of a linear system at each time discretization point, with the following structure:
\begin{equation}\label{eq:linear_sys}
y^{j+1} + \Delta t\hbox{diag}(y^j)Uy^{j+1}=\Delta t f^{j+1}+y^j.
\end{equation}

Considering both the spatial and time discretization, we can rewrite the matrix $y\in\mathbb{R}^{n\times N_t}$ as a vector with the following structure
\[
\mathbf{y}=\left[
\begin{array}{c}
y^1\\ \vdots\\ y^{N_t}
\end{array}
\right],
\]
where $y^i\in \mathbb{R}^n$, thus $\mathbf{y}\in \mathbb{R}^m$ with $m=n\cdot N_t$. In this context, we define the new matrices
\begin{equation}\label{eq:matrices}
\mathbb{E}=\dfrac{1}{\Delta t}\left[\begin{array}{ccccc}
\mathbb{I}&0&0&\ldots&0\\
-\mathbb{I}&\mathbb{I}&0&\ldots&0\\
0&-\mathbb{I}&\mathbb{I}&\ldots&0\\
\vdots&\vdots&\vdots&\ddots&\vdots\\
0&0&0&\ldots&\mathbb{I}
\end{array}\right],\qquad\mathbb{U}=\left[\begin{array}{cccc}
0&0&\ldots&0\\
0&U&\ldots&0\\
\vdots&\vdots&\ddots&\vdots\\
0&0&\ldots&U
\end{array}\right]
\end{equation}
\[
\mathbb{Z}=\left[\begin{array}{cccc}
0&0&\ldots&0\\
0&\hbox{diag}(y^1)&\ldots&0\\
\vdots&\vdots&\ddots&\vdots\\
0&0&\ldots&\hbox{diag }(y^{N_{t}-1})
\end{array}\right],\qquad\mathbf{f}=\left[ \begin{array}{c}
\nicefrac{u}{\Delta t}\\f^2\\ \vdots\\ f^{N_t}
\end{array}\right].
\]
Thus, the system that we have to solve is given by:
\begin{equation}\label{eq:fulldis_eq}
e(\mathbf{y},u)=0,
\end{equation}
where
\begin{subequations}\label{eq:burgers op}
  \begin{align}
  e: \mathbb R^m \times \mathbb R^n & \to \mathbb R^m\\
  (\mathbf{y},u) & \rightarrow \mathbb{E}\mathbf{y}+\mathbb{Z}(\mathbf{y})\mathbb{U}\mathbf{y}-\mathbf{f}.
  \end{align}
\end{subequations}

In order to prove the local existence of solutions of the state equation we will use the implict function theorem (see e.g.~\cite{ciarlet2013linear}). For this purpose, we start by proving the following lemmata.
\begin{lemma}\label{lm:G_invertible}
Let $G_j$ be the matrix defined by
\begin{equation}\label{eq:Gmatrix}
G_j=\dfrac{1}{\Delta t}\mathbb{I}+ \hbox{diag }(y^{j})U^T+\hbox{diag }(Uy^{j+1}), \forall j=1,\ldots,N_t-1.
\end{equation}
If $\hbox{diag }(Uy^{j+1})\geq 0$, then $G_j$ is invertible for each $j=1,\ldots, N_t-1$.
\end{lemma}
\begin{proof}
Let's start noticing that the matrix $G_j$ is a bi--diagonal matrix given by the following expression
%{\color{red} no entiendo bien la estructura de la matriz $G_j$; puedes porfa aclarar eso}
{\tiny{\[
G_j=\left[\begin{array}{cccccc}
\nicefrac{1}{\Delta t}+ \nicefrac{y_1^j}{h}+\nicefrac{y_1^{j+1}}{h}&0&0&\cdots&0&0\\
-\nicefrac{y_2^j}{h}&\nicefrac{1}{\Delta t}+ \nicefrac{y_2^j}{h}+U_2y^{j+1}&0&\cdots&0&0\\
0&-\nicefrac{y_3^j}{h}&\nicefrac{1}{\Delta t}+ \nicefrac{y_3^j}{h}+U_3y^{j+1}&\cdots&0&0\\
\vdots&\vdots&\ddots&\ddots&\vdots&\vdots\\
0&0&0&\cdots&-\nicefrac{y_n^j}{h}&\nicefrac{1}{\Delta t}+ \nicefrac{y_n^j}{h}+U_ny^{j+1}
\end{array}\right]
\] }}
%\[
%G_j=\hbox{bidiag }([\nicefrac{1}{\Delta t}+ \nicefrac{y_1^j}{h}+\nicefrac{y_1^{j+1}}{h},\nicefrac{1}{\Delta t}+ \nicefrac{y_2^j}{h}+U_2y^{j+1},\ldots,\nicefrac{1}{\Delta t}+ \nicefrac{y_n^j}{h}+U_ny^{j+1}],[-\nicefrac{y_2^j}{h},\ldots,-\nicefrac{y_n^j}{h}],0,-1)
%\]
where $U_i$ is the $i$--th row of the matrix $U$. Therefore, the determinant of this matrix is given by
\[
\det(G_j)=\left(\nicefrac{1}{\Delta t}+\nicefrac{y_1^j}{h}+\nicefrac{y_1^{j+1}}{h}\right)\prod_{i=2}^{n}\left(\nicefrac{1}{\Delta t}+\nicefrac{y_i^j}{h}+U_iy^{j+1}\right)
\]
and since $\hbox{diag }(Uy^{j+1})\geq 0$ we can conclude that $\det (G_j)\neq 0$  and therefore the matrix $G_j$ is invertible for all $j=1,\ldots, N_t-1$.\qedhere.
\end{proof}
\begin{lemma}\label{lm:e_y_invertible}
The derivative operator of the state equation with respect of $\mathbf{y}$, given by
\[
e_y(\mathbf{y},u)=\mathbb{E} +\mathbb{Z}\mathbb{U}+\hbox{diag }(\mathbb{U}\mathbf{y}),
\]
is invertible.
\end{lemma}
\begin{proof}
We will prove that $e_y(\mathbf{y},u)$ is nonsingular. Let  $\mathbf{v}=[v^1,v^2,\ldots,v^{N_t}]\in \mathbb{R}^m$ be an arbitrary vector such that
\[
e_y(\mathbf{y},u)\mathbf{v}=0.
\]
We need to prove that $\mathbf{v}=0$. Analyzing the previous equation component--wise in time, one can distinguish the following cases:
\begin{description}
\item[$j=1$]

Directly from the definition of the matrices $\mathbb{E,Z,U}$ we get that
\[
v^1=0;
\]
\item[$j=2,\ldots,N_t$]

In this case the system can be written as
\[
\dfrac{v^{j}-v^{j-1}}{\Delta t} + \hbox{diag }(y^{j-1})Uv^{j} +\hbox{diag }(Uy^{j})v^{j}=0
\]
or equivalently,
\[
\left(\dfrac{1}{\Delta t}\mathbb{I}+ \hbox{diag }(y^{j-1})U^T+\hbox{diag }(Uy^{j})\right)v^{j}=G_jv^{j}=\dfrac{v^{j-1}}{\Delta t}.
\]
where $G_j$ is given in~\eqref{eq:Gmatrix}. Moreover, from Lemma~\ref{lm:G_invertible} $G_j$ is invertible and hence
\[
v^{j}=\dfrac{G_j^{-1}v^{j-1}}{\Delta t}
\]
Using a recursive argument we can conclude that $v^{j}=0$ for all $j=2,\ldots, N_t$, and therefore $e_{y}(\mathbf{y},u)$ is invertible.\qedhere
\end{description}
\end{proof}
\begin{proposition}\label{prop:existence_burgers}
Let $e(\mathbf{y},u)$ be the discretized Burgers equation operator given by~\eqref{eq:burgers op} and $(\bar{\mathbf{y}},\bar{u})$ such that $e(\bar{\mathbf{y}},\bar{u})=0$. Then there exists a neighborhood $\mathcal{U}$ of $\bar{u}$ and a continuously differentiable implicit function $\mathbf{y}(u)$ such that $e(\mathbf{y}(u),u)=0$, for all $u\in\mathcal{U}$.
\end{proposition}
% The proof of the result holds automatically by using the Implicit Function Theorem since $e(\mathbf{y},u)$ is continuously differentiable and from Lemma~\ref{lm:e_y_invertible} $e_y(\mathbf{y},u)$ is invertible.
\begin{proof}
As we mentioned in the beginning of the section, we will use the implicit function theorem. Due to the definitions of the matrices $\mathbb{E,Z,U}$ we know that $e(\mathbf{y},u)$ is continuously differentiable and its derivatives are given by
\[
e_y(\mathbf{y},u)=\mathbb{E} +\mathbb{Z}\mathbb{U}+\hbox{diag }(\mathbb{U}\mathbf{y})
\]
and $e_u(\mathbf{y},u)=[-\nicefrac{1}{\Delta t}\mathbb{I},0,0,\ldots,0]^T$. Let $(\bar{\mathbf{y}},\bar{u})$ be such that $e(\bar{\mathbf{y}},\bar{u})=0$.
From Lemma \ref{lm:e_y_invertible} we know that $e_y(\bar{\mathbf{y}},\bar{u})$ is invertible for any $(\mathbf{y},u)$ and, from the implicit function theorem, there exists an open neighborhood $\mathcal{U}$ of $\bar{u}$ and a function $\mathbf{y}(u)$ continuously differentiable such that $e(\mathbf{y}(u),u)=0$, for all $u\in\mathcal{U}$.
\end{proof}
%Once we had proved the existence of solutions, we can prove different results such as the boundedness of $\mathbf{y}$ which is proved in the following Lemma.
In the next Lemma we prove the boundedness of $\mathbf{y}$, a necessary result for the analysis of  the consistency of the regularization.
\begin{lemma}\label{lem:boundedness_y}
Let $\mathbf{y}$ be the solution of the state equation $e(\mathbf{y},u)=0$. The following inequality holds
\begin{equation} \label{eq: state bound}
  \parallel y^j\parallel\leq \parallel u\parallel + \Delta t \sum_{i=1}^j \parallel f^i\parallel \qquad\forall j=1,\ldots,N_t.
\end{equation}
\end{lemma}
\begin{proof}
We start by recalling the state equation given by
\[
e(\mathbf{y},u)=\mathbb{E}\mathbf{y}+\mathbb{Z}(\mathbf{y})\mathbb{U}\mathbf{y}-\mathbf{f}(u)=0
\]
where
$\mathbf{y}=(
y^1,\dots, y^{N_t})^T$ and $\mathbf{f}(u)=(\nicefrac{u}{\Delta t},f^2, \dots, f^{N_t})^T.$
Performing a component-wise analysis we distinguish two cases:
\begin{description}
\item[For $j=1$]
\[
\dfrac{y^1}{\Delta t}-\dfrac{u}{\Delta t}=0.
\]
\item[For $j=2,\ldots,N_t$]
\[
\dfrac{y^j-y^{j-1}}{\Delta t} + \hbox{diag }(y^{j-1})U y^j - f^j=0.
\]
\end{description}
Multiplying the previous equality with $(y^j)^T$ we get
\[
(y^j)^Ty^j - (y^j)^T y^{j-1}+ \Delta t (y^j)^T \hbox{diag }(y^{j-1})Uy^j- \Delta t (y^j)^T f^j=0
\]
Now, we analyze the second term of the previous equality, using the definition of the matrix $U$. With out lose of generality we assume that $U$ is given by the forward finite difference scheme. Hence,
\[
(y^j)^T \hbox{diag }(y^{j-1})Uy^j=\sum_{i=1}^n y^j_i y^{j-1}_i \left(\dfrac{y^j_{i+1}-y^j_{i}}{h}\right)
\]
Using the well--known summation by parts (see e.g.~\cite{knuth1997art}) we have
\[
\sum_{i=1}^n y^j_i y^{j-1}_i \left(\dfrac{y^j_{i+1}-y^j_{i}}{h}\right)=\dfrac{1}{h}[y^j_ny^{j-1}_n-y^j_1y^{j-1}_1]-\sum_{i=1}^n y^j_i y^{j-1}_i \left(\dfrac{y^j_{i+1}-y^j_{i}}{h}\right).
\]
Due to the boundary conditions we get that $y^j_n=y^j_1=0$ for all $j=1,\ldots,N_t$. Therefore,
\[
\sum_{i=1}^n y^j_i y^{j-1}_i \left(\dfrac{y^j_{i+1}-y^j_{i}}{h}\right)=0.
\]
Consequently,
\[
\parallel y^j\parallel^2 = (y^j)^T(y^{j-1})+\Delta t (y^j)^Tf^j,
\]
and, using Cauchy--Schwarz inequality, we get
\[
\parallel y^j\parallel \leq \parallel y^{j-1}\parallel + \Delta t \parallel f^j\parallel, \quad \text{ for all }j=1,\ldots,N_t.
\]
Using the discrete Gronwall inequality we finally obtain
%
% Finally, let be $k$ a fixed index, $1\leq k\leq N_t$. Summing over $j$ from 2 to $k$ we have
% \[
% \sum_{j=2}^k\left[\parallel y^j\parallel -\parallel y^{j-1}\parallel \right] \leq \Delta t \sum_{j=2}^k\parallel f^j\parallel.
% \]
% Since the sum in the left hand side of the inequality is a telescoping sum and due to the initial condition of the state equation we can conclude that
\begin{equation*}
  \parallel y^k\parallel \leq \parallel u\parallel +\Delta t \sum_{j=2}^k\parallel f^j\parallel, \quad
  \text{ for all }k=1,\ldots,N_t. \qedhere
\end{equation*}
\end{proof}

%The solution of the equation is reduced to the solution of the nonlinear system above. Because of the size of the matrices described above, the full discretization is not applicable in the computational sense. For that reason, the numerical solution of the state equation was developed solving the $N_t$ linear systems given in the equation~\eqref{eq:linear_sys}. %different from the solution of the fully--discretized nonlinear system above. The difference resides in the solution of the separable linear systems, for each one of the time steps. It is easy to realize that the previously defined nonlinear system can be separated into $N_t$ linear systems of a smaller size.

%\paragraph{Remark.-} The matrix showed in the equation~\eqref{eq:UP} is used only for the numerical solution of the problem. The following analysis will be done assuming that the matrix $\mathbb{U}$ does not depend on the state variable $y$; that is, it represents the discrete gradient.

\subsection{Discretized data assimilation problem}
For the discretization of problem \eqref{eq:02}, we assume that the observation operator $\mathcal{H}(\cdot)$ is linear and, therefore, there exist matrices $S\in\mathbb{R}^{m_o\times (n_o N_t)}$ and $H\in\mathbb{R}^{n_o N_t\times m}$ such that $\mathcal{H}(x)=S Hx$, with $m=nN_t$, $n_o$ the amount of observations taken in space, $N_o$ the number of instants at which we take the observations and $m_o=n_oN_o$. Thus, the problem is given by:
\begin{equation}\label{eq:discrete_problem_tgv}
\left\{
\arraycolsep=1.4pt\def\arraystretch{2.2}
\begin{array}{ll}
\displaystyle \min_{(\mathbf{y},u,w)\in \mathbb{R}^{m}\times \mathbb{R}^n\times\mathbb{R}^{n-1}}&J(\mathbf{y},u,w)=\dfrac{1}{2} (SH\mathbf{y}-\mathbf{z})^T R^{-1}(SH\mathbf{y}-\mathbf{z})+\dfrac{\mu}{2}w^Tw\\
&\displaystyle+ \dfrac{1}{2}(u-u^b)^T B^{-1}(u-u^b)\displaystyle+\alpha\sum_{i=1}^{n-1}|D_iu-w_i|+\beta \sum_{i=1}^{n-1}|E_i w|\\
\hbox{subject to:}
& \mathbb{E}\mathbf{y}+\mathbb{Z}(\mathbf{y})\mathbb{U}\mathbf{y}-\mathbf{f}=0,
\end{array}
\right.
\end{equation}
where the matrix $D$ corresponds to the discrete gradient associated with a forward finite differences scheme and $E$ is the one associated with backward finite differences. It is worth to remark that since we work with 1D--spatial functions the symmetrized discrete gradient match with the usual one. The Tikhonov term $\dfrac{\mu}{2}w^Tw$ is further added to obtain the well-posedness of the inverse problem.

%From the discretization of the Burger's equation, we know that $\mathbf{y}\in \mathbb{R}^{m}$ where $m=n\cdot N_t$, $n$ is the number of discretization points for space and $N_t$ the number of discretization points for the time. Furthermore, as the control is defined just in space in the discrete problem $u\in \mathbb{R}^n$ and $D$ is the discrete gradient of a function given by the \emph{forward} finite difference scheme, that is
%\[
%D=\dfrac{1}{h}\left[\begin{array}{cccccc}
%1&0&0&\dots&0&0\\
%-1&1&0&\dots&0&0\\
%\vdots&\vdots&\vdots&\ddots&\vdots&\vdots\\
%0&0&0&\dots&-1&1
%\end{array}\right],\quad E=\dfrac{1}{h}\left[\begin{array}{ccccc}
%-1&1&0&\dots&0\\
%0&-1&1&\dots&0\\
%\vdots&\vdots&\vdots&\ddots&\vdots\\
%0&0&0&\ldots&1\\
%0&0&0&\dots&-1
%\end{array}\right]
%\]
%%Let us mention that this matrix has $n-1$ rows and $n$ columns in order to guarantee the stability of the discretization. We denote as $D_{i}$ the $i$--th row of the matrix $D$. Furthermore, the matrix $E$ denotes the symmetrized gradient. As we mentioned in the previous sections since we are working with one spatial dimension the symmetrized gradient is equal to the usual gradient.
%In order to compute the optimality system we have to compute the derivatives of the objective function $J_\gamma(\mathbf{y},u)$ and of the state equation $e(\mathbf{y},u)$.

\subsection{Adjoint State}
This subsection is devoted to discuss the main features of the adjoint equation. %This analysis will allow us to guarantee the existence and uniqueness of the solutions of the adjoint equation or equivalently the bijectivity of $e_y(\mathbf{y},u)$.
We start by noticing that the adjoint equation is given by the following expression:
\begin{equation}\label{eq:adjoint_state}
\mathbb{E}^T\mathbf{p}+\mathbb{Z}\mathbb{U}^T \mathbf{p} + \hbox{diag}(\mathbb{U}\mathbf{y})\mathbf{p}-H^TS^TR^{-1}(SH\mathbf{y}-\mathbf{z})=0
\end{equation}
or, equivalently,
\[
e_y(\mathbf{y},u)^T \mathbf{p} =H^TS^TR^{-1}(SH\mathbf{y}-\mathbf{z})
\]
Thanks to Lemma~\ref{lm:e_y_invertible}, which establishes the invertibility of the matrix $e_y(\mathbf{y},u)$, existence of a unique solution to the adjoint equation \eqref{eq:adjoint_state} follows.

In the next Lemma we prove the boundedness of the adjoint state; for this purpose we use the discrete Gronwall's Theorem~\cite{quarteroni2010numerical}.
\begin{lemma}\label{lm:p_inf}
Let $(\mathbf{y,p})$ be the state and adjoint state associated to some control $u$. Let us assume that $diag(U y^j) \geq 0$ and that there exists a constant $c>0$ such that
\begin{equation} \label{eq: hyp bound adjoint}
  1-\Delta t \|U\| \| u \| - |\Delta t|^2 \|U\| \|\mathbf f\|  \geq c.
\end{equation}
Then the following inequality holds
\[
\| \mathbf{p}\| \leq \rho \| H^TS^TR^{-1}(SH\mathbf{y}-\mathbf{z})\|,
\]
with $\rho>0$ a constant independent of $u$ and $\mathbf{y}$.
\end{lemma}
\begin{proof}
We start recalling the adjoint equation given in~\eqref{eq:adjoint_state}
\[
\mathbb{E}^T\mathbf{p}+\mathbb{Z}\mathbb{U}^T \mathbf{p} + \hbox{diag}(\mathbb{U}\mathbf{y})\mathbf{p}-H^TS^TR^{-1}(HS\mathbf{y}-\mathbf{z})=0.
\]
Defining $\mathbf{p}=\mathbb{P}\mathbf{v}$, with $\mathbb{P}$ a permutation matrix given by
\begin{equation}\label{eq:P_per}
\mathbb{P}=\left[
\begin{array}{cccc}
0&0&\ldots&\mathbb{I}\\
\vdots&\vdots&\iddots&\vdots\\
0&\mathbb{I}&\hdots&0\\
\mathbb{I}&0&\hdots&0
\end{array}
\right]
\end{equation}
and $\mathbb{I}$ the identity matrix, we get that
\[
p^j=v^{N_t-j+1}
\]
and therefore the adjoint system in the variable $\mathbf{v}$ is given by
\[
\mathbb{E}^T\mathbb{P}\mathbf{v}+\mathbb{Z}\mathbb{U}^T \mathbb{P}\mathbf{v} + \hbox{diag}(\mathbb{U}\mathbf{y})\mathbb{P}\mathbf{v}-H^TS^TR^{-1}(HS\mathbf{y}-\mathbf{z})=0.
\]
Performing a component-wise analysis in time we distinguish the following cases:
\begin{description}
\item[For $j=1$]
\[
\left(\dfrac{1}{\Delta t}\mathbb{I}+\hbox{diag}(Uy^1)\right)v^{N_t}=\dfrac{1}{\Delta t}v^{N_t-1}+\left[H^TS^TR^{-1}(SH\mathbf{y}-\mathbf{z})\right]_{1}
\]
\item[For $j=2,\ldots,N_t-1$]
\[
\left(\dfrac{1}{\Delta t}\mathbb{I} + \hbox{diag}(y^{j-1})U^T+\hbox{diag }(Uy^{j})\right)v^{N_t-j+1}=\dfrac{1}{\Delta t}v^{N_t-j}+\left[H^TS^TR^{-1}(SH\mathbf{y}-\mathbf{z})\right]_{j}
\]
\item[For $j=N_t$]
\[
\left(\dfrac{1}{\Delta t}\mathbb{I} + \hbox{diag}(y^{N_t-1})U^T+\hbox{diag }(Uy^{N_t})\right)v^1=\left[H^TS^TR^{-1}(SH\mathbf{y}-\mathbf{z})\right]_{N_t}
\]
\end{description}

Multiplying with $v^{N_t-j+1}$ it then follows, for all $j=1,\ldots,N_t-1$, that
\begin{align*}
  &\left(\left(\dfrac{1}{\Delta t}\mathbb{I} + \hbox{diag}(y^{j-1})U^T+\hbox{diag }(Uy^{N_t})\right)v^{N_t-j+1}, v^{N_t-j+1} \right)\\
  &\hspace{1cm} \geq \dfrac{1}{\Delta t} \|v^{N_t-j+1}\|^2 +\left( \hbox{diag}(y^{N_t-j-1})U^T v^{N_t-j+1}, v^{N_t-j+1} \right)\\
  &\hspace{1cm} \geq \dfrac{1}{\Delta t} \|v^{N_t-j+1}\|^2 -\|U\| ~\|\mathbf y\| ~\|v^{N_t-j+1}\|^2\\
  &\hspace{1cm} \geq \dfrac{1}{\Delta t} \|v^{N_t-j+1}\|^2 -\|U\| ~(\|u\| + \Delta t \|\mathbf f\|) ~\|v^{N_t-j+1}\|^2,
\end{align*}
where the last inequality was obtained thanks to the bound \eqref{eq: state bound}.

Since by hypothesis there exists a positive constant $c$ such that \eqref{eq: hyp bound adjoint} holds, we get that
\begin{equation}
  \dfrac{c}{\Delta t} \|v^{N_t-j+1}\| \leq \dfrac{1}{\Delta t} \|v^{N_t-j}\|+ \left\| \left[H^TS^TR^{-1}(SH\mathbf{y}-\mathbf{z})\right]_{j} \right\|.
\end{equation}
Finally, applying the discrete Gronwall inequality the result follows with $\rho=\dfrac{\Delta t}{c}\exp\left(\nicefrac{1}{c}-1\right)$.
\end{proof}

\begin{remark}
  Let us remark that the assumed condition $diag(U y^j) \geq 0$ is not indispensable for getting the result, but allows to get a sharper bound. This condition, however, is of importance for the numerical solution of both the state equation and the data assimilation problem, and will be guaranteed by using an upwinding scheme.
\end{remark}

\subsection{Optimality system}
Next we derive formally a necessary optimality condition for problem \eqref{eq:discrete_problem_tgv} in form of a Karush-Kuhn-Tucker system.

\begin{theorem}\label{teo:opt_sys_tgv}
Let $\bar u$ be an optimal solution for problem \eqref{eq:discrete_problem_tgv}. Then there exists an adjoint state $\bar{\mathbf{p}} \in\mathbb{R}^{m}$ such that the following optimality system is satisfied:
\begin{subequations} \label{eq:discrete_optimality_system}
\begin{align}
    & \mathbb{E}\bar{\mathbf{y}}+\mathbb{Z} \mathbb{U}\bar{\mathbf{y}}-\mathbf{f}=0 \\
    &\mathbb{E}^T\bar{\mathbf{p}}+\mathbb{Z}\mathbb{U}^T \bar{\mathbf{p}} + \hbox{diag}(\mathbb{U}\bar{\mathbf{y}})\bar{\mathbf{p}}-H^TS^TR^{-1}(SH\bar{\mathbf{y}}-\mathbf{z})=0\\
    &\left(B^{-1}(\bar{u}-u^b)+\dfrac{1}{\Delta t}\bar{p}^1,u-\bar{u}\right)+\alpha \displaystyle\sum_{i=1}^{n-1}|D_iu-w_i|-\alpha \displaystyle\sum_{i=1}^{n-1}|D_i \bar u- \bar w_i|\\ &\hspace{2cm}+(\mu \bar{w},w-\bar{w})+ \beta\displaystyle\sum_{i=1}^{n-1}|E_iw|-\beta \displaystyle\sum_{i=1}^{n-1}|E_i \bar w| \geq 0,
    \quad \forall (u,w)\in \mathbb{R}^n \times \mathbb{R}^{n-1}. \nonumber
\end{align}
\end{subequations}
\end{theorem}
\begin{proof}
Due to the nondifferentiability of the objective function we use Theorem 6.1 of~\cite{de2015numerical} in order to compute the optimality system of the problem. Thus, we define the functions
\begin{equation*}
  j_1(u,w)=\dfrac{1}{2} (SH\mathbf{y}-\mathbf{z})^T R^{-1}(SH\mathbf{y}-\mathbf{z})+ \dfrac{1}{2}(u-u^b)^T B^{-1}(u-u^b) + \dfrac{\mu}{2}w^Tw,
\end{equation*}
and
\[
j_2(u,w)=\displaystyle\sum_{i=1}^{n-1}|D_iu-w_i|+\displaystyle\sum_{i=1}^{n-1}|E_iw|.
\]
We notice that the function $j_1(u,w)$ is differentiable and $j_2(u,w)$ is continuous and convex. Thus, using \cite[Thm~6.1]{de2015numerical} we know that the optimality condition of the problem is given by the following variational inequality.
\begin{equation}\label{eq:opt cond primal form}
  (\nabla j_1(\bar{u},\bar{w}),(u,w)-(\bar{u},\bar{w}))+ j_2(u,w)-j_2(\bar{u},\bar{w})\geq 0, \quad \forall (u,w)\in \mathbb{R}^n\times \mathbb{R}^{n-1}.
\end{equation}
For the first term we have
\begin{align}
(\nabla j_1(\bar{u},\bar{w}),(u,w)-(\bar{u},\bar{w})) &= (SH\mathbf{y}-z)R^{-1} SH \mathbf{y}'(\bar{u})(u-\bar{u}) \nonumber\\
&\hspace{1cm} +(\bar{u}-u^b)^T B^{-1}(u-\bar{u})+\mu \bar{w}^T(w-\bar{w})\nonumber\\
&=\left(H^TS^T R^{-1}(SH\mathbf{y}-z),\mathbf{y}'(\bar{u})(u-\bar{u})\right)  \nonumber\\
&\hspace{1cm}+\left( B^{-1}(\bar{u}-u^b),u-\bar{u} \right)+ (\mu \bar{w},w-\bar{w}). \label{eq:cost gradient}
\end{align}

Introducing now the adjoint state $\bar{\mathbf{p}}$ as the unique solution to \eqref{eq:adjoint_state}
% \[
% e_y(\mathbf{y},u)^T\bar{\mathbf{p}}= H^TS^T R^{-1}(SH\mathbf{y}-z)=0
% \]
and replacing this in \eqref{eq:cost gradient}, we get
\begin{align}
(\nabla j_1(\bar{u},\bar{w}),(u,w)-(\bar{u},\bar{w}))&{}=\left(B^{-1}(\bar{u}-u^b),u-\bar{u}\right)+ \left(e_y(\mathbf{y},u)^T\bar{\mathbf{p}},\mathbf{y}' (\bar{u})(u-\bar{u})\right) \nonumber\\
&{}\hspace{1cm}+(\mu\bar{w},w-\bar{w}).
\end{align}
Furthermore, we know that the linearized equation is given by
\begin{equation}\label{eq:lin_eq}
e_y(\mathbf{y}(\bar{u}),\bar{u})\mathbf{y}' (\bar{u})h+e_u(\mathbf{y}(\bar{u}),\bar{u})h=0
\end{equation}

Using the adjoint operator of $e_y(\mathbf{y},\bar{u})$, the linearized equation~\eqref{eq:lin_eq} and the fact that
\[
e_u(\mathbf{y},u)=(
-\dfrac{1}{\Delta t}\mathbb{I}, 0, \dots, 0)^T
\]
we get, together with \eqref{eq:opt cond primal form}, that
\begin{align}
  &\left(B^{-1}(\bar{u}-u^b)+\dfrac{1}{\Delta t}p^1,u-\bar{u}\right)+(\mu\bar{w},w-\bar{w})+\alpha \displaystyle\sum_{i=1}^{n-1}|D_iu-w_i|\\ &\hspace{2cm}-\alpha \displaystyle\sum_{i=1}^{n-1}|D_i \bar u- \bar w_i|+ \beta\displaystyle\sum_{i=1}^{n-1}|E_iw|-\beta \displaystyle\sum_{i=1}^{n-1}|E_i \bar w| \geq 0,
  \quad \forall (u,w)\in \mathbb{R}^n \times \mathbb{R}^{n-1}. \nonumber
\end{align}
\end{proof}

%%%%%%%%%%%%%%%%%%%%%%%%%%%%%%%%%%%%%%%%%%%%%%%
%%%%%%%%%%%%%%%%%%%%%%%%%%%%%%%%%%%%%%%%%%%%%%%

\section{Regularized data assimilation problem}
As a preparatory step for the solution of the data assimilation problem \eqref{eq:discrete_problem_tgv} we consider next a properly regularized version, which consists in replacing the non-differentiable part of the objective function, by a differentiable function. Since we are going to use second order optimization methods we choose a $\mathcal{C}^2$ Huber regularization, which, for $t\in \mathbb{R}$, is given by the following expression:
\begin{equation}\label{eq:huber2}
  H_{\gamma}(t)=\left\{\begin{array}{ll}
  |t|+C_1+K&\hbox{if  }t\in\mathcal{A}
  \\
  \dfrac{\gamma}{2}t^2& \hbox{ if } t\in\mathcal{B}
  \\
  F|t|+\dfrac{G}{2}|t|^2+\dfrac{C}{3}|t|^3+D&\hbox{if }t\in\mathcal{I}
  \end{array}\right.
\end{equation}
where:
  {\small{
  \begin{align*}
  \mathcal{A}:=\left\{t\in \mathbb{R}\colon \gamma|t|\geq 1+ \dfrac{1}{2\gamma}\right\},
  \mathcal{B}:=\left\{t\in \mathbb{R}\colon \gamma |t|\leq 1-\dfrac{1}{2\gamma} \right\},
  \mathcal{I}:=\left\{t\in \mathbb{R}\colon |\gamma|t|-1|\leq \dfrac{1}{2\gamma}\right\},
  \end{align*}}}
with constants:  $l_1=\dfrac{1}{\gamma}\left(1-\dfrac{1}{2\gamma}\right)$, $l_2=\dfrac{1}{\gamma}\left(1+\dfrac{1}{2\gamma}\right)$ and
\[
  \begin{array}{ll}
  F= 1-\dfrac{(2\gamma+1)^2}{8\gamma}&G=\dfrac{\gamma}{2}(2\gamma+1)\\
  C=-\dfrac{\gamma^3}{2}&D=\left(\dfrac{\gamma}{2}-\dfrac{G}{2}\right)l_1^2-F|l_1|-\dfrac{C}{3}|l_1|^3\\
  C_1=\dfrac{\gamma}{2}l_1^2-l_2&K=F(l_2-l_1)+\dfrac{G}{2}(l_2^2-l_1^2)+\dfrac{C}{3}(l_2^3-l_1^3)
  \end{array}
\]
Therefore, the regularized version of the $\ell^1$--norm is given by
\[
  \sum_{i=1}^{n}|x_i|= \sum_{i=1}^{n}H_\gamma(x_i) \quad \text{for } x\in\mathbb{R}^n.
\]
The first derivative of the regularized $\ell^1$--norm is the vector which components are given by:
\begin{equation}\label{eq:der_huber_c2}
(h_\gamma(x))_i=\left\{\begin{array}{ll}
\dfrac{x_i}{|x_i|}&\hbox{if }x_i\in\mathcal{A}
\\
\gamma x_i&\hbox{if }x_i\in\mathcal{B}
\\
\dfrac{x_i}{|x_i|}\left(1-\dfrac{\gamma}{2}\left(1-\gamma|x_i|+\dfrac{1}{2\gamma}\right)^2\right)&\hbox{if }x_i\in\mathcal{I}
\end{array}\right.
\end{equation}
and the second derivative is a diagonal matrix which elements are:
{\small{\begin{equation}\label{eq:2der_huber_c2}
(h'_{\gamma}(x))_{ii}=\left\{\begin{array}{ll}
\dfrac{1}{|x_i|}-\dfrac{x_i\cdot x_i}{|x_i|^3}&\hbox{if }x_i\in\mathcal{A}
\\
\gamma &\hbox{if }x_i\in\mathcal{B}
\\
\left(1-\dfrac{\gamma}{2}\theta_\gamma^2(x_i)\right)\left[\dfrac{1}{|x_i|}-\dfrac{x_i\cdot x_i}{|x_i|^3}\right]+\gamma^2\theta_\gamma(x_i)\dfrac{x_i \cdot x_i}{|x_i|^2}&\hbox{if }x_i\in\mathcal{I}\\
\end{array}
\right.
 \end{equation}}}
and $\theta_\gamma(x_i)=1-\gamma|x_i|+\dfrac{1}{2\gamma}$.

Therefore, the regularized objective function takes the following form
\begin{align}
\nonumber J_\gamma(\mathbf{y},u,w)&{}=\dfrac{1}{2}(SH\mathbf{y}-\mathbf{z})^T R^{-1}(SH\mathbf{y}-\mathbf{z}) + \dfrac{1}{2}(u-u^b)^T B^{-1}(u-u^b)\\
\label{eq:cost_reg_tgv}&\displaystyle \hspace{0.5cm}+\dfrac{\mu}{2} w^Tw+\alpha\sum_{i=1}^{n-1}H_\gamma(D_iu-w_i)+\beta \sum_{i=1}^{n-1}H_\gamma(E_i w)
\end{align}
Furthermore, the TGV--regularized problem can be written in the following way
\begin{equation}\label{eq:reg_discrete_problem_tgv}
\begin{array}{ll}
\displaystyle\min_{(\mathbf{y},u,w)}&J_\gamma(\mathbf{y},u,w)\\
\hbox{subject to:}
&\mathbb{E}\mathbf{y}+\mathbb{Z}(\mathbf{y})\mathbb{U}\mathbf{y}-\mathbf{f}=0.
\end{array}
\end{equation}

\subsection{Consistency of the regularization}
In this section, we analyze the convergence of the regularized solutions towards the solution of the original data assimilation problem. Let us start by noticing that, for $a\in \mathbb{R}$, the following holds:
\[
\lim_{\gamma\to\infty} H_\gamma(a)=|a|.
\]

In order to prove that the solutions of the regularized problem converge to the original, we need to show the uniform boundedness of the solutions to the regularized data assimilation problems with respect to the parameter $\gamma$ from the Huber regularization. %Since the objective function depends only on a few components of $\mathbf{y}$, we need to use the state equation to guarantee its boundedness.

The next theorem shows the convergence of the solutions to the regularized problems to the solution of the original one.
\begin{theorem}

Let $(\mathbf{y}_\gamma,u_\gamma,w_\gamma)$ be a sequence of solutions to problem~\eqref{eq:reg_discrete_problem_tgv} where $\mathbf{y}_\gamma$ satisfies $e(\mathbf{y}_\gamma,u_\gamma)=0$.
Furthermore, let $(\mathbf{y},u,w)$ be a global solution of problem~\eqref{eq:discrete_problem_tgv}. Then, there exists a subsequence $(\mathbf{y}_{\gamma_k},u_{\gamma_k},w_{\gamma_k})$ such that
\[
(\mathbf{y}_{\gamma_k},u_{\gamma_k},w_{\gamma_k})\rightarrow (\mathbf{y},u,w), \text{ for }k\to \infty.
\]
\end{theorem}
\begin{proof}
We start by recalling that the matrix $B$ is a covariance symmetric and positive definite matrix. Using spectral decomposition we have
\[
(u_\gamma-u^b)^T B^{-1}(u_\gamma-u^b)\geq \lambda_{min}(B^{-1})\parallel
 u_\gamma- u^b\parallel^2,
\]
where $\lambda_{min}(B^{-1})$ is the minimum eigenvalue of the matrix $B^{-1}$. From the definition of the objective function it follows that
\[
J_\gamma (\mathbf{y}_\gamma,u_\gamma,w_\gamma)\geq (u_\gamma-u^b)^T B^{-1}(u_\gamma-u^b)\geq \lambda_{min}(B^{-1})\parallel
 u_\gamma- u^b\parallel^2.
\]
Using the optimality of $(\mathbf{y}_\gamma,u_\gamma,w_\gamma)$ and the fact that $H_\gamma(a)\leq |a|$, for all $a\in \mathbb{R}$,
\[
J_\gamma(\mathbf{y}_\gamma,u_\gamma,w_\gamma)\leq J_\gamma(\mathbf{y},u,w)\leq J(\mathbf{y},u,w)
\]
Thus, we can conclude that
\[
\parallel u_\gamma-u^b\parallel^2\leq \dfrac{J(\mathbf{y},u,w)}{\lambda_{min}(B^{-1})}.
\]
and, therefore, the sequence $\{u_\gamma\}$ is bounded. By the Bolzano--Weierstrass Theorem there exists a convergent subsequence $\{u_{\gamma_k}\}$ whose limit is denoted by $\bar{u}$.  Using similar arguments we can conclude that
\[
\parallel w_\gamma\parallel \leq \sqrt{\dfrac{2 J(\mathbf{y},u,w)}{\mu} }
\]
which is a uniform bound for $\{w_\gamma\}$. Using again the Bolzano-Weierstrass theorem, there exists a convergent subsequence denoted by $\{w_{\gamma_k}\}$ whose limit will be denoted by $\bar{w}$. Moreover, using the definition of the vector $\mathbf{y}_\gamma$ and by Lemma~\ref{lem:boundedness_y}, the sequence $\{\mathbf{y}_\gamma\}$ is also bounded and there exists a convergent subsequence $\{\mathbf{y}_{\gamma_k} \}$ with its limit denoted by $\bar{\mathbf{y}}$.
Thus, the triplet $(\bar{\mathbf{y}},\bar{u},\bar{w})$ is a candidate for the solution of problem~\eqref{eq:discrete_problem_tgv}.

Next we will prove the feasibility of $(\bar{\mathbf{y}},\bar{u},\bar{w})$, that is, $e(\bar{\mathbf{y}},\bar{u})=0$. Recalling the state equation given in~\eqref{eq:fulldis_eq} we know that
\[
e(\mathbf{y}_{\gamma_k},u_{\gamma_k})=\mathbb{E}\mathbf{y}_{\gamma_k}+\mathbb{Z}_{\gamma_k} \mathbb{U}\mathbf{y}_{\gamma_k}-f(u_{\gamma_k})=0,
\]
where $\mathbb{Z}_{\gamma_k}=\mathbb{Z}(\mathbf{y}_{\gamma_k})$. Taking the limit as $k \to \infty$ we have
\[
\mathbb{E}\bar{\mathbf{y}}+\lim_{\gamma\to\infty}\left(\mathbb{Z}_{\gamma_k} \mathbb{U}\mathbf{y}_{\gamma_k}\right)-f(\bar{u})=0.
\]
In order to guarantee that $\bar{\mathbf{y}}$ satisfies the state equation we have to prove that
\[
\lim_{k\to \infty}\mathbb{Z}_{\gamma_k} \mathbb{U}\mathbf{y}_{\gamma_k}=\bar{\mathbb{Z}}\mathbb{U}\bar{\mathbf{y}}
\]
where $\bar{\mathbb{Z}}=\mathbb{Z}(\bar{\mathbf{y}})$. This follows from the continuity of $\mathbb{Z}(\mathbf y)$ with respect to the argument.
% Since $\mathbf{y}_{\gamma_k}\to \bar{\mathbf{y}}$ as $k\to +\infty$ and due to the definition of the matrices $\mathbb{Z}$ and $\mathbb{U}$ we get
% \[
% \lim_{k\to \infty}\mathbb{Z}_{\gamma_k} \mathbb{U}\mathbf{y}_{\gamma_k}=\bar{\mathbb{Z}}\mathbb{U}\bar{\mathbf{y}}
% \]
Consequently,
\[
\mathbb{E}\bar{\mathbf{y}}+\bar{\mathbb{Z}} \mathbb{U}\bar{\mathbf{y}}-f(\bar{u})=0.
\]

Finally, since $H_\gamma(a)\rightarrow |a|$ when $\gamma\to \infty$, we obtain
 \[
 J(\bar{\mathbf{y}},\bar{u},\bar{w})=\lim_{k \to \infty} J_{\gamma} (\mathbf{y}_{\gamma_k},u_{\gamma_k},w_{\gamma_k})\leq J(\mathbf{y},u,w).
\]
Since $(\mathbf{y},u,w)$ is a global solution of problem~\eqref{eq:discrete_problem_tgv} we conclude that $(\bar{\mathbf{y}},\bar{u},\bar{w})$ is also a global solution \qedhere.
\end{proof}

In a similar manner as in the proof of Theorem \ref{teo:opt_sys_tgv} we then obtain the following optimality system characterizing the optimal solutions of \eqref{eq:reg_discrete_problem_tgv}:
\begin{subequations} \label{eq:reg_discrete_optimality_system}
\begin{align}
\mathbb{E}\mathbf{y}_\gamma+\mathbb{U}\mathbb{Z}\mathbf{y}_\gamma-\mathbf{f}_\gamma=0 \\
\mathbb{E}^T\mathbf{p}_\gamma+\mathbb{Z}\mathbb{U}^T \mathbf{p}_\gamma + \hbox{diag}(\mathbb{U}\mathbf{y}_\gamma)\mathbf{p}_\gamma-H^TS^TR^{-1}(HS\mathbf{y}_\gamma-\mathbf{z})=0\\
B^{-1}(u_\gamma-u^b)+ \alpha D^T h_\gamma(Du_\gamma-w_\gamma)+\dfrac{1}{\Delta t}p_\gamma^1=0\\
\mu w_\gamma -  \alpha h_\gamma(Du_\gamma-w_\gamma)+\beta E^Th_\gamma(Ew_\gamma)=0
\end{align}
\end{subequations}
%{\color{red} no deberian las variables tener subindice $\gamma$?}

where $h_\gamma(\cdot)$ is given in~\eqref{eq:der_huber_c2}. Differently from system \eqref{eq:discrete_optimality_system}, the regularized optimality condition \eqref{eq:reg_discrete_optimality_system} does not involve variational inequalities. This fact facilitates the solution of the resulting nonlinear system, although enough care has still to be taken in order to get a robust and fast convergent method for its solution.

%%%%%%%%%%%%%%%%%%%%%%%%%%%%%%%%%%%%%%%%%%%%%%%
%%%%%%%%%%%%%%%%%%%%%%%%%%%%%%%%%%%%%%%%%%%%%%%

\section{Reduced primal--dual Newton type method}
The numerical solution of the data assimilation problem can in principle be carried out by solving the optimality system or, alternatively, by using iterative methods for the solution of the optimization problem. In this paper, we focus our attention on the use of an iterative Newton--type method, with the additional use of auxiliary dual multipliers for improving the robustness of the approach. The proposed algorithm has the following ingredients:
\begin{enumerate}
\item Use of Newton's method for optimization problems;
\item Use of primal and dual variables for the solution of a saddle point system in order to get search directions;
\item Project the dual variables and get a modified second order matrices to obtain descent directions;
\item Use of polynomial line--search rule for the step length.
\end{enumerate}

We start by recalling Newton's method for optimization problems, which is given through the following steps:
\begin{algorithm}[H]
\caption{Newton's method}
\label{alg:newton}
\begin{algorithmic}[1]
\State Initialize $k=0$, $u_0$
\Repeat
\State Compute $\mathbf{y}^k$ the solution of the state equation given in~\eqref{eq:fulldis_eq}.
\State Compute $mathbf{p}^k$ the solution of the adjoint equation given in~\eqref{eq:adjoint_state}.
\State Solve the system
{\small{\[
\left(\begin{array}{cc}
\mathcal{L}''_{(\mathbf{y},u,w)}(\mathbf{y}^k,u^k,w^k,p^k)& e'(\mathbf{y}^k,u^k,w^k)^T\\
e' (\mathbf{y}^k,u^k,w^k)& 0
\end{array}\right)\left(\begin{array}{c}
\left(\begin{array}{c}
\delta_y\\ \delta_u\\ \delta_w
\end{array}\right)\\ \delta_\pi
\end{array}\right)=\left(\begin{array}{c}
0\\ e_u(\mathbf{y},u,w)^T p - J_u(\mathbf{y}^k,u^k,w^k)\\e_w(\mathbf{y},u,w)^T p - J_w(\mathbf{y}^k,u^k,w^k)\\0
\end{array}\right)
\]}}
\State Set $u^{k+1}=u^k+\delta_u$, $w^{k+1}=u^k+\delta_w$.
\State $k\gets k+1$.
\Until{\hbox{Stopping criteria}}
\end{algorithmic}
\end{algorithm}
\noindent where $e'(\mathbf{y},u,w)(v_1,v_2,v_3) =e_\mathbf{y}(\mathbf{y},u,w)v_1 +e_u(\mathbf{y},u,w)v_2 +e_w(\mathbf{y},u,w)v_3$ and the Lagrangian is given by
% In the remainder of the section we present how to apply this algorithm to our problem, and the main aspects of this application, like the derivatives of the state equation, the Lagrangian and the different modifications we perform in order to increase the stability and get descent directions.
% The first step is to compute the Lagrangian of the problem and its derivatives. The Lagrangian is given by
\begin{align*}
\mathcal{L}(\mathbf{y},u,w,\mathbf{p})&{}=\dfrac{1}{2}(SH\mathbf{y}-\mathbf{z})^T R^{-1}(SH\mathbf{y}-\mathbf{z})+\dfrac{1}{2}(u-u^b)^TB^{-1}(u-u^b)+\dfrac{\mu}{2}w^Tw\\
&{}+\alpha \sum_{i=1}^{n-1} H_\gamma(D_iu-w_i)+ \beta \sum_{i=1}^{n-1} H_\gamma(E_iw) -\mathbf{p}^T(\mathbb{E}\mathbf{y}+\mathbb{Z}(\mathbf{y})\mathbb{U}\mathbf{y}-\mathbf{f}(u)),
\end{align*}
whose first derivative is
\[
\nabla_{(\mathbf{y},u,w)}\mathcal{L}(\mathbf{y},u,w,\mathbf{p})=\left[\begin{array}{c}
H^TS^TR^{-1}(HS\mathbf{y}-\mathbf{z})-\mathbb{E}^T\mathbf{p}-\mathbb{Z}\mathbb{U}^T\mathbf{p}-\hbox{diag}(\mathbb{U}\mathbf{y})\mathbf{p}\\
(\nicefrac{1}{\Delta t})p^1+B^{-1}(u-u^b)+\alpha D^Th_\gamma(Du-w)\\
\mu w -\alpha h_\gamma(Du-w)+\beta h_\gamma(Ew)
\end{array}\right].
\]
Exploiting the structure of the first derivative of the Lagrangian we include dual variables in order to transform the system into a saddle point problem. This technique is known to increase the stability of the system (see \cite{chan1999nonlinear}). Therefore, we add the variables $q_1$ and $q_2$ and get the following system
\[
\nabla_{(\mathbf{y},u,w)}\mathcal{L}(\mathbf{y},u,w,\mathbf{p})=\left[\begin{array}{c}
H^TS^TR^{-1}(HS\mathbf{y}-\mathbf{z})-\mathbb{E}^T\mathbf{p}-\mathbb{Z}\mathbb{U}^T\mathbf{p}-\hbox{diag}(\mathbb{U}\mathbf{y})\mathbf{p}\\
(\nicefrac{1}{\Delta t})p^1+B^{-1}(u-u^b)+\alpha D^Tq_1\\
\mu w-\alpha q_1+\beta q_2\\
q_1-h_\gamma(Du-w)\\
q_2-h_\gamma(Ew)
\end{array}\right]=0.
\]
The second derivative of the Lagrangian is
\[
\nabla^2_{(\mathbf{y},u,w,q_1,q_2)}\mathcal{L}(\mathbf{y},u,w,q_1,q_2,\mathbf{p})=\left[\begin{array}{ccccc}
\Psi&0&0&0&0\\
0&B^{-1}&0&\alpha D^T&0\\
0&0&\mu\mathbb{I}&-\alpha\mathbb{I}&\beta E^T\\
0&-h'_\gamma(Du-w)D&h'_\gamma(Du-w)&\mathbb{I}&0\\
0&0&-h'_\gamma(Ew)E&0&\mathbb{I}
\end{array}\right],
\]
where $h'_\gamma(\cdot)$ is the second derivative of the Huber's regularization given in~\eqref{eq:2der_huber_c2} and the matrix
\begin{equation}\label{eq:PSI}
\Psi=H^TS^TR^{-1}HS-\mathbb{K}
\end{equation}
with
{\scriptsize{\begin{equation}\label{eq:UP}
\mathbb{K}=\left[\begin{array}{ccccc}
0&0&0&\cdot&0\\
0&\hbox{diag}(U^Tp^2)+\hbox{diag}(U p^2)&0&\cdot &0\\
0&0&\hbox{diag}(U^Tp^3)+\hbox{diag}(U p^3)&\cdot&0\\
\vdots&\vdots&\vdots&\ddots&\vdots\\
0&0&0&\cdots& \hbox{diag}(U^Tp^{N_t})+\hbox{diag}(U p^{N_t})
\end{array}
\right].
\end{equation}}}
By analyzing the elements of the diagonal of the second derivative given by \eqref{eq:2der_huber_c2} it can be verified that there are a lot of zero entries, which causes ill-conditioning of the matrix, with the corresponding algorithmic robustness issues. Moreover, the positive definitness of the matrix does not necessarily hold. For these reasons, we modify the matrix in such a way that better conditioning properties and postive definiteness are obtained. To do so, we consider the primal dual formulation of the problem and perform a projection of the dual variables (see ~\cite{HintermuellerStadler2007}).

Specifically, this process consists in replacing the term
\[
\dfrac{(Du-w)\odot (Du-w)}{|Du-w|^3},
\]
which appears in the second derivative of the Huber's regularization, by the following term
\[
\dfrac{q_1}{\max\{1,|q_1|\}}\odot \dfrac{Du-w}{|Du-w|^2}.
\]
where $(u\odot w)$ is the component-wise product of the vectors $u$ and $w$. Thus, $h'_\gamma(Du-w)$ is replaced by the diagonal matrix $Q_1$ given by
{\small{\begin{equation}\label{eq:Q1}
\hbox{diag}(Q_1)=\left\{\begin{array}{ll}
\dfrac{1}{|\zeta|}-\dfrac{q_1}{\max\{1,|q_1|\}}\odot\dfrac{\zeta}{|\zeta|^2}&\hbox{if }\zeta_i\in\mathcal{A}\\
\gamma &\hbox{if }\zeta_i\in\mathcal{B}\\
\left(1-\dfrac{\gamma}{2}\theta_\gamma^2\right)\left[\dfrac{1}{|\zeta|}-\dfrac{q_1}{\max\{1,|q_1|\}}\odot\dfrac{\zeta}{|\zeta|^2}\right]+\gamma^2\theta_\gamma\dfrac{\zeta}{|\zeta|}\odot\dfrac{\zeta}{|\zeta|}&\hbox{if }\zeta_i\in\mathcal{I}\\
\end{array}
\right.
 \end{equation}}}
where $\zeta_i=D_iu-w_i$. Moreover, $h'_\gamma(Ew)$ is replaced by the diagonal matrix $Q_2$ given by
 {\small{\begin{equation}\label{eq:Q2}
\hbox{diag}(Q_2)=\left\{\begin{array}{ll}
\dfrac{1}{|Ew|}-\dfrac{q_2}{\max\{1,|q_2|\}}\odot\dfrac{Ew}{|Ew|^2}&\hbox{if }E_iw\in\mathcal{A}\\
\gamma &\hbox{if }E_iw\in\mathcal{B}\\
\left(1-\dfrac{\gamma}{2}\theta_\gamma^2\right)\left[\dfrac{1}{|Ew|}-\dfrac{q_2}{\max\{1,|q_2|\}}\odot\dfrac{Ew}{|Ew|^2}\right]+\gamma^2\theta_\gamma\dfrac{Ew}{|Ew|}\odot\dfrac{Ew}{|Ew|}&\hbox{if }E_iw\in\mathcal{I}\\
\end{array}
\right.
 \end{equation}}}
where the division, $\max$ and $|\cdot|$ are component-wise operations. It is worth to remark that in the last term of the third case of the $Q_1$ and $Q_2$ matrices we do not perform the projection.

In this way, the modified second derivative of the Lagrangian used in the method is
\[
\nabla^2_{(\mathbf{y},u,w,q_1,q_2)}\mathcal{L}(\mathbf{y},u,w,q_1,q_2,\mathbf{p})=\left[\begin{array}{ccccc}
\Psi&0&0&0&0\\
0&B^{-1}&0&\alpha D^T&0\\
0&0&\mu \mathbb{I}&-\alpha\mathbb{I}&\beta E^T\\
0&-Q_1D&Q_1&\mathbb{I}&0\\
0&0&-Q_2E&0&\mathbb{I}
\end{array}\right].
\]
On the other hand, the derivatives of the state equation are given by
\begin{equation}\label{eq:Xi}
e_y(\mathbf{y},u)
= \mathbb{E}+\mathbb{Z}\mathbb{U}+\hbox{diag}(\mathbb{U}\mathbf{y})=: \Xi
\end{equation}
and
\begin{equation}\label{eq:Upsilon}
e_u(\mathbf{y},u)=
-(\nicefrac{1}{\Delta t})\left[\begin{array}{c}
\mathbb{I}\\0\\ \vdots\\0
\end{array}
\right]=:\Upsilon
\end{equation}
Furthermore, since the state equation does not depend on $\mathbf{p},q_1$ nor $q_2$, its derivatives with respect to these variables are equal to zero.

Now, the system that defines the method is given by:
\begin{equation}\label{eq:Htnewton_tgv}
\left[\begin{array}{cccccc}
\Psi&0&0&0&0&\Xi^T\\0&B^{-1}&0&\alpha D^T&0& \Upsilon^T\\0&0&\mu\mathbb{I}&-\alpha\mathbb{I}&\beta E^T&0\\0 &-Q_1D&Q_1&\mathbb{I}&0&0\\ 0&0&-Q_2E&0&\mathbb{I}&0\\ \Xi&\Upsilon&0&0&0&0
\end{array}
\right]\left[\begin{array}{c}
\delta_y\\ \delta_u\\ \delta_w \\ \delta_{q_1}\\ \delta_{q_2}\\ \delta_{\pi}
\end{array}
\right]=\left[\begin{array}{c}
0\\ -\nicefrac{p^1}{\Delta t}-B^{-1}(u-u^b)-\alpha D^Tq_1\\-\mu w+ \alpha q_1-\beta E^Tq_2\\-q_1+h_\gamma(Du-w)\\ -q_2+h_\gamma(Ew)\\0
\end{array}
\right]
\end{equation}
% In the later section we will prove that this method generates indeed descent directions.

% The method is given through the following steps:
%
% \begin{algorithm}
% \caption{Globalized primal--dual Newton method}
% \label{alg:gnewton_tv}
% \begin{algorithmic}[1]
% \State Initialize $k=0$, $u_0$, $B_0$
% \Repeat
% \State Compute $y^k$ the solution of the state equation given in~\eqref{eq:fulldis_eq}.
% \State Compute $p^k$ the solution of the adjoint equation given in~\eqref{eq:adjoint_state}.
% \State Solve the system given in~\eqref{eq:Htnewton_tgv}.
% \State Compute the step length $s$ with a polynomial line--search procedure.
% \State Update $u^{k+1}=u^k + s\delta_u$, $w^{k+1}=w^k + s\delta_w$, ${q_1}^{k+1}=q_1^k + s\delta_{q_1}$ and $q_2^{k+1}=q_2^k + s\delta_{q_2}$
% \State $k\gets k+1$.
% \Until{\hbox{A stopping criteria is satisfied}}
% \end{algorithmic}
% \end{algorithm}

\subsection{Reduced method and matrix properties}
Before presenting the convergence result of the proposed algorithm we start by exhibiting and analyzing the reduced version of it.
% We recall the matrix presented in equation~\eqref{eq:Htnewton_tgv} and given by the following expression:
% \[
% \left[\begin{array}{cccccc}
% \Psi&0&0&0&0&\Xi^T\\0&B^{-1}&0&\alpha D^T&0& \Upsilon^T\\0&0&0&-\alpha\mathbb{I}&\beta E^T&0\\0 &-Q_1D&Q_1&\mathbb{I}&0&0\\ 0&0&-Q_2E&0&\mathbb{I}&0\\ \Xi&\Upsilon&0&0&0&0
% \end{array}
% \right]
% \]
% where $\Xi$ is given in~\eqref{eq:Xi} and $\Upsilon$ in~\eqref{eq:Upsilon}.
From Lemma~\ref{lm:e_y_invertible} we know that the matrix $\Xi$ in \eqref{eq:Xi} is invertible and, therefore, the following equalities hold:
\begin{align*}
\delta_{q_1}&{}=Q_1D\delta_u-Q_1\delta_w -q_1+h_\gamma(Du-w),\\
\delta_{q_2}&{}=Q_2E\delta_w - q_2 + h_\gamma(Du-w),\\
\delta_y&{}=-(\Xi)^{-1}\Upsilon\delta_u,\\
\delta_{\pi}&{}=(\Xi)^{-T}\Psi(\Xi)^{-1}\Upsilon \delta_u.
\end{align*}
System \eqref{eq:Htnewton_tgv} can then be rewritten as follows:
\begin{align*}
(B^{-1}+\alpha D^TQ_1D+\Upsilon^T(\Xi)^{-T}&{}(\Xi)^{-1}\Upsilon)\delta_u-\alpha D^TQ_1\delta_w\\
&{}=-\nicefrac{1}{\Delta t}p^1-B^{-1}(u-u^b)-\alpha D^T h_\gamma(Du-w),\\
-\alpha Q_1D\delta_u + (\mu \mathbb{I}+\alpha Q_1+\beta E^TQ_2E)\delta_w&{}=-\mu w+\alpha h_\gamma(Du-w)-\beta E^Th_\gamma(Ew),
\end{align*}
or, equivalently
{{\begin{align}
\nonumber \left[
\begin{array}{cc}
B^{-1}+\alpha D^TQ_1D + \Upsilon^T(\Xi)^{-T}\Psi(\Xi)^{-1}\Upsilon& -\alpha D^TQ_1\\
-\alpha Q_1D&\mu \mathbb{I}+\alpha Q_1+ \beta E^TQ_2E
\end{array}\right]\left[\begin{array}{c}
\delta_u\\ \delta_w
\end{array}\right] \hspace{2cm} &{}\\
\label{eq:Pi_system}=\left[\begin{array}{c}
-\nicefrac{p^1}{\Delta t}-B^{-1}(u-u^b)-\alpha D^Th_\gamma(Du-w)\\
\mu w+ \alpha h_\gamma(Du-w)-\beta E^Th_\gamma(Ew)
\end{array}\right]&{}
\end{align}}}
Now, defining the matrix
\begin{equation}\label{eq:Pimatrix}
\Pi=\left[
\begin{array}{cc}
B^{-1}+\alpha D^TQ_1D + \Upsilon^T(\Xi)^{-T}\Psi(\Xi)^{-1}\Upsilon& -\alpha D^TQ_1\\
-\alpha Q_1D&\mu \mathbb{I}+\alpha Q_1+ \beta E^TQ_2E
\end{array}\right],
\end{equation}
and the vector
\[
\phi(u,w)=-\left[\begin{array}{c}
-\nicefrac{p^1}{\Delta t}-B^{-1}(u-u^b)-\alpha D^Th_\gamma(Du-w)\\
-\mu w+ \alpha h_\gamma(Du-w)-\beta E^Th_\gamma(Ew)
\end{array}\right].
\]
%the reduced system that determines the direction is
%\begin{equation}\label{eq:pisystem}
%\Pi[\delta_u,\delta_w]^T=\phi(u,w).
%\end{equation}
Henceforth, we focus on the proof of the main features of the matrix $\Pi$ that can guarantee that $[\delta_u,\delta_w]^T$ is indeed a descent direction. We start proving its symmetry.
\begin{proposition}
The matrix $\Pi$ given in~\eqref{eq:Pimatrix} is symmetric.
\end{proposition}
\begin{proof}
Since $\Pi$ is a block--matrix we know that
\[
\Pi^T=\left[\begin{array}{cc}
\Pi_{11}^T&\Pi_{21}^T\\
\Pi_{12}^T&\Pi_{22}^T
\end{array}\right]
\]
We analyze each of the components starting with $\Pi_{11}$, which is given by
\begin{align*}
\Pi_{11}^T&{}=\left(B^{-1}+\beta D^TQ_1D+\Upsilon^T(\Xi)^{-T}\Psi(\Xi)^{-1}\Upsilon\right)^T\\
&{}=(B^T)^{-1}+\beta D^TQ_1^TD+\Upsilon^T(\Xi)^{-T}\Psi^T(\Xi)^{-1}\Upsilon
\end{align*}
Since $B$ is a covariance matrix, it is symmetric and its inverse as well. It remains to prove that the matrices $Q_1$ and $\Psi$ are also symmetric. For the matrix $Q_1$ we know by definition that it is a diagonal matrix. Finally we are going to analyze the matrix $\Psi$, given by
\begin{align*}
\Psi^T = \left(H^TS^TR^{-1}SH-\mathbb{K}\right)^T = H^TS^TR^{-T}-\mathbb{K}^T.
\end{align*}
We notice that $R$ is also a covariance matrix and, therefore symmetric. Finally, $\mathbb{K}$ is a diagonal matrix (see eq. \eqref{eq:Q1}), $\Psi$ is also symmetric and therefore $\Pi_{11}$ is symmetric.

We now focus on
\[
\Pi_{22}^T=\mu \mathbb{I}+\alpha Q_1^T+ \beta E^TQ_2^TE.
\]
Since $Q_1$ and $Q_2$ are diagonal matrices, it immediately follows that $\Pi_{22}$ is symmetric.

By definition we have
\[
\Pi^T=\left[\begin{array}{cc}
\Pi_{11}^T&(-\alpha Q_1D)^T\\
(-\alpha D^TQ_1)^T&\Pi_{22}^T
\end{array}\right]=\left[\begin{array}{cc}
\Pi_{11}&(-\alpha D^TQ_1)\\
(-\alpha Q_1D)&\Pi_{22}
\end{array}\right]=\Pi
\]
and the proof is complete. \qedhere.
\end{proof}
Now we are going to prove that $\Pi$ is positive definite, and therefore that the directions given in~\eqref{eq:Pi_system} are indeed descent directions for the data assimilation problem. Before stating the main result, we prove some required lemmata.

% The following Lemma is an intermediate result needed in order to prove the positive definiteness of the matrix $\Pi$.
\begin{lemma}\label{lm:lyy_bound}
If the term $\parallel H^TS^TR^{-1}(SH\mathbf{y}-\mathbf{z})\parallel_\infty$ is small enough, then there exists a constant $\kappa>0$ such that for all $\xi\in\mathbb{R}^m\backslash\{0\}$ we have
\[
\xi^T \mathcal{L}''_{(\mathbf{y,y})}(\mathbf{y},u,w,\mathbf{p})\xi\geq \kappa \parallel \xi \parallel^2.
\]
\end{lemma}
\begin{proof}
The Lagrangian of problem~\eqref{eq:reg_discrete_problem_tgv} given by the following expression
\begin{align*}
\mathcal{L}(\mathbf{y},u,w,\mathbf{p})&{}=\dfrac{1}{2}(SH\mathbf{y}-\mathbf{z})^TR^{-1}(SH\mathbf{y}-\mathbf{z})+\dfrac{1}{2}(u-u^b)^TB^{-1}(u-u^b)+\dfrac{\mu}{2} w^Tw\\
&{}+\alpha \sum_{i=1}^{n-1}H_\gamma(D_iu-w_i)+\beta\sum_{i=1}^{n-1}H_\gamma(E_iw) -\mathbf{p}^T(\mathbb{E}\mathbf{y}+\mathbb{Z}(\mathbf{y})\mathbb{U}\mathbf{y}-\mathbf{f}(u)).
\end{align*}
The first derivative of the Lagrangian with respect of $\mathbf{y}$ is given by
\[
\mathcal{L}'_{\mathbf{y}}(\mathbf{y},u,w,\mathbf{p})=H^TS^TR^{-1}(SH\mathbf{y}-\mathbf{z})-\mathbb{E}^T\mathbf{p}-\mathbb{Z}\mathbb{U}^T\mathbf{p}-\hbox{diag}(\mathbb{U}\mathbf{y})\mathbf{p}.
\]
Furthermore, the second derivative of the Lagrangian with respect to $\mathbf{y}$ is as follows
\[
\mathcal{L}''_{(\mathbf{y,y})}(\mathbf{y},u,w,\mathbf{p})=H^TS^TR^{-1}SH-\mathbb{K}
\]
where $\mathbb{K}$ is given in~\eqref{eq:UP}. Using the definition of the matrix $\mathbb{K}$ and multiplying both sides by $\xi\in\mathbb{R}^m\setminus \{0\}$ we have:
\begin{align*}
\xi^T\mathcal{L}''_{(\mathbf{y,y})}\xi &{}=\xi^TH^TS^TR^{-1}SH\xi-\xi^T\mathbb{K}\xi \geq c\parallel \xi\parallel^2-2\displaystyle\sum_{i=1}^{N_t}\xi_i^T(\mathbb{U}_ip^i) \xi_i\\
&{}=c\parallel \xi\parallel^2 - 2\displaystyle\sum_{i=1}^{N_t}\sum_{j=1}^n(U_{ij})p^i_j(\xi^i_j)^2 \geq c\parallel \xi\parallel^2-2\parallel \mathbb{U}\mathbf{p}\parallel_\infty \sum_{i=1}^{N_t}\sum_{j=1}^n(\xi_j^i)^2\\
&{}\geq c \parallel \xi \parallel^2 - 2 \parallel \mathbb{U}\parallel_\infty \parallel \mathbf{p}\parallel_\infty\parallel \xi\parallel^2\\
&{}\geq \left(c-2\rho \parallel \mathbb{U}\parallel_\infty \parallel H^TS^TR^{-1}(SH\mathbf{y}-\mathbf{z})\parallel_\infty\right) \parallel \xi\parallel ^2
\end{align*}
where $c$ is the smallest eigenvalue of the matrix $H^TS^TR^{-1}SH$. The last inequality was obtained by applying Lemma~\ref{lm:p_inf}. Then, we define
\[
\kappa=c-2\rho \parallel \mathbb{U}\parallel_\infty \parallel H^TS^TR^{-1}(SH\mathbf{y}-\mathbf{z})\parallel_\infty.
\]
By setting
\[
\parallel H^TS^TR^{-1}(SH\mathbf{y}-\mathbf{z})\parallel_\infty\leq \dfrac{c}{2\rho}
\]
we guarantee that $\kappa>0$ and thus we get the desired result.\qedhere
\end{proof}

%Now we can prove that the matrix $\Pi$ is in fact positive definite.
\begin{proposition} \label{prop: Pi is positive}
The matrix $\Pi$ defined in~\eqref{eq:Pimatrix} is positive definite.
\end{proposition}
\begin{proof}
Multiplying on both sides of $\Pi$ by a block vector $(x,y)$ we obtain that
\begin{align*}
  (x,y) \Pi \begin{pmatrix}
  x\\
  y
  \end{pmatrix}  &=(x,y) \left(
  \begin{array}{cc}
  B^{-1}+\alpha D^TQ_1D + \Upsilon^T(\Xi)^{-T}\Psi(\Xi)^{-1}\Upsilon& -\alpha D^TQ_1\\
  -\alpha Q_1D&\mu \mathbb{I}+\alpha Q_1+ \beta E^TQ_2E
  \end{array}\right) \begin{pmatrix}
  x\\
  y
\end{pmatrix}\\
  & \geq x^T B^{-1} x + \alpha x^T D^T Q_1 D x- 2 \alpha x^T D^T Q_1 y+ \mu y^Ty + \alpha y^T Q_1 y + \beta y^T E^T Q_2 E y,
\end{align*}
since the matrix $\Upsilon^T(\Xi)^{-T}\Psi(\Xi)^{-1}\Upsilon$ is positive semi-definite thanks to Lemma ~\ref{lm:lyy_bound}. To proof the statement we have to verify that both $Q_1$ and $Q_2$ are positive semi-definite.

The matrix $Q_1$ is diagonal, thus we just need to prove that each element of its diagonal is positive. We start noticing that $\dfrac{|(q_1)_i|}{\max\{|(q_1)_i|,1\}}\leq 1$ and, therefore,
\begin{equation}\label{eq:boundq1}
\dfrac{(q_1)_i}{\max\{|(q_1)_i|,1\}}\odot \dfrac{D_iu-w_i}{|D_iu-w_i|^2}\leq \dfrac{D_iu-w_i}{|D_iu-w_i|^2}.
\end{equation}
Now using definition~\eqref{eq:Q1} we analyze the following three cases.

\begin{description}
\item[$D_iu-w_i \in\mathcal{A}$]:
\[
(Q_1)_{ii}=\dfrac{1}{|D_iu-w_i|}-\dfrac{(q_1)_i}{\max\{|(q_1)_i|,1\}}\odot\dfrac{D_iu-w_i}{|D_iu-w_i|^2}.
\]
Using the bound given in~\eqref{eq:boundq1} we have
\begin{align*}
(Q_1)_{ii}&{}=\dfrac{1}{|D_iu-w_i|}-\dfrac{(q_1)_i}{\max\{|(q_1)_i|,1\}}\odot \dfrac{D_iu-w_i}{|D_iu-w_i|^2},
&{}\geq \dfrac{1}{|D_iu-w_i|}\left(1-\dfrac{D_iu-w_i}{|D_iu-w_i|}\right) \geq 0.
\end{align*}
\item[$D_iu-w_i \in\mathcal{B}$]:
From the definition we have
\[
(Q_1)_{ii}=\gamma>0
\]
\item[$D_iu-w_i \in\mathcal{I}$]:
We start by recalling the structure of the set $\mathcal{I}:=\{i\colon |\gamma |D_iu-w_i|-1|\leq \nicefrac{1}{2\gamma}\}$. The elements of the diagonal of $Q_1$ are given by
\begin{align*}
(Q_1)_{ii}&{}=\left(1-\dfrac{\gamma}{2}\theta_\gamma^2\right)\left[\dfrac{1}{|D_iu-w_i|}-\dfrac{q_i}{\max\{1,\gamma|(q_1)_i|\}}\odot\dfrac{D_iu-w_i}{|D_iu-w_i|^2}\right]\\
&{}+\gamma^2\theta_\gamma\dfrac{D_iu-w_i}{|D_iu-w_i|}\odot\dfrac{D_iu-w_i}{|D_iu-w_i|},
\end{align*}
with $\theta_\gamma=(1-\gamma|D_iu-w_i|+\nicefrac{1}{2\gamma})$. Using the same argument as in the first case we have
\[
\left[\dfrac{1}{|D_iu-w_i|}-\dfrac{(q_1)_i}{\max\{|(q_1)_i|,1\}}\odot \dfrac{D_iu-w_i}{|D_iu-w_i|^2}\right]\geq 0.
\]
By the definition of the set $\mathcal{I}$ it follows
that $\theta_\gamma \geq 0$ and $1-\dfrac{\gamma}{2}\theta_\gamma^2 \geq 1-\dfrac{1}{2\gamma}\geq 0.$ Therefore,
\[
(Q_1)_{ii}\geq \gamma^2 \theta_\gamma \left(\dfrac{D_iu-w_i}{|D_iu-w_i|}\odot \dfrac{D_iu-w_i}{|D_iu-w_i|}\right)=\gamma^2 \theta_\gamma\left(\dfrac{D_iu-w_i}{|D_iu-w_i|}\right)^2\geq 0.
\]
Since each element of the diagonal of $Q_1$ is non-negative we can conclude that the matrix is in fact a positive semi definite matrix. Furthermore, we justify why the last term on this expression was not projected, because otherwise we can not guarantee that this term is positive.
\end{description}

In a similar manner it can be proved that $Q_2$ is also positive semi-definite.
Altogether we obtain that
\begin{align*}
  (x,y) \Pi \begin{bmatrix}
  x\\
  y
\end{bmatrix}
  & \geq \lambda_{min}(B^{-1}) \|x\|^2 + \mu \| y\|^2+ \alpha \| Q_1^{1/2} D x- Q_1^{1/2} y\|^2.
\end{align*}

\end{proof}

\subsection{Convergence analysis}
This subsection is devoted to the study of the convergence properties of the proposed algorithm. Our aim is to use the general results developed for the convergence of descent methods (see \cite[Theorem 4.2]{de2015numerical}).
%\begin{theorem}\label{th:conv2}
%Let be $f$ continuously differentiable and bounded from below. Let $\{x_k\}$, $\{\alpha_k\}$ and $\{d_k\}$ be sequences generated by a descent method with:
%\begin{align}
% \label{eq:cond_ang}-\nabla f(x^k)^Td^k&{}\geq \eta \parallel \nabla f(x^k)\parallel \parallel d^k\parallel \\
%\label{eq:decrecimiento} f(x^k+\alpha_k d^k)&{}< f(x^k)\\
% \label{eq:cond_alf} f(x^k+\alpha_kd^k)-f(x^k)\xrightarrow{k\to+\infty}&{} 0\Longrightarrow \dfrac{\nabla f(x^k)^T d^k}{\parallel d^k\parallel}\xrightarrow{k\to+\infty} 0
%\end{align}
%holding. Then,
%\[
%\lim_{k\to+\infty} \nabla f(x^k)=0
% \]
% and every accumulation point of $\{x^k\}$ is a stationary point of $f$.
%\end{theorem}
%In order to apply this Theorem, we need to verify that equations~\eqref{eq:cond_ang},~\eqref{eq:decrecimiento} and~\eqref{eq:cond_alf} hold. First we introduce the definition of the level set $N^\rho_0$.
In order to apply this Theorem, we need to verify that the angle condition and the feasibility conditions of the line search hold. First we introduce the definition of the level set $N^\rho_0$.
\begin{definition}
Let $\rho>0$ be a constant, the level set of the problem if given by
\[
N^\rho_0:= \{(u,w)+(d_u,d_w)\colon J_\gamma(\mathbf{y}(u),u,w)\leq J_\gamma(\mathbf{y}(u_0),u_0,w_0), \parallel (d_u,d_2)\parallel \leq \rho \}
\]
for a given $u_0,w_0$.
\end{definition}

Due to the continuous differentiability of $\mathbf{y}(u)$ and the structure of the objective function it is easy to prove that $J_\gamma(\mathbf{y}(u),u)$ is radially unbounded and continuous, therefore the level set $N^\rho_0$ is a compact set. Now, we prove that the angle condition holds.
% that claims the existence of a constants $M,m>0$ independent of $k$ such that
% \begin{equation}\label{eq:cond_1}
% m\parallel q\parallel^2\leq q^T\Pi_kq\leq M\parallel q\parallel^2.
% \end{equation}
% for all $k$
%The second proposition shows that the inequality~\eqref{eq:cond_1} guarantee  the directions $\{d_k\}$ satisfy the angle condition given in~\eqref{eq:cond_ang}.

\begin{proposition}\label{prop:mM}
Let $\Pi_k=\Pi(u_k)$ be defined in~\eqref{eq:Pimatrix}, then there exist constants $0~<~m~<~M$ independent of $k$ such that
\[
m\parallel q\parallel^2\leq q^T\Pi_kq\leq M\parallel q\parallel^2, \qquad \text{for all } q\in \mathbb{R}^n, \text{ for all } u_k\in N_0^\rho.
\]
\end{proposition}
\begin{proof}

We start recalling the structure of the matrix $\Pi$.
\[
\Pi=\left[\begin{array}{cc}
\Pi_{11}&\Pi_{12}\\\Pi_{21}&\Pi_{22}
\end{array}\right]=
\left[\begin{array}{cc}
B^{-1}+\alpha DQ_1D+\Upsilon^T\Xi^{-T}\Psi\Xi^{-1}\Upsilon& -\alpha D^TQ_1\\
-\alpha Q_1 D&\mu \mathbb{I}+ \alpha Q_1 +\beta E^TQ_2E
\end{array}
\right]
\]
From Proposition \ref{prop: Pi is positive} it immediately follows that
\begin{equation*}
  q^T\Pi_kq \geq m \|q\|^2, \forall q\in \mathbb{R}^n, \text{ with }m=\min\{\lambda_{\hbox{min}}(B^{-1}),\mu\}.
\end{equation*}

Concerning the upper bound, we get that
\begin{align}
  (q_1,q_2) \Pi \begin{pmatrix}
  q_1\\
  q_2
  \end{pmatrix}
  & = q_1^T B^{-1} q_1 + \alpha q_1^T D^T Q_1 D q_1- 2 \alpha q_1^T D^T Q_1 q_2+ \alpha q_2^T Q_1 q_2 +\mu q_2^Tq_2\nonumber\\
  & \hspace{2cm}+ \beta q_2^T E^T Q_2 E q_2 +q_1^T \Upsilon^T\Xi^{-T}\Psi\Xi^{-1}\Upsilon q_1 \nonumber\\
  & \leq C_B \|q_1\|^2 + \alpha \|Q_1^{1/2} Dq_1 - Q_1^{1/2} q_2\|^2 +\mu \|q_2\|^2 \nonumber\\
  & \hspace{2cm} + \beta \gamma \|E\|^2 \|q_2\|^2+ \frac{1}{(\Delta t)^2} \|\Psi\| \|\Xi^{-1}\|^2 \|q_1\|^2. \label{eq: upper bound Pi 1}
\end{align}
For the last term in the previous inequality we first get a bound on $\Psi_k$, which is given by
\[
\parallel \Psi_k \parallel \leq \parallel H^TS^TR^{-1}SH\parallel +\parallel \mathbb{K}_k\parallel.
\]
Using the definition of $\mathbb{K}_k$ given in~\eqref{eq:UP} and Lemma~ \ref{lm:p_inf},
\begin{align*}
\parallel \mathbb{K}_k\parallel &{}= 2\parallel \mathbb{U}\parallel \parallel \mathbf{p}_k\parallel \le 2\rho \parallel \mathbb U \parallel \parallel H^TS^TR^{-1}\parallel  \parallel SH\mathbf{y}_k-\mathbf{z}\parallel \\
%Furthermore, the last inequality was obtained using Lemma~ \ref{lm:p_inf}. Thereby,
% \parallel \mathbb{K}_k\parallel &{}\leq 2\rho \parallel \mathbb{U}\parallel\parallel H^TS^TR^{-1} \left[\parallel SH\parallel \parallel \mathbf{y}_k\parallel +\parallel \mathbf{z}\parallel \right]\\
&{}\leq 2\rho C_0 \parallel \mathbb{U}\parallel \parallel H^TS^TR^{-1} \parallel,
\end{align*}
where the last inequality was obtained using Lemma~\ref{lem:boundedness_y} and the fact that $u_k\in N^\rho_0$. Defining
\[
\xi:= 2 \rho C_0 \parallel \mathbb{U}\parallel \parallel H^TS^TR^{-1} \parallel
\]
we have that
\[
\parallel \Psi_k\parallel \leq \parallel H^TS^TR^{-1}SH\parallel + \xi=: C_{\Psi}.
\]
For the norm of $\Xi^{-1}$ we consider an a-priori estimate of the adjoint equation. Specifically, let $\mathbf v$ be solution of
\begin{equation}
  \Xi \mathbf v= \mathbb{E} \mathbf v+ \mathbb{Z}_k\mathbb{U} \mathbf v+ \hbox{diag }(\mathbb{U}\mathbf{y}_k) \mathbf v=\mathbf g.
\end{equation}
Proceeding as in the proof of Lemma \ref{lm:p_inf} we get that
\begin{equation*}
  \|\mathbf v\| \leq \rho \|\mathbf g\|,
\end{equation*}
with $\rho$ independent of $\mathbf y$ and $u$, and therefore $\|\Xi^{-1}\| \leq \rho$.

Inserting the last results in \eqref{eq: upper bound Pi 1} and taking into account that $\|Q_i\| \leq \gamma$, for $i=1,2$, we then obtain
\begin{align*}
q^T\Pi q&{}\leq \left(C_B+\gamma C_D + \widetilde{C}\right)\parallel q_1\parallel^2 + [\mu+ \gamma(\alpha+\beta C_E)] \parallel q_2 \parallel^2\\
&{}\leq \max\left\{\left(C_B+\gamma C_D + \widetilde{C}\right), \mu + \gamma(\alpha+\beta C_E)\right\}\parallel q\parallel^2,
\end{align*}
where $\widetilde C:= \frac{\rho^2 C_{\Psi}}{(\Delta t)^2}$. Therefore, the constant that we were looking for is
\[
M:=\max\{\left(C_B+\gamma C_D + \widetilde{C}\right), \mu + \gamma(\alpha+\beta C_E)\}
\]
\end{proof}
It is worth to remark that, from Proposition~\ref{prop:mM} it follows immediately that the descent directions $\{d^k\}$ generated by Algorithm~\ref{alg:newton} satisfy the angle condition.

Before presenting the Lemma that guarantees the uniformly continuity of the reduced gradient, we present the following intermediate results that allow us to write in a shorter way the main result of this subsection.
\begin{lemma}\label{lem:diff_y_bound}
Let $u_1,u_2\in N_0^\rho$. Furthermore, let $\mathbf{y}(u_1)$ and $\mathbf{y}(u_2)$ be solutions of the state equations $e(\mathbf{y}(u_1),u_1)=0$ and $e(\mathbf{y}(u_2),u_2)=0$, respectively.
Then the following bound holds
\[
\parallel y^j(u_1)-y^j(u_2)\parallel \leq C \parallel u_1-u_2\parallel,\qquad j=1,\ldots,N_t,
\]
where $C>0$ is a constant independent of $u_1,u_2$ and $\mathbf{y}$.
\end{lemma}
\begin{proof}
From Proposition~\ref{prop:existence_burgers} we know that $\mathbf{y}(u)=[y^1(u),y^2(u),\ldots, y^{N_t}(u)]$ is a continuously differentiable function on $\mathbb{R}^m$, that means that $y^j(u)$ is continuously differentiable for all $j=1,\ldots,N_t$. Using the mean value theorem we can conclude that
\[
\parallel y^j(u_1)-y^j(u_2)\parallel \leq \parallel (y^j)'(\xi)\parallel \parallel u_1-u_2\parallel
\]
for some $\xi$ that belongs to a neighborhood which contains $u_1,u_2$. Since $(y^j)'$ is continuous and $u_1,u_2$ belongs to the level set $N_0^\rho$ and this is compact, there exists $C>0$ that satisfies the required bound.\qedhere
\end{proof}
\begin{lemma}\label{lem:diff_p_bound}
Let $u_1,u_2\in N_0^\rho$. Furthermore, let $\mathbf{y}(u_1)$ and $\mathbf{y}(u_2)$ be solutions of the state equations $e(\mathbf{y}(u_1),u_1)=0$ and $e(\mathbf{y}(u_2),u_2)=0$, respectively. In addition let $\mathbf{p}(u_1),\mathbf{p}(u_2)$ be its associate adjoint states.
Then the following bound holds
\[
\parallel p^j(u_1)-p^j(u_2)\parallel \leq \tilde{C} \parallel u_1-u_2\parallel,\qquad j=1,\ldots, N_t
\]
where $\tilde{C}>0$ is a constant independent of $u_1,u_2$ and $\mathbf{y}$ or $\mathbf{p}$.
\end{lemma}
\begin{proof}
From the analysis component--wise in time of two adjoint equations~\eqref{eq:adjoint_state} in the variables $u_1,u_2\in N^\rho_0$ and denoting by
\[
A=\left(\dfrac{1}{\Delta t}\mathbb{I}+\hbox{diag }(y^{j-1}(u_1))U^T+\hbox{diag }(Uy^j(u_1)) \right),
\]
we get the following inequality:
\begin{align*}
&{}\left[A(q^{N_t-j+1}(u_1)-q^{N_t-j+1}(u_2)),(q^{N_t-j+1}(u_1)-q^{N_t-j+1}(u_2))\right]\\
&{}\geq \dfrac{1}{\Delta t}\parallel q^{N_t-j+1}(u_1)-q^{N_t-j+1}(u_2)\parallel^2 \\
&{}\quad+ \left(\hbox{diag }(y^j(u_1))U^T (q^{N_t-j+1}(u_1)-q^{N_t-j+1}(u_2)),(q^{N_t-j+1}(u_1)-q^{N_t-j+1}(u_2))\right)\\
&{}\geq \dfrac{1}{\Delta t}\parallel q^{N_t-j+1}(u_1)-q^{N_t-j+1}(u_2)\parallel^2 - \parallel U\parallel \parallel \mathbf{y}(u_1)\parallel \parallel q^{N_t-j+1}(u_1)-q^{N_t-j+1}(u_2)\parallel^2\\
&{}\geq \left(\dfrac{1}{\Delta t}- \parallel U\parallel \left(\|u_1\| + \Delta t \|\mathbf{f}\|\right)\right) \parallel q^{N_t-j+1}(u_1)-q^{N_t-j+1}(u_2)\parallel^2\\
&{}\geq \dfrac{c}{\Delta t} \parallel q^{N_t-j+1}(u_1)-q^{N_t-j+1}(u_2)\parallel^2
\end{align*}
%{\color{red} no entiendo el ultimo paso. aclara porfa o sino modificale pasando al otro lado}

where the last inequalities were obtained using Lemma~\ref{lem:boundedness_y} and the fact that $u_1\in N_0^\rho$, that is, there exists $K>0$ such that $\|u_1\|\leq K$. Denoting $c=\nicefrac{1}{\Delta t}-\|U\|(K+\Delta t\|\mathbf{f}\|)$ we get the result. Therefore,
\begin{align*}
\dfrac{c}{\Delta t}\parallel q^{N_t-j+1}(u_1)&{}-q^{N_t-j+1}(u_2)\parallel \leq 2\parallel \mathbf{y}(u_1)-\mathbf{y}(u_2)\parallel \parallel U\parallel \parallel q^{N_t-j+1}(u_1)\parallel \\
&{}+ \dfrac{1}{\Delta t}\parallel q^{N_t-j}(u_1)-q^{N_t-j}(u_2)\parallel + \parallel H^TS^TR^{-1}SH\parallel \parallel \mathbf{y}(u_1)-\mathbf{y}(u_2)\parallel\\
&{}\mu\left[2\parallel U\parallel\rho \eta +\parallel H^TS^TR^{-1}SH\parallel \right]\parallel u_1-u_2\parallel + \dfrac{1}{\Delta t}\parallel q^{N_t-j}(u_1)-q^{N_t-j}(u_2)\parallel
\end{align*}
Using Lemma~\ref{lem:diff_y_bound} and the fact that $u_1,u_2\in N_0^\rho$ leads to:
\[
\parallel H^TS^TR^{-1}(SH\mathbf{y}-\mathbf{z})\parallel \leq c\parallel H^TS^TR^{-1}SH\parallel + \parallel H^TS^TR^{-1}\mathbf{z}\parallel=: \eta
\]
with $\eta>0$ a constant independent of $u_1,u_2$ and $\mathbf{y}$ or $\mathbf{p}$. Finally applying discrete Gronwall's inequality we get the result with
\[
\tilde{C}=C\exp(c-1)\left[\left(1+\dfrac{2\Delta t}{c}\right)\rho\eta\parallel U\parallel + 2 \parallel H^TS^TR^{-1}SH\parallel \right]. \qedhere
\]
\end{proof}

Next, we prove the uniform continuity of $\nabla f$, where $f$  is the reduced functional. %Due to the non-linearity of the state equation, we could not find an explicit expression for the reduced functional. For that reason we use the adjoint equation and get
\begin{align}\label{eq:grad}
\nabla f(u,w)&{}= \left[\begin{array}{c}
-e_u(\mathbf{y},u,w)^T p + \nabla_uJ(\mathbf{y},u,w)\\
-e_w(\mathbf{y},u,w)^T p + \nabla_wJ(\mathbf{y},u,w)
\end{array}\right]\\
&{}=\left[\begin{array}{c}
\dfrac{1}{\Delta t}p^1 + B^{-1}(u-u^b) + \alpha D^T h_\gamma(Du-w)\\
\mu w-\alpha h_\gamma(Du-w)+\beta E^Th_\gamma(Ew)
\end{array}
\right].
 \nonumber
\end{align}

\begin{theorem}\label{lem:Nabla_cont}
%  Let $u \in N_0^\rho$, where $N_0^\rho$ is the level set of the problem defined by
%  \[
%  N_0^\rho:=\{u+d\colon J_\gamma(\mathbf{y},u),w\leq J_\gamma(\mathbf{y}_0,u_0,w_0),\parallel d\parallel \leq \rho\},
%  \]
%  with $e(\mathbf{y}_0,u_0)=0$ and $\rho>0$.
 Let $\nabla f(u,w)$ be defined as in~\eqref{eq:grad}. Then , $\nabla f$ is uniformly continuous on the level set $N_0^\rho$.  %in the level set $N_0^\rho$, for $\Delta t$ small enough.
\end{theorem}
\begin{proof}
  The proof follows directly from the Heine Theorem since $f(u,w)$ is continuous over the level set $N_0^\rho$, which is compact. \qedhere
\end{proof}

\subsection{Polynomial line-search}
Concerning to the line--search procedure, we consider the polynomial backtracking algorithm proposed in~\cite{dennis1996numerical}. The main feature of this algorithm is that, unlike the usual backtracking procedure, we choose the step length $s_k$ by using a quadratic or cubic approximation of the objective function. Nevertheless, the step--length also satisfies the Armijo condition given by
\begin{equation} \label{eq: armijo}
  J_\gamma(\mathbf{y}_s,u_s,w_s)\leq J_\gamma(\mathbf{y},u,w)-c_1 s \nabla J_\gamma(\mathbf{y},u,w)^Td,
\end{equation}
where $u_s=u+sd$,  $\mathbf{y}_s$ is its associated state , $w_s=w+sd$ and $0<c_1<<1$ is a constant.
Using Lemma 2.2 from~\cite{hinze2008optimization} we can guarantee that there exist steps $s_k$ that satisfy the Armijo condition, no matter which backtraking rule we are using.

In particular, for the polynomial backtraking the following bounds for the step lengths hold. If in the first iteration the Armijo condition is not fulfilled the previous step--length is bounded by
\begin{equation}\label{eq:s_1_pol}
\tilde{s}<\dfrac{1}{2(1-c_1)}
\end{equation}
and if the Armijo condition is not fulfilled in the next iterations the following bound holds
\begin{equation}\label{eq:s_2_pol}
\tilde{s}<\dfrac{2}{3}s_{prev}
\end{equation}
where $s_{prev}$ is the previous  step--length.

The proof of the following proposition was developed following the techniques presented in~\cite{ulbrich2012nichtlineare}
\begin{proposition}
Let $\nabla f$ be uniformly continuous on the level set $N_0^\rho$.
If the iterations generated by Algorithm~\ref{alg:newton}, with $\{u_k,w_k\}_K$ convergent and $\{s_k\}$ satisfying the Armijo condition given by the polynomial backtracking line--search procedure, are such that
\begin{equation}\label{eq:condition_feas}
\parallel d^k\parallel \geq \varphi\left(\dfrac{-(\nabla f(u^k,w^k),d^k)}{\parallel d^k\parallel}\right), \quad\forall k\in K
\end{equation}
for $\varphi(\cdot)$ a non--decreasing and continuous function. Then $\{s_k\}_K$ satisfies the feasibility conditions of the line search.
\end{proposition}
\begin{proof}
We start we some notation. The $k$--th descent direction $d^k$ can be decomposed in the following way $d^k=(d^k_u,d^k_w)^T$. Furthermore, we denote as $f(u,w)$ the reduced objective function and  $\nabla f(u,w)$ its  gradient. % given in~\eqref{eq:nablaf1}.

Since Algorithm~\ref{alg:newton} generates descent directions and \eqref{eq: armijo} is fulfilled, $\{f(u^k,w^k)\}_{k\in K}$ is monotonically decreasing. Therefore, it is enough to prove that
\[
f(u^k+s_kd^k,w^k+s_kd^k)-f(u^k,w^k)\xrightarrow{k\to+\infty} 0\Longrightarrow \dfrac{\nabla f(u^k,w^k)^T d^k}{\parallel d^k\parallel}\xrightarrow{k\to+\infty} 0,\quad\forall k\in K.
\]
%or equivalently,
%\[
%\left(\dfrac{\nabla f(u^k,w^k)^T d^k}{\parallel d^k\parallel}\right)_K \nrightarrow 0 \Longrightarrow f(u^k+s_kd_u^k,w^k+s_kd_w^k)-f(u^kw,^k)\nrightarrow 0
%\]

Next we will argue by contradiction. Therefore, assuming that
\[
\left(\dfrac{\nabla f(u^k,w^k)^T d^k}{\parallel d^k\parallel}\right)\nrightarrow 0,\qquad \forall k\in K
\]
there exists a subsequence $\{(u^k,w^k)\}_{k\in K'}$, from $\{(u^k,w^k)\}_{k\in K}$ %{\color{red}aqui no queda claro que es K'}
with $\varepsilon>0$ such that
\[
\dfrac{-\nabla f(u^k,w^k)^Td^k}{\parallel d^k\parallel}\geq \varepsilon
\]
for all $k\in K'\subset K$. Now, since~\eqref{eq:condition_feas} holds, it follows that for all $k\in K'$
\[
\parallel d^k\parallel \geq \varphi\left(\dfrac{-\nabla f(u^k,w^k)^Td^k}{\parallel d^k \parallel}\right)\geq \varphi (\varepsilon)=:\delta > \varphi(0)\geq 0
\]
Using the mean value Theorem, we know that for all $k\in K'$, there exists $\tau_k\in [0,s_k]$ such that
\begin{align}
\nonumber&{}\dfrac{f(u^k+s_kd_u^k,w^k+sd_w^k)-f(u^k,w^k)}{\parallel s_kd^k\parallel }-\dfrac{s_kc_1\nabla f(u^k,w^k)^Td^k}{\parallel s_k d^k\parallel}\\
&{}=\dfrac{\nabla f(u^k+\tau_k d_u^k,w^k+\tau_k d_w^k  )^Td^k}{\parallel d^k \parallel }-\dfrac{c_1 \nabla f(u^k,w^k)^Td^k}{\parallel d^k\parallel }\\
\nonumber&{}=\left(\dfrac{\nabla f(u^k+\tau_kd_u^k,w^k + \tau_kd^k_w)^Td^k-\nabla f(u^k,w^k)^Td^k}{\parallel d^k\parallel}\right)+(1-c_1)\dfrac{\nabla f(u^k,w^k)^Td^k}{\parallel d^k\parallel}\\
\nonumber&{}\leq \parallel \nabla f(u^k+\tau_kd_u^k,w^k+\tau_kd_w^k)-\nabla f(u^k,w^k)\parallel + (1-c_1)\dfrac{\nabla f(u^k,w^k)^Td^k}{\parallel d^k\parallel}\\
\label{eq:condaa}&{}\leq \parallel \nabla f(u^k+\tau_kd_u^k,w^k+\tau_kd_w^k)-\nabla f(u^k,w^k)\parallel - (1-c_1)\varepsilon
\end{align}
On the other hand, since $\nabla f$ is uniformly continuous on the compact level set $N_0^\rho$, there exists $\eta$ independent from $u^k$ and $w^k$ such that
\[
\parallel \nabla f(u^k +d_u,w^k+d_w)-\nabla f(u^k,w^k)\parallel < (1-c_1)\varepsilon\qquad\forall k\in K', d\colon \parallel d \parallel \leq \eta
\]
Jointly with~\eqref{eq:condaa} and the result above we had proved that the Armijo condition holds when $s_k\parallel d^k \parallel \leq \eta$.
Using the bounds obtained for the polynomial backtracking~\eqref{eq:s_1_pol} and~\eqref{eq:s_2_pol} we know that $s_k=1$ and the Armijo condition is fulfilled in the first iteration or we have two possible scenarios:
\begin{description}
\item[Armijo condition is fulfilled in the second iteration]
\[
s_k\leq \dfrac{1}{2(1-c_1)}=:\sigma <1\hbox{ and } s_k\parallel d^k\parallel > \eta
\]
\item[Armijo condition is fulfilled in a later iteration]
\[
s_k\leq \dfrac{2}{3}s_{prev}< \left(\dfrac{2\sigma}{3}\right)^n\hbox{ and } \dfrac{s_k}{\left(\nicefrac{2\sigma}{3}\right)}\parallel d^k\parallel > \eta
\]
\end{description}
Therefore, for all $k\in K'$ it follows that
\[
s_k\parallel d^k\parallel \geq \min\left\{\eta,\dfrac{2\sigma\eta}{3},\delta\right\}=: \theta>0
\]
Using again the Armijo condition we have
\begin{align*}
f(u^k,w^k)- f(u^k + s_kd_u^k,w^k+s_kd_w^k)&{}\geq -s_kc_1 \nabla f(u^k,w^k)^T d^k\\
&{}= c_1 \left(\dfrac{-\nabla f(u^k,w^k)^T d^k}{\parallel d^k\parallel }\right)(s_k\parallel d^k\parallel)\geq c_1 \delta\theta>0
\end{align*}
Consequently,
\[
f(u^k,w^k)- f(u^k + s_kd_u^k,w^k + s_kd_w^k)\nrightarrow 0.
\]
and we get the result.\qedhere
\end{proof}

Condition \eqref{eq:condition_feas} is fulfilled automatically for our algorithm since the matrices $\{\Pi_k\}$ satisfy the condition given in Proposition~\ref{prop:mM}, with $\varphi(x)=\dfrac{x}{M}$ as the non-decreasing and continuous function. In that case, using the notation given in the reduced system~\eqref{eq:Pi_system} we have that:
\begin{align*}
\dfrac{-\nabla f(u^k,w^k)^Td^k}{\parallel d^k\parallel}&{}=\dfrac{\nabla f(u^k,w^k)^T\Pi_k^{-1}\Pi_k\Pi_k^{-1}\nabla f(u_k,w^k)}{\parallel d^k\parallel},\\
&{}=\dfrac{(d^k)^T\Pi_kd^{k}}{\parallel d^k\parallel }\leq \dfrac{M\parallel d^k\parallel^2}{\parallel d^k\parallel }=M\|d^k\|
\end{align*}
Then,
\[
\varphi\left(\dfrac{-(\nabla f(u^k,w^k),d^k)}{\parallel d^k\parallel}\right) \leq \varphi(M\|d^k\|) = \| d^k\|.
\]
and $\nabla f$ is uniformly continuous thanks to Lemma~\ref{lem:Nabla_cont}. Thus, we can guarantee that the feasibility conditions in our algorithm are fulfilled and,
consequently, every accumulation point of the sequence $\{u^k,w^k\}$ generated by Algorithm~\ref{alg:newton} is a stationary point of problem~\eqref{eq:reg_discrete_problem_tgv}.

%%%%%%%%%%%%%%%%%%%%%%%%%%%%%%%
%%%%%%%%%%%%%%%%%%%%%%%%%%%%%%%
%%%%%%%%%%%%%%%%%%%%%%%%%%%%%%%

\section{Numerical Experiments}
In this section, we report on numerical experiments whose primary goal is to show the behavior of TGV concerning other regularizations for the solution of the data assimilation problem of the inviscid Burgers equation. Moreover, the performance of the proposed solution algorithm for the TGV data assimilation problem is experimentally verified.

Two main experiments have been designed for these purposes. The first experiment aims to verify the well-known staircase effect of total variation (TV) and also to show how the total generalized variation regularization (TGV) behaves in such cases. The second experiment has been constructed with the purpose of illustrating the performance of the proposed algorithm regarding its convergence, and its behavior with respect to the regularization parameters.

We perform the following steps for the construction of each experiment.
\begin{enumerate}
\item Set an exact solution (initial condition).
\item Solve the state equation with the exact solution and extract perfect observations.
\item Add a Gaussian noise to the exact solution for the background information.
\item Fix the parameters $\alpha$, $\beta$, $\gamma$ and $\mu$.
\end{enumerate}

The parameters for the TGV--regularization were chosen following the heuristic described in~\cite{de2017bilevel}, where $\dfrac{\beta}{\alpha}\in \dfrac{1}{n}(0.75,1.5).$
For the background matrix we set $B=(0.1)\mathbb{I}$, and we assume perfect observations, i.e., $R=\mathbb{I}$. Finally, as stopping criteria we use
\[
| u^{k}-u^{k-1}| <1e-3.
\]

\subsection{Comparison between TV and TGV regularization}
This experiment is designed with the purpose of showing the staircase effect of the TV regularization on the data assimilation problem and the way how the TGV regularization eliminates it. We solve problem~\eqref{eq:reg_discrete_optimality_system} with $n=50$ spatial discretization points and $N_t=150$ points for the time discretization. Furthermore, we use 25 observations from 7500 possible choices.
The exact solution for this experiment is given by the following expression:
\begin{figure}[!htb]
\centering
\begin{minipage}{.5\textwidth}
%u_ex=(x>=2.5).*(x<=5).*((2*x/5)-1)+(x>=5).*(x<=7.5).*((-4*x/5)+6);
\centering
\[
u_{\hbox{ex}}=\left\{
\begin{array}{ll}
\dfrac{2x}{5}-1&\hbox{if }2.5\leq x\leq 5,\\
-\dfrac{4x}{5}+6&\hbox{if }5\leq x\leq 7.5,\\
0&\hbox{otherwise.}
\end{array}
\right.
\]
%\caption{$dt=0.1$}
\label{fig:prob1_6_2}
\end{minipage}%
\begin{minipage}{0.5\textwidth}
\centering
\begin{figure}[H]
\begin{center}
\begin{tikzpicture}[scale=0.75]
\begin{axis}[xlabel={$\Omega$},
 legend pos=north east,
 legend style={font=\tiny,draw=none}]
\addplot[domain=0:2.5,blue] {0};
\addplot[domain=2.5:5,blue] {(2*x/5)-1};
\addplot[domain=5:7.5,blue] {(-4*x/5)+6};
\addplot[domain=7.5:10,blue] {0};
\draw[dotted] (axis cs: 5,1) -- (axis cs: 5,2);
\addplot[holdot] coordinates{(5,1)};
\addplot[soldot] coordinates{(5,2)};
%\legend{Exact solution}
\end{axis}
\end{tikzpicture}
\end{center}
\end{figure}
\end{minipage}
\caption{Exact solution }
\label{fig:func1_exacta}
\end{figure}

As quality measure we will use the SSIM (structural similarity index)~\cite{wang2004image}, typically used in image and signal processing. This index is given by the following expression
\[
SSIM(x,y)=\dfrac{(2\mu_x\mu_y+C_1)(2\sigma_{xy}+C_2)}{(\mu_x^2+\mu_y^2+C_1)(\sigma_x^2+\sigma_y^2+C_2)}
\]
where $\mu_x,\mu_y$ are the mean and $\sigma_x^2,\sigma_y^2$ are the variances of the vectors $x,y$, respectively. Furthermore, $\sigma_{xy}$ is the covariance between $x,y$ and the constants $C_1=k_1 L^2$, $C_2=k_2 L^2$ with $k_1=0.01$, $k_2=0.03$ and $L=2$. It is easy to note that $\hbox{SSIM}=1$ if the two functions are identical. Therefore, we will look for a set of parameters that guarantee us to be as close as possible to 1. Figure~\ref{fig:TV_TGV_SSIM} shows the value of the SSIM index as we vary the parameters for the TV and TGV regularizations.
\begin{figure}
\centering
\input{SSIM_surface_TGV.tex}
\caption{SSIM index for the TV and TGV regularizations}
\label{fig:TV_TGV_SSIM}
\end{figure}
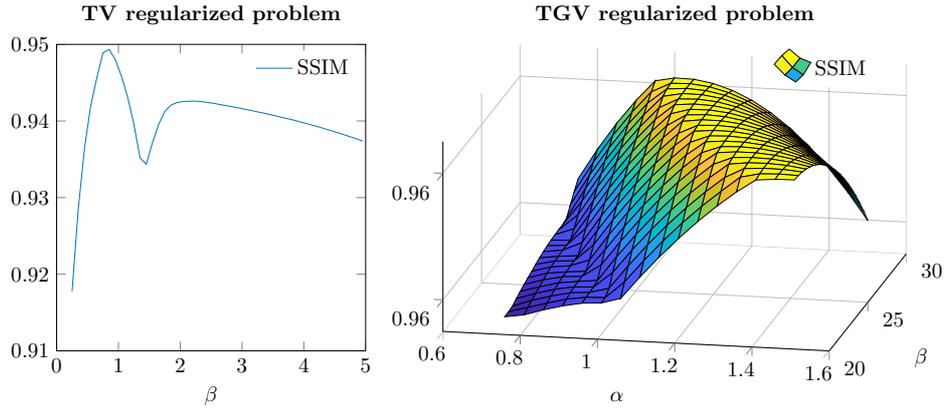

Now we choose the maximum value of SSIM for each problem and the solutions obtained for the best cases in terms of the SSIM index are presented in Figure~\ref{fig:TV_TGV_solutions}.
The best TV solution was obtained by setting $\beta=0.85$ and the algorithm converges after 21 iterations  with $\gamma=1e5$ and $SSIM=0.9495$. On the other hand, the best TGV solution was obtained with $\alpha=23.5$ and $\beta=0.611$ and the algorithm converges after 15 iterations with $\gamma=1e4$, $\mu=1e-10$ and the value of $SSIM=0.9581$. Furthermore, Figure~\ref{fig:state_vs_observations} shows the way the optimal state obtained with the TGV regularization fits with the given observations for 3 specific snapshots.
\begin{figure}
\centering
\input{BEST_solution.tex}
\caption{Best solutions for TV (left) and TGV (right) regularizations}
\label{fig:TV_TGV_solutions}
\end{figure}
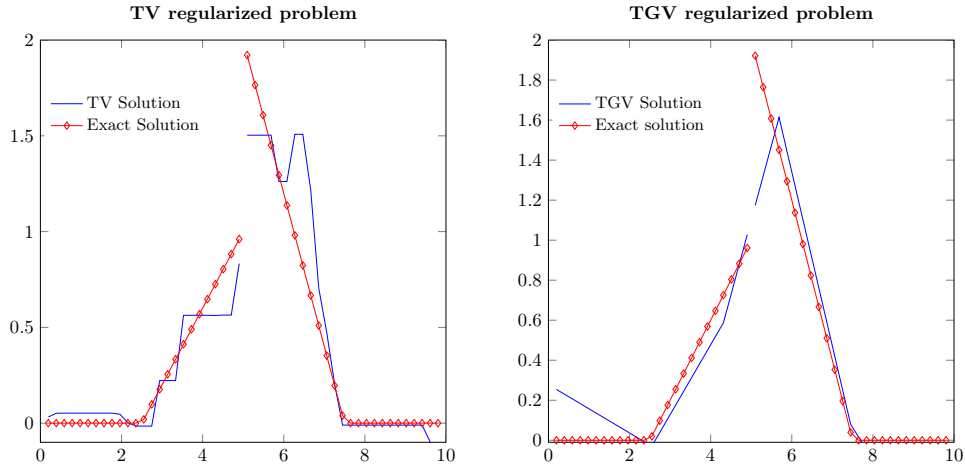

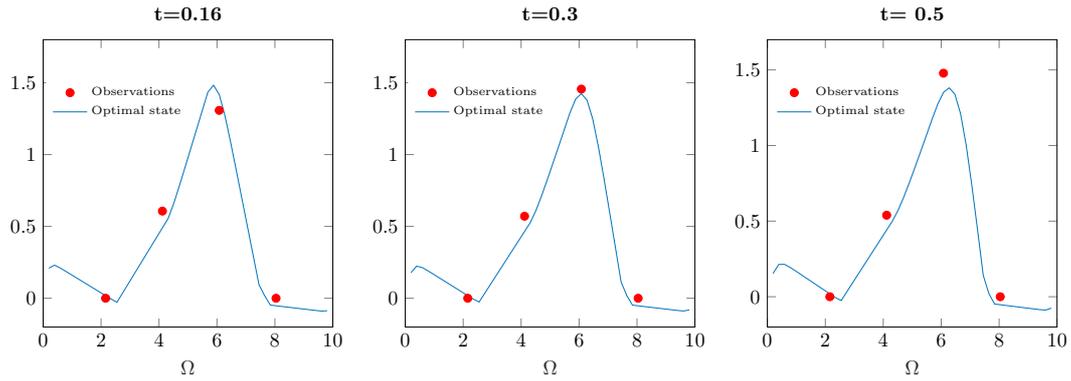
\begin{figure}
\centering
\input{state_vs_observations.tex}
\caption{Observations vs Optimal state}
\label{fig:state_vs_observations}
\end{figure}

\subsection{Performance of the algorithm}
This experiment is devoted to the study of the performance of the algorithm concerning its rate of convergence, global convergence, behavior with respect to the parameters and generation of descent directions. We use different exact solutions, shown in  Figure~\ref{fig:exact_solutions}.
\begin{figure}[!htb]
\centering
\begin{minipage}{.3\textwidth}
\centering
\begin{figure}[H]
\begin{center}
\begin{tikzpicture}[scale=0.75]
\begin{axis}[xlabel={$\Omega$},
 legend pos=north east,
 legend style={font=\tiny,draw=none},
 title style={font=\bfseries},
 title={Experiment 1}]
\addplot[domain=0:5,blue] {x/5};
\addplot[domain=5:10,blue] {(-2/5)*(x-10)};
\draw[dotted] (axis cs: 5,1) -- (axis cs: 5,2);
\addplot[holdot] coordinates{(5,1)};
\addplot[soldot] coordinates{(5,2)};
%\legend{Exact solution}
\end{axis}
\end{tikzpicture}
\end{center}
\end{figure}
\end{minipage}%
\begin{minipage}{0.3\textwidth}
\centering
\begin{figure}[H]
\begin{center}
\begin{tikzpicture}[scale=0.75]
\begin{axis}[xlabel={$\Omega$},
 legend pos=north east,
 legend style={font=\tiny,draw=none},
 title style={font=\bfseries},
title={Experiment 2}]
\addplot[domain=0:2.5,blue] {0};
\addplot[domain=2.5:5,blue] {(2*x/5-1)};
\addplot[domain=5:7.5,blue] {(-4*x/5+6)};
\addplot[domain=7.5:10,blue] {0};
\draw[dotted] (axis cs: 5,1) -- (axis cs: 5,2);
\addplot[holdot] coordinates{(5,1)};
\addplot[soldot] coordinates{(5,2)};
%\legend{Exact solution}
\end{axis}
\end{tikzpicture}
\end{center}
\end{figure}
\end{minipage}
\begin{minipage}{0.3\textwidth}
\begin{figure}[H]
\begin{center}
\begin{tikzpicture}[scale=0.75]
\begin{axis}[xlabel={$\Omega$},
 legend pos=north east,
 legend style={font=\tiny,draw=none},
 title style={font=\bfseries},
title={Experiment 3}]
\addplot[domain=0:2,blue] {3*x/2};
\addplot[domain=2:5,blue] {x-2};
\addplot[domain=5:8,blue] {x-5};
\addplot[domain=8:10,blue] {(-x+10)};
\draw[dotted] (axis cs: 2,0) -- (axis cs: 2,3);
\draw[dotted] (axis cs: 5,0) -- (axis cs: 5,3);
\draw[dotted] (axis cs: 8,2) -- (axis cs: 8,3);
\addplot[holdot] coordinates{(2,0) (5,0) (8,2)};
\addplot[soldot] coordinates{(2,3) (5,3) (8,3)};
%\legend{Exact solution}
\end{axis}
\end{tikzpicture}
\end{center}
\end{figure}
\end{minipage}
\caption{Exact solutions for the experiment}
\label{fig:exact_solutions}
\end{figure}

The first part of the experiment is concerned with the global convergence of the algorithm. For this purpose we use different exact solutions and provide the algorithm with different starting points.  In particular, we use 5 different initial points which are $u^0\equiv 0.5,1,2$, $u^0$ randomly generated with a uniform distribution in the interval $[0,1]$ and $u^0$ as the solution of the problem with the TV regularization. Figure~\ref{fig:global_convergence} shows that, in each case, the same value of the objective function is obtained when the algorithm stops. Moreover, this experiment allows us to show numerically that in all the cases the algorithm generates in fact descent directions. We use a fixed value of the parameters $\alpha,\beta$ as follows: for experiments 1 and 2 we use $\alpha=10$ and $\beta=0.2$ and for experiment 3, $\alpha=5,\beta=0.1$. Furthermore, we use the parameters  $\gamma=1e4$, $\mu=1e-10$.

\begin{figure}
\centering
\input{descent_objective.tex}
\caption{Global convergence of the algorithm}
\label{fig:global_convergence}
\end{figure}
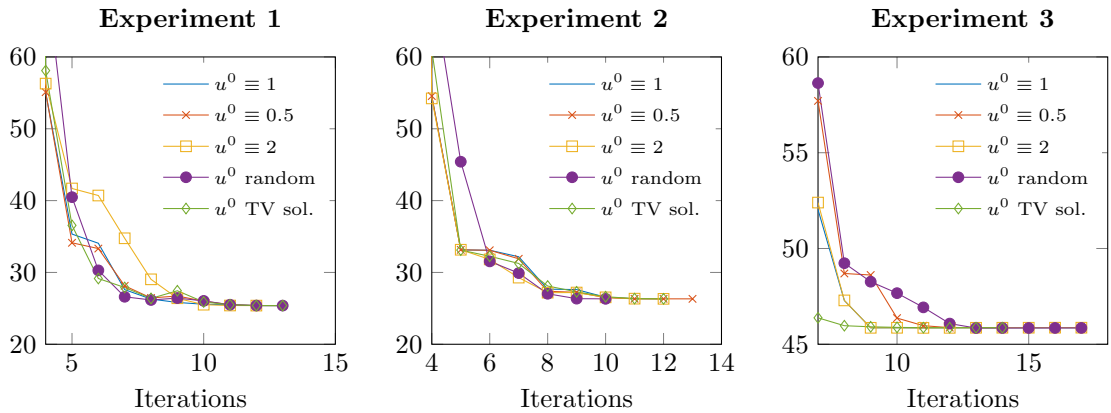

Another critical aspect in the study of the performance of the algorithm is the analysis of the rate of convergence. Since the proposed method is a second order method, it is natural to expect a local superlinear order, and that is what is analyzed in the following experiment. For this purpose, we solve the same problems as in the previous case but setting the initial value as $u^0\equiv 1$, and we vary the parameters $\alpha,\beta$. Figure~\ref{fig:superlinear_convergence} shows the decreasing of the residual computed as
\[
RES=\dfrac{\|u^{k}-u^\ast\|}{\|u^{k-1}-u^\ast\|},
\]
where we use $u^\ast$ as the final value to which the algorithm converges. Furtheremore, this Figure suggests the algorithm converges locally superlinearly, and since these results have been obtained by varying the values of $\alpha$ and $\beta$ we can also assure that this behavior is independent of the value of the parameters of the regularization.
\begin{figure}
\centering
\input{superlineal_convergence.tex}
\caption{Locally superlinear convergence of the algorithm}
\label{fig:superlinear_convergence}
\end{figure}
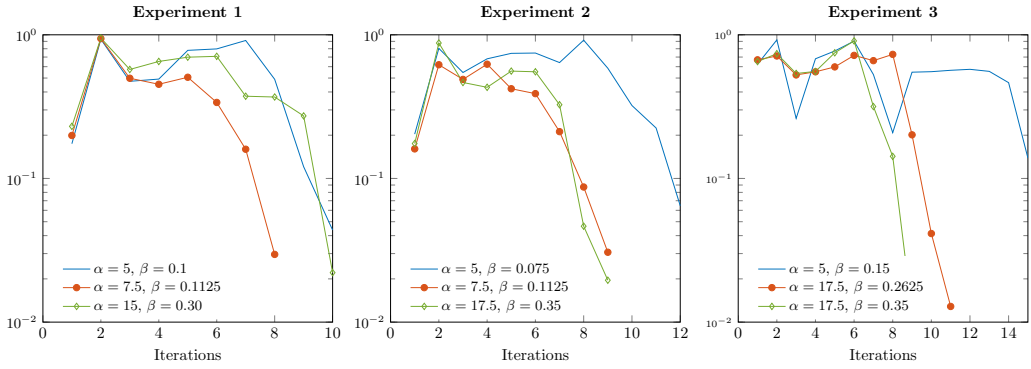

Finally, we study the behavior of the algorithm when we vary the value of $\mu$. From the theoretical point of view, the addition of this parameter is important to guarantee the generation of descent directions. However, we are also aware that adding this term is only an artifice. It is for this reason that the value of this parameter is important to obtain an equivalence with the original problem without losing the benefit that it has on the performance of the algorithm and also the principal aim of the last experiment.  We solve two experiments using as exact solutions the ones depicted in Figure~\ref{fig:exact_solutions_mu}. Both experiments were solved with $\alpha=2.5$, $\beta=$ and $\gamma=1e4$. The results are presented in Table~\ref{tab:varying_mu}, where \emph{NaN} means that the algorithm could not find a descent direction due to the singularity of the matrix. Adding the term involving the quadratic norm of $w$ let us assure that the matrix is positive definite and therefore nonsingular, that is, in any iteration, the algorithm finds a descent direction. Finally Table~\ref{tab:varying_mu} also shows that we can use small values for the parameter $\mu$ without changing the optimal value of the objective function.

\begin{figure}[!htb]
\begin{minipage}{0.3\textwidth}
\begin{figure}[H]
\begin{center}
\begin{tikzpicture}[scale=0.75]
\begin{axis}[xlabel={$\Omega$},
 legend pos=north east,
 legend style={font=\tiny,draw=none},
 title style={font=\bfseries},
 title={Experiment 3}]
\addplot[domain=0:2,blue] {3*x/2};
\addplot[domain=2:5,blue] {x-2};
\addplot[domain=5:8,blue] {x-5};
\addplot[domain=8:10,blue] {(-x+10)};
\draw[dotted] (axis cs: 2,0) -- (axis cs: 2,3);
\draw[dotted] (axis cs: 5,0) -- (axis cs: 5,3);
\draw[dotted] (axis cs: 8,2) -- (axis cs: 8,3);
\addplot[holdot] coordinates{(2,0) (5,0) (8,2)};
\addplot[soldot] coordinates{(2,3) (5,3) (8,3)};
%\legend{Exact solution}
\end{axis}
\end{tikzpicture}
\end{center}
%\caption{Exact solution for experiment 5}
\end{figure}
\end{minipage}
\centering
\begin{minipage}{0.3\textwidth}
\centering
\begin{figure}[H]
\begin{center}
\begin{tikzpicture}[scale=0.75]
\begin{axis}[xlabel={$\Omega$},
 legend pos=north east,
 legend style={font=\tiny,draw=none},
 title style={font=\bfseries},
 title={Experiment 4}]
\addplot[domain=0:2,blue] {(3*x/2)};
\addplot[domain=2:4,blue] {(x-2)};
\addplot[domain=4:5,blue] {(2)};
\addplot[domain=5:6,blue] {3*x-15};
\addplot[domain=6:8,blue] {(-3/4)*x+6};
\addplot[domain=8:10,blue] {0};
\draw[dotted] (axis cs: 2,0) -- (axis cs: 2,3);
\draw[dotted] (axis cs: 5,0) -- (axis cs: 5,2);
\draw[dotted] (axis cs: 6,3) -- (axis cs: 6,1.5);
\addplot[holdot] coordinates{(5,0) (2,0) (6,1.5)};
\addplot[soldot] coordinates{(5,2) (2,3) (6,3)};
%\legend{Exact solution}
\end{axis}
\end{tikzpicture}
\end{center}
%\caption{\scriptsize{Exact solution for experiment 4}}
\end{figure}
\end{minipage}
\caption{Exact solutions for the experiment}
\label{fig:exact_solutions_mu}
\end{figure}

\begin{table}[]
\centering
\caption{Results of the experiments with different values of $\mu$ }
\label{tab:varying_mu}
\begin{tabular}{|c|c|c|c|c|}
\hline
Experiment          & $\mu$	& iter 	& SSIM  	& F \\ \hline
\multirow{5}{*}{3}  & 0  	&  --  	& NaN		& NaN \\ \cline{2-5}
                    & 1e-6  &  13   & 0.9590	& 39.2468\\ \cline{2-5}
                    & 1e-8  &  12   & 0.9592	& 39.2434\\ \cline{2-5}
                    & 1e-10 &  12   & 0.9592	& 39.2434 \\ \cline{2-5}
                    & 1e-12 &  13   & 0.9592	& 39.2434\\ \hline
\multirow{5}{*}{4}  & 0  	& --   	& NaN		& NaN\\ \cline{2-5}
                    & 1e-6  & 12    & 0.9592	& 35.8368\\ \cline{2-5}
                    & 1e-8 	&  13   & 0.9594	& 35.8332\\ \cline{2-5}
                    & 1e-10 & 13    & 0.9594	& 35.8332\\ \cline{2-5}
                    & 1e-12 & 16    & 0.9594	& 35.8331\\ \cline{1-5}

\end{tabular}
\end{table}

\bibliographystyle{plain}
\bibliography{biblioTGVBurgers}
%\nocite{*}

\end{document}

%% file: SSIM_surface_TGV.tex
\definecolor{mycolor1}{rgb}{0.00000,0.44700,0.74100}%
\begin{tikzpicture}[scale=0.8]

\begin{axis}[%
width=2in,
height=2in,
at={(0in,0.481in)},
scale only axis,
xlabel={$\beta$},
xmin=0,
xmax=5,
ymin=0.91,
ymax=0.95,
axis background/.style={fill=white},
legend style={legend cell align=left, align=left, draw=white!15!white},
title style={font=\bfseries},
 title={TV regularized problem}]
]
\addplot [color=mycolor1]
  table[row sep=crcr]{%
0.25	0.917766616841741\\
0.35	0.928483463007088\\
0.45	0.936357412974057\\
0.55	0.941935599094487\\
0.65	0.945629715909434\\
0.75	0.948897662230417\\
0.85	0.949332580270104\\
0.95	0.94800125649831\\
1.05	0.945819173294259\\
1.15	0.943077341760768\\
1.25	0.939574547578742\\
1.35	0.935178838146576\\
1.45	0.934361095960717\\
1.55	0.937154481296492\\
1.65	0.939548298552809\\
1.75	0.941244284007365\\
1.85	0.942073769407147\\
1.95	0.942415444553723\\
2.05	0.942531938469788\\
2.15	0.94258921738851\\
2.25	0.942594007353969\\
2.35	0.942521509608846\\
2.45	0.942403211067091\\
2.55	0.942292037871536\\
2.65	0.942131659385232\\
2.75	0.941998660285272\\
2.85	0.941833303927692\\
2.95	0.941692114099615\\
3.05	0.941519700741915\\
3.15	0.941368954959991\\
3.25	0.941187832938524\\
3.35	0.941026309408058\\
3.45	0.940835262286638\\
3.55	0.940662332386698\\
3.65	0.940460014511501\\
3.75	0.940274374081947\\
3.85	0.940082518329363\\
3.95	0.93986084436171\\
4.05	0.939655583566111\\
4.15	0.939420746251704\\
4.25	0.939202025559455\\
4.35	0.938953669660781\\
4.45	0.938721738912904\\
4.55	0.938459387961434\\
4.65	0.938213319735944\\
4.75	0.937936682974073\\
4.85	0.937676295957869\\
4.95	0.937385110666688\\
};
\addlegendentry{SSIM}
\end{axis}
%\end{tikzpicture}
%\hfill
%\begin{tikzpicture}

\begin{axis}[%
width=3in,
height=2in,
at={(2.5in,0.481in)},
scale only axis,
xlabel={$\alpha$},
ylabel={$\beta$},
xmin=0.6,
xmax=1.6,
tick align=outside,
ymin=20,
ymax=30,
zmin=0.9555,
zmax=0.9585,
view={11.3}{27.6},
axis background/.style={fill=white},
axis x line*=bottom,
axis y line*=left,
axis z line*=left,
xmajorgrids,
ymajorgrids,
zmajorgrids,
legend style={at={(0.7,1)}, anchor=north west, legend cell align=left, align=left, draw=white!15!white},
title style={font=\bfseries},
 title={TGV regularized problem}]
]

\addplot3[%
surf,
shader=flat corner, draw=black, z buffer=sort, colormap={mymap}{[1pt] rgb(0pt)=(0.2422,0.1504,0.6603); rgb(1pt)=(0.25039,0.164995,0.707614); rgb(2pt)=(0.257771,0.181781,0.751138); rgb(3pt)=(0.264729,0.197757,0.795214); rgb(4pt)=(0.270648,0.214676,0.836371); rgb(5pt)=(0.275114,0.234238,0.870986); rgb(6pt)=(0.2783,0.255871,0.899071); rgb(7pt)=(0.280333,0.278233,0.9221); rgb(8pt)=(0.281338,0.300595,0.941376); rgb(9pt)=(0.281014,0.322757,0.957886); rgb(10pt)=(0.279467,0.344671,0.971676); rgb(11pt)=(0.275971,0.366681,0.982905); rgb(12pt)=(0.269914,0.3892,0.9906); rgb(13pt)=(0.260243,0.412329,0.995157); rgb(14pt)=(0.244033,0.435833,0.998833); rgb(15pt)=(0.220643,0.460257,0.997286); rgb(16pt)=(0.196333,0.484719,0.989152); rgb(17pt)=(0.183405,0.507371,0.979795); rgb(18pt)=(0.178643,0.528857,0.968157); rgb(19pt)=(0.176438,0.549905,0.952019); rgb(20pt)=(0.168743,0.570262,0.935871); rgb(21pt)=(0.154,0.5902,0.9218); rgb(22pt)=(0.146029,0.609119,0.907857); rgb(23pt)=(0.138024,0.627629,0.89729); rgb(24pt)=(0.124814,0.645929,0.888343); rgb(25pt)=(0.111252,0.6635,0.876314); rgb(26pt)=(0.0952095,0.679829,0.859781); rgb(27pt)=(0.0688714,0.694771,0.839357); rgb(28pt)=(0.0296667,0.708167,0.816333); rgb(29pt)=(0.00357143,0.720267,0.7917); rgb(30pt)=(0.00665714,0.731214,0.766014); rgb(31pt)=(0.0433286,0.741095,0.73941); rgb(32pt)=(0.0963952,0.75,0.712038); rgb(33pt)=(0.140771,0.7584,0.684157); rgb(34pt)=(0.1717,0.766962,0.655443); rgb(35pt)=(0.193767,0.775767,0.6251); rgb(36pt)=(0.216086,0.7843,0.5923); rgb(37pt)=(0.246957,0.791795,0.556743); rgb(38pt)=(0.290614,0.79729,0.518829); rgb(39pt)=(0.340643,0.8008,0.478857); rgb(40pt)=(0.3909,0.802871,0.435448); rgb(41pt)=(0.445629,0.802419,0.390919); rgb(42pt)=(0.5044,0.7993,0.348); rgb(43pt)=(0.561562,0.794233,0.304481); rgb(44pt)=(0.617395,0.787619,0.261238); rgb(45pt)=(0.671986,0.779271,0.2227); rgb(46pt)=(0.7242,0.769843,0.191029); rgb(47pt)=(0.773833,0.759805,0.16461); rgb(48pt)=(0.820314,0.749814,0.153529); rgb(49pt)=(0.863433,0.7406,0.159633); rgb(50pt)=(0.903543,0.733029,0.177414); rgb(51pt)=(0.939257,0.728786,0.209957); rgb(52pt)=(0.972757,0.729771,0.239443); rgb(53pt)=(0.995648,0.743371,0.237148); rgb(54pt)=(0.996986,0.765857,0.219943); rgb(55pt)=(0.995205,0.789252,0.202762); rgb(56pt)=(0.9892,0.813567,0.188533); rgb(57pt)=(0.978629,0.838629,0.176557); rgb(58pt)=(0.967648,0.8639,0.16429); rgb(59pt)=(0.96101,0.889019,0.153676); rgb(60pt)=(0.959671,0.913457,0.142257); rgb(61pt)=(0.962795,0.937338,0.12651); rgb(62pt)=(0.969114,0.960629,0.106362); rgb(63pt)=(0.9769,0.9839,0.0805)}, mesh/rows=16]
table[row sep=crcr, point meta=\thisrow{c}] {%
x	y	z	c\\
0.75	20.5	0.955700908169542	0.955700908169542\\
0.75	21	0.955705208562196	0.955705208562196\\
0.75	21.5	0.955717298298227	0.955717298298227\\
0.75	22	0.955757808412333	0.955757808412333\\
0.75	22.5	0.955794317457559	0.955794317457559\\
0.75	23	0.955827093533647	0.955827093533647\\
0.75	23.5	0.95585612557294	0.95585612557294\\
0.75	24	0.955882441078922	0.955882441078922\\
0.75	24.5	0.955913402896018	0.955913402896018\\
0.75	25	0.955939954315692	0.955939954315692\\
0.75	25.5	0.955961528222269	0.955961528222269\\
0.75	26	0.955978632476925	0.955978632476925\\
0.75	26.5	0.955990858150562	0.955990858150562\\
0.75	27	0.955998609437563	0.955998609437563\\
0.75	27.5	0.956001078762953	0.956001078762953\\
0.75	28	0.956003744078545	0.956003744078545\\
0.75	28.5	0.956029401126448	0.956029401126448\\
0.75	29	0.956163361172749	0.956163361172749\\
0.75	29.5	0.956294594761192	0.956294594761192\\
0.75	30	0.956421709054459	0.956421709054459\\
0.8	20.5	0.95574706460845	0.95574706460845\\
0.8	21	0.9557872962438	0.9557872962438\\
0.8	21.5	0.9558229424343	0.9558229424343\\
0.8	22	0.955854980166517	0.955854980166517\\
0.8	22.5	0.955882319907202	0.955882319907202\\
0.8	23	0.955916003973475	0.955916003973475\\
0.8	23.5	0.955943117355852	0.955943117355852\\
0.8	24	0.955965435483516	0.955965435483516\\
0.8	24.5	0.95598267990052	0.95598267990052\\
0.8	25	0.955981529886333	0.955981529886333\\
0.8	25.5	0.956000168000739	0.956000168000739\\
0.8	26	0.956001568038667	0.956001568038667\\
0.8	26.5	0.956004376517162	0.956004376517162\\
0.8	27	0.956110227028324	0.956110227028324\\
0.8	27.5	0.956251035254505	0.956251035254505\\
0.8	28	0.95638794571012	0.95638794571012\\
0.8	28.5	0.956521685819707	0.956521685819707\\
0.8	29	0.956651833383795	0.956651833383795\\
0.8	29.5	0.956778245704011	0.956778245704011\\
0.8	30	0.956901264840434	0.956901264840434\\
0.85	20.5	0.955841117821066	0.955841117821066\\
0.85	21	0.955871731682246	0.955871731682246\\
0.85	21.5	0.95590556551367	0.95590556551367\\
0.85	22	0.95593670839024	0.95593670839024\\
0.85	22.5	0.955961661574276	0.955961661574276\\
0.85	23	0.955980535460375	0.955980535460375\\
0.85	23.5	0.955980548568246	0.955980548568246\\
0.85	24	0.956000170420628	0.956000170420628\\
0.85	24.5	0.956002135191028	0.956002135191028\\
0.85	25	0.956004524281123	0.956004524281123\\
0.85	25.5	0.956136384576665	0.956136384576665\\
0.85	26	0.95628592774451	0.95628592774451\\
0.85	26.5	0.956430122374585	0.956430122374585\\
0.85	27	0.956570908535063	0.956570908535063\\
0.85	27.5	0.956707660101434	0.956707660101434\\
0.85	28	0.956840128033599	0.956840128033599\\
0.85	28.5	0.956968839935942	0.956968839935942\\
0.85	29	0.957093359184683	0.957093359184683\\
0.85	29.5	0.95721373879469	0.95721373879469\\
0.85	30	0.95733003872472	0.95733003872472\\
0.9	20.5	0.955919197589546	0.955919197589546\\
0.9	21	0.955949173732677	0.955949173732677\\
0.9	21.5	0.955972658519001	0.955972658519001\\
0.9	22	0.955988849945505	0.955988849945505\\
0.9	22.5	0.955998606079814	0.955998606079814\\
0.9	23	0.956001002763156	0.956001002763156\\
0.9	23.5	0.956004345001529	0.956004345001529\\
0.9	24	0.956110233801503	0.956110233801503\\
0.9	24.5	0.956268463615903	0.956268463615903\\
0.9	25	0.956421718895592	0.956421718895592\\
0.9	25.5	0.956570911729003	0.956570911729003\\
0.9	26	0.956715583963996	0.956715583963996\\
0.9	26.5	0.956855422847586	0.956855422847586\\
0.9	27	0.956991446384152	0.956991446384152\\
0.9	27.5	0.957122054988826	0.957122054988826\\
0.9	28	0.957248374620001	0.957248374620001\\
0.9	28.5	0.95737010261396	0.95737010261396\\
0.9	29	0.957487202997537	0.957487202997537\\
0.9	29.5	0.957599631085915	0.957599631085915\\
0.9	30	0.95770740510977	0.95770740510977\\
0.95	20.5	0.955977691176031	0.955977691176031\\
0.95	21	0.955993122616023	0.955993122616023\\
0.95	21.5	0.955994681883887	0.955994681883887\\
0.95	22	0.956002936374913	0.956002936374913\\
0.95	22.5	0.956029414797135	0.956029414797135\\
0.95	23	0.956198594764559	0.956198594764559\\
0.95	23.5	0.956362484198703	0.956362484198703\\
0.95	24	0.956521695452454	0.956521695452454\\
0.95	24.5	0.956675820609341	0.956675820609341\\
0.95	25	0.956824794955218	0.956824794955218\\
0.95	25.5	0.956968814003477	0.956968814003477\\
0.95	26	0.957107740200252	0.957107740200252\\
0.95	26.5	0.957241479513732	0.957241479513732\\
0.95	27	0.957370105558091	0.957370105558091\\
0.95	27.5	0.957493576787172	0.957493576787172\\
0.95	28	0.95761184115947	0.95761184115947\\
0.95	28.5	0.957724918634	0.957724918634\\
0.95	29	0.957832798265743	0.957832798265743\\
0.95	29.5	0.957935470792243	0.957935470792243\\
0.95	30	0.958032868772021	0.958032868772021\\
1	20.5	0.955996466858435	0.955996466858435\\
1	21	0.956003772280403	0.956003772280403\\
1	21.5	0.956074384883735	0.956074384883735\\
1	22	0.956250983150947	0.956250983150947\\
1	22.5	0.956421725593163	0.956421725593163\\
1	23	0.95658721475192	0.95658721475192\\
1	23.5	0.956747055278008	0.956747055278008\\
1	24	0.956901277278279	0.956901277278279\\
1	24.5	0.957049901665445	0.957049901665445\\
1	25	0.957192776709668	0.957192776709668\\
1	25.5	0.957330018699151	0.957330018699151\\
1	26	0.957461586326319	0.957461586326319\\
1	26.5	0.957587387768968	0.957587387768968\\
1	27	0.957707410616805	0.957707410616805\\
1	27.5	0.957821691554353	0.957821691554353\\
1	28	0.95793020173283	0.95793020173283\\
1	28.5	0.958032871439217	0.958032871439217\\
1	29	0.958073375952427	0.958073375952427\\
1	29.5	0.958085858404339	0.958085858404339\\
1	30	0.958093592265607	0.958093592265607\\
1.05	20.5	0.956083349177562	0.956083349177562\\
1.05	21	0.956268350633352	0.956268350633352\\
1.05	21.5	0.956446907420183	0.956446907420183\\
1.05	22	0.956619646793811	0.956619646793811\\
1.05	22.5	0.956786051495945	0.956786051495945\\
1.05	23	0.956946454268036	0.956946454268036\\
1.05	23.5	0.957100577865996	0.957100577865996\\
1.05	24	0.957248388689251	0.957248388689251\\
1.05	24.5	0.95738995492015	0.95738995492015\\
1.05	25	0.957525221893816	0.957525221893816\\
1.05	25.5	0.957654122585954	0.957654122585954\\
1.05	26	0.957776676182173	0.957776676182173\\
1.05	26.5	0.957892893904529	0.957892893904529\\
1.05	27	0.958002698257706	0.958002698257706\\
1.05	27.5	0.958069499177209	0.958069499177209\\
1.05	28	0.958083753958352	0.958083753958352\\
1.05	28.5	0.958092533663262	0.958092533663262\\
1.05	29	0.95809628181119	0.95809628181119\\
1.05	29.5	0.95809429492841	0.95809429492841\\
1.05	30	0.958087047738116	0.958087047738116\\
1.1	20.5	0.956438531806606	0.956438531806606\\
1.1	21	0.956619649960271	0.956619649960271\\
1.1	21.5	0.956793859733293	0.956793859733293\\
1.1	22	0.956961402812059	0.956961402812059\\
1.1	22.5	0.95712208030578	0.95712208030578\\
1.1	23	0.95727584168929	0.95727584168929\\
1.1	23.5	0.957422734432315	0.957422734432315\\
1.1	24	0.957562695264698	0.957562695264698\\
1.1	24.5	0.957695691198023	0.957695691198023\\
1.1	25	0.957821720997402	0.957821720997402\\
1.1	25.5	0.957940742088837	0.957940742088837\\
1.1	26	0.958052904967031	0.958052904967031\\
1.1	26.5	0.958077661451203	0.958077661451203\\
1.1	27	0.958089527360133	0.958089527360133\\
1.1	27.5	0.958095520334364	0.958095520334364\\
1.1	28	0.958095542438905	0.958095542438905\\
1.1	28.5	0.958089705679173	0.958089705679173\\
1.1	29	0.958077938814572	0.958077938814572\\
1.1	29.5	0.958060212160155	0.958060212160155\\
1.1	30	0.958036505514452	0.958036505514452\\
1.15	20.5	0.956770617917965	0.956770617917965\\
1.15	21	0.956946460454427	0.956946460454427\\
1.15	21.5	0.957114931160629	0.957114931160629\\
1.15	22	0.957275844668644	0.957275844668644\\
1.15	22.5	0.95742925021966	0.95742925021966\\
1.15	23	0.957575079096367	0.957575079096367\\
1.15	23.5	0.957713291610307	0.957713291610307\\
1.15	24	0.957843886083568	0.957843886083568\\
1.15	24.5	0.957966814780996	0.957966814780996\\
1.15	25	0.958065349006692	0.958065349006692\\
1.15	25.5	0.958082075368571	0.958082075368571\\
1.15	26	0.958092393090064	0.958092393090064\\
1.15	26.5	0.95809627324493	0.95809627324493\\
1.15	27	0.958093620357045	0.958093620357045\\
1.15	27.5	0.958084567206383	0.958084567206383\\
1.15	28	0.958069011162265	0.958069011162265\\
1.15	28.5	0.958046929054015	0.958046929054015\\
1.15	29	0.958018058811974	0.958018058811974\\
1.15	29.5	0.957982889468494	0.957982889468494\\
1.15	30	0.957941126945984	0.957941126945984\\
1.2	20.5	0.957078959396997	0.957078959396997\\
1.2	21	0.957248397656646	0.957248397656646\\
1.2	21.5	0.957409672505166	0.957409672505166\\
1.2	22	0.957562701021681	0.957562701021681\\
1.2	22.5	0.957707443004293	0.957707443004293\\
1.2	23	0.957843888860372	0.957843888860372\\
1.2	23.5	0.957971987923442	0.957971987923442\\
1.2	24	0.95806705859326	0.95806705859326\\
1.2	24.5	0.958083786292208	0.958083786292208\\
1.2	25	0.958093532494977	0.958093532494977\\
1.2	25.5	0.958096245078262	0.958096245078262\\
1.2	26	0.958091885695401	0.958091885695401\\
1.2	26.5	0.958080523530357	0.958080523530357\\
1.2	27	0.9580620730647	0.9580620730647\\
1.2	27.5	0.958036508082394	0.958036508082394\\
1.2	28	0.958003577752816	0.958003577752816\\
1.2	28.5	0.957963741312567	0.957963741312567\\
1.2	29	0.957916846867622	0.957916846867622\\
1.2	29.5	0.957862565598567	0.957862565598567\\
1.2	30	0.957801042192007	0.957801042192007\\
1.25	20.5	0.957363487144478	0.957363487144478\\
1.25	21	0.957525233444301	0.957525233444301\\
1.25	21.5	0.957677972301483	0.957677972301483\\
1.25	22	0.957821729352017	0.957821729352017\\
1.25	22.5	0.957956435903258	0.957956435903258\\
1.25	23	0.95806535172124	0.95806535172124\\
1.25	23.5	0.958083229898045	0.958083229898045\\
1.25	24	0.95809353382866	0.95809353382866\\
1.25	24.5	0.958096207899991	0.958096207899991\\
1.25	25	0.958091210270615	0.958091210270615\\
1.25	25.5	0.958078606711968	0.958078606711968\\
1.25	26	0.958058308294375	0.958058308294375\\
1.25	26.5	0.958030285366247	0.958030285366247\\
1.25	27	0.95799429265157	0.95799429265157\\
1.25	27.5	0.957950770268144	0.957950770268144\\
1.25	28	0.957899557508668	0.957899557508668\\
1.25	28.5	0.95784034560938	0.95784034560938\\
1.25	29	0.957773265299027	0.957773265299027\\
1.25	29.5	0.957698288542983	0.957698288542983\\
1.25	30	0.957615387561899	0.957615387561899\\
1.3	20.5	0.957624014787445	0.957624014787445\\
1.3	21	0.957776712687692	0.957776712687692\\
1.3	21.5	0.957919635431787	0.957919635431787\\
1.3	22	0.958052915629882	0.958052915629882\\
1.3	22.5	0.958080261236123	0.958080261236123\\
1.3	23	0.958092397096252	0.958092397096252\\
1.3	23.5	0.958096288102677	0.958096288102677\\
1.3	24	0.958091888319783	0.958091888319783\\
1.3	24.5	0.958079259696631	0.958079259696631\\
1.3	25	0.958058309585857	0.958058309585857\\
1.3	25.5	0.958029004651156	0.958029004651156\\
1.3	26	0.957991098893893	0.957991098893893\\
1.3	26.5	0.957945024563431	0.957945024563431\\
1.3	27	0.957890612449666	0.957890612449666\\
1.3	27.5	0.957827561168387	0.957827561168387\\
1.3	28	0.957755993325653	0.957755993325653\\
1.3	28.5	0.957675877443136	0.957675877443136\\
1.3	29	0.957587173415871	0.957587173415871\\
1.3	29.5	0.957489876961417	0.957489876961417\\
1.3	30	0.957383930835697	0.957383930835697\\
1.35	20.5	0.957860366743343	0.957860366743343\\
1.35	21	0.958016658491468	0.958016658491468\\
1.35	21.5	0.958074152962774	0.958074152962774\\
1.35	22	0.958089534074585	0.958089534074585\\
1.35	22.5	0.958096066473774	0.958096066473774\\
1.35	23	0.958093625615857	0.958093625615857\\
1.35	23.5	0.958082334804112	0.958082334804112\\
1.35	24	0.958062076943734	0.958062076943734\\
1.35	24.5	0.958032814709704	0.958032814709704\\
1.35	25	0.957994295205556	0.957994295205556\\
1.35	25.5	0.957946953987298	0.957946953987298\\
1.35	26	0.957890613704589	0.957890613704589\\
1.35	26.5	0.957824967482309	0.957824967482309\\
1.35	27	0.9577501356843	0.9577501356843\\
1.35	27.5	0.957666083050935	0.957666083050935\\
1.35	28	0.95757276603624	0.95757276603624\\
1.35	28.5	0.95747017592987	0.95747017592987\\
1.35	29	0.957357910731303	0.957357910731303\\
1.35	29.5	0.957237162622531	0.957237162622531\\
1.35	30	0.957106541728119	0.957106541728119\\
1.4	20.5	0.958063599282224	0.958063599282224\\
1.4	21	0.958083791672401	0.958083791672401\\
1.4	21.5	0.958094478822193	0.958094478822193\\
1.4	22	0.958095550368229	0.958095550368229\\
1.4	22.5	0.958087094184181	0.958087094184181\\
1.4	23	0.958069017642997	0.958069017642997\\
1.4	23.5	0.95804126923075	0.95804126923075\\
1.4	24	0.958003582869895	0.958003582869895\\
1.4	24.5	0.957956408580141	0.957956408580141\\
1.4	25	0.957899561277316	0.957899561277316\\
1.4	25.5	0.957832715715093	0.957832715715093\\
1.4	26	0.957755995803249	0.957755995803249\\
1.4	26.5	0.957669362185516	0.957669362185516\\
1.4	27	0.95757276725899	0.95757276725899\\
1.4	27.5	0.957466198329469	0.957466198329469\\
1.4	28	0.957349251602991	0.957349251602991\\
1.4	28.5	0.957223114838522	0.957223114838522\\
1.4	29	0.957086117019669	0.957086117019669\\
1.4	29.5	0.956938821098442	0.956938821098442\\
1.4	30	0.956782359540831	0.956782359540831\\
1.45	20.5	0.958089937232262	0.958089937232262\\
1.45	21	0.958096246481206	0.958096246481206\\
1.45	21.5	0.958092212234783	0.958092212234783\\
1.45	22	0.958077947920303	0.958077947920303\\
1.45	22.5	0.958053327778661	0.958053327778661\\
1.45	23	0.958018066510789	0.958018066510789\\
1.45	23.5	0.957972639065436	0.957972639065436\\
1.45	24	0.957916853161749	0.957916853161749\\
1.45	24.5	0.957850352516711	0.957850352516711\\
1.45	25	0.957773270261412	0.957773270261412\\
1.45	25.5	0.95768556253712	0.95768556253712\\
1.45	26	0.957587177087442	0.957587177087442\\
1.45	26.5	0.95747809762185	0.95747809762185\\
1.45	27	0.957357913149435	0.957357913149435\\
1.45	27.5	0.957227811921971	0.957227811921971\\
1.45	28	0.957086118213757	0.957086118213757\\
1.45	28.5	0.95693401121833	0.95693401121833\\
1.45	29	0.956770135068595	0.956770135068595\\
1.45	29.5	0.956595899784546	0.956595899784546\\
1.45	30	0.956410538486939	0.956410538486939\\
1.5	20.5	0.958095788241064	0.958095788241064\\
1.5	21	0.958087096798823	0.958087096798823\\
1.5	21.5	0.958067359446396	0.958067359446396\\
1.5	22	0.958036515786545	0.958036515786545\\
1.5	22.5	0.957994299037583	0.957994299037583\\
1.5	23	0.957941135750445	0.957941135750445\\
1.5	23.5	0.957876821407779	0.957876821407779\\
1.5	24	0.957801049647443	0.957801049647443\\
1.5	24.5	0.9577139231226	0.9577139231226\\
1.5	25	0.957615393692269	0.957615393692269\\
1.5	25.5	0.957505413538841	0.957505413538841\\
1.5	26	0.95738393567542	0.95738393567542\\
1.5	26.5	0.957251099273906	0.957251099273906\\
1.5	27	0.957106545316312	0.957106545316312\\
1.5	27.5	0.956949654830998	0.956949654830998\\
1.5	28	0.956781748143357	0.956781748143357\\
1.5	28.5	0.956602721675888	0.956602721675888\\
1.5	29	0.956411080576036	0.956411080576036\\
1.5	29.5	0.956207769474675	0.956207769474675\\
1.5	30	0.955992642764459	0.955992642764459\\
};
\addlegendentry{SSIM}

\end{axis}
\end{tikzpicture}%

%% file: BEST_solution.tex
\definecolor{mycolor1}{rgb}{0.00000,0.44700,0.74100}%
\definecolor{mycolor2}{rgb}{0.85000,0.32500,0.09800}%
\begin{tikzpicture}[scale=0.7]

\begin{axis}[%
width=3in,
height=3in,
at={(0in,0.481in)},
scale only axis,
xmin=0,
xmax=10,
ymin=-0.1,
ymax=2,
axis background/.style={fill=white},
legend style={at={(0.01,0.75)}, anchor=south west, legend cell align=left, align=left, draw=white!15!white,font=\small},
title style={font=\bfseries},
 title={TV regularized problem}]
]
\addplot [color=blue]
  table[row sep=crcr]{%
.196078431372549	0.0321463660032439\\
0.392156862745098	0.0517960424650268\\
0.588235294117647	0.0522026614381039\\
0.784313725490196	0.052203122894656\\
0.980392156862745	0.0522026830959553\\
1.17647058823529	0.0522026903659775\\
1.37254901960784	0.0522020501033978\\
1.56862745098039	0.0522026534169827\\
1.76470588235294	0.0522022730996443\\
1.96078431372549	0.0463569267535216\\
2.15686274509804	0.00271484936557364\\
2.35294117647059	-0.0154232002524197\\
2.54901960784314	-0.015423788227555\\
2.74509803921569	-0.0154231441003094\\
2.94117647058824	0.222196888367394\\
3.13725490196078	0.222197898135283\\
3.33333333333333	0.222198137149364\\
3.52941176470588	0.563039344160285\\
3.72549019607843	0.56303940934267\\
3.92156862745098	0.563040112733704\\
4.11764705882353	0.563041075595357\\
4.31372549019608	0.562311378857078\\
4.50980392156863	0.564138096950054\\
4.70588235294118	0.564171004887146\\
4.90196078431373	0.832795362381735\\
};
\addplot [color=red,mark=diamond]
  table[row sep=crcr]{%
0.196078431372549	0\\
0.392156862745098	0\\
0.588235294117647	0\\
0.784313725490196	0\\
0.980392156862745	0\\
1.17647058823529	0\\
1.37254901960784	0\\
1.56862745098039	0\\
1.76470588235294	0\\
1.96078431372549	0\\
2.15686274509804	0\\
2.35294117647059	0\\
2.54901960784314	0.0196078431372551\\
2.74509803921569	0.0980392156862746\\
2.94117647058824	0.176470588235294\\
3.13725490196078	0.254901960784314\\
3.33333333333333	0.333333333333333\\
3.52941176470588	0.411764705882353\\
3.72549019607843	0.490196078431373\\
3.92156862745098	0.568627450980392\\
4.11764705882353	0.647058823529412\\
4.31372549019608	0.725490196078431\\
4.50980392156863	0.803921568627451\\
4.70588235294118	0.882352941176471\\
4.90196078431373	0.96078431372549\\
};
\addplot [color=blue]
  table[row sep=crcr]{%
5.09803921568628	1.50330615149446\\
5.29411764705882	1.50330720654404\\
5.49019607843137	1.50330804830467\\
5.68627450980392	1.50330849715953\\
5.88235294117647	1.26195069351949\\
6.07843137254902	1.26174891801921\\
6.27450980392157	1.50783110976155\\
6.47058823529412	1.50780028169074\\
6.66666666666667	1.2148306474393\\
6.86274509803922	0.705490135498072\\
7.05882352941177	0.475347952463893\\
7.25490196078431	0.203631711946308\\
7.45098039215686	-0.0102171574555555\\
7.64705882352941	-0.0102166691222573\\
7.84313725490196	-0.0102176208008309\\
8.03921568627451	-0.0102173724433222\\
8.23529411764706	-0.0102177362123882\\
8.43137254901961	-0.0102182914795476\\
8.62745098039216	-0.0102174874839479\\
8.82352941176471	-0.0102164349233521\\
9.01960784313725	-0.0101908171810018\\
9.2156862745098	-0.0101914838554064\\
9.41176470588235	-0.01019175668178\\
9.6078431372549	-0.0973274856990584\\
9.80392156862745	-0.234228695214752\\
};
\addlegendentry{TV Solution}

\addplot [color=red,mark=diamond]
  table[row sep=crcr]{%
5.09803921568628	1.92156862745098\\
5.29411764705882	1.76470588235294\\
5.49019607843137	1.6078431372549\\
5.68627450980392	1.45098039215686\\
5.88235294117647	1.29411764705882\\
6.07843137254902	1.13725490196078\\
6.27450980392157	0.980392156862744\\
6.47058823529412	0.823529411764706\\
6.66666666666667	0.666666666666666\\
6.86274509803922	0.509803921568627\\
7.05882352941177	0.352941176470588\\
7.25490196078431	0.196078431372548\\
7.45098039215686	0.0392156862745097\\
7.64705882352941	0\\
7.84313725490196	0\\
8.03921568627451	0\\
8.23529411764706	0\\
8.43137254901961	0\\
8.62745098039216	0\\
8.82352941176471	0\\
9.01960784313725	0\\
9.2156862745098	0\\
9.41176470588235	0\\
9.6078431372549	0\\
9.80392156862745	0\\
};
\addlegendentry{Exact Solution}
\end{axis}
\begin{axis}[%
width=3in,
height=3in,
at={(3.758in,0.481in)},
scale only axis,
xmin=0,
xmax=10,
ymin=-0.01,
ymax=2,
axis background/.style={fill=white},
legend style={at={(0.01,0.75)}, anchor=south west, legend cell align=left, align=left, draw=white!15!white, font=\small},
title style={font=\bfseries},
 title={TGV regularized problem}
]
\addplot [color=blue]
  table[row sep=crcr]{%
0.196078431372549	0.254256430929368\\
0.392156862745098	0.230587639445857\\
0.588235294117647	0.20691934920742\\
0.784313725490196	0.183251433713883\\
0.980392156862745	0.159584425609957\\
1.17647058823529	0.135919397851319\\
1.37254901960784	0.112256040940421\\
1.56862745098039	0.0885952691369372\\
1.76470588235294	0.0649356966801421\\
1.96078431372549	0.0412769707931848\\
2.15686274509804	0.0176203262022144\\
2.35294117647059	-0.00603360536521185\\
2.54901960784314	-0.029683332596686\\
2.74509803921569	0.0388393518023909\\
2.94117647058824	0.107365962072075\\
3.13725490196078	0.175894086297837\\
3.33333333333333	0.244423368939108\\
3.52941176470588	0.312954816101907\\
3.72549019607843	0.381485691270206\\
3.92156862745098	0.450017659728674\\
4.11764705882353	0.518551431156735\\
4.31372549019608	0.587087823094768\\
4.50980392156863	0.733749606432402\\
4.70588235294118	0.880414761157562\\
4.90196078431373	1.0270827519883\\
};
\addplot [color=red,mark=diamond]
  table[row sep=crcr]{%
0.196078431372549	0\\
0.392156862745098	0\\
0.588235294117647	0\\
0.784313725490196	0\\
0.980392156862745	0\\
1.17647058823529	0\\
1.37254901960784	0\\
1.56862745098039	0\\
1.76470588235294	0\\
1.96078431372549	0\\
2.15686274509804	0\\
2.35294117647059	0\\
2.54901960784314	0.0196078431372551\\
2.74509803921569	0.0980392156862746\\
2.94117647058824	0.176470588235294\\
3.13725490196078	0.254901960784314\\
3.33333333333333	0.333333333333333\\
3.52941176470588	0.411764705882353\\
3.72549019607843	0.490196078431373\\
3.92156862745098	0.568627450980392\\
4.11764705882353	0.647058823529412\\
4.31372549019608	0.725490196078431\\
4.50980392156863	0.803921568627451\\
4.70588235294118	0.882352941176471\\
4.90196078431373	0.96078431372549\\
};
\addplot [color=blue]
  table[row sep=crcr]{%
5.09803921568628	1.17375136050075\\
5.29411764705882	1.32041705343865\\
5.49019607843137	1.46707925289619\\
5.68627450980392	1.61373762183228\\
5.88235294117647	1.44332395440831\\
6.07843137254902	1.27290770470029\\
6.27450980392157	1.10248920127764\\
6.47058823529412	0.932065863339458\\
6.66666666666667	0.761639112395554\\
6.86274509803922	0.591210636224297\\
7.05882352941177	0.42078231483429\\
7.25490196078431	0.250355377684177\\
7.45098039215686	0.079931206942977\\
7.64705882352941	0.0161544813133475\\
7.84313725490196	-0.0476192304786482\\
8.03921568627451	-0.0522560192361415\\
8.23529411764706	-0.0568889865559912\\
8.43137254901961	-0.0615189311178032\\
8.62745098039216	-0.0661455562269677\\
8.82352941176471	-0.0707701307525338\\
9.01960784313725	-0.0753935237634316\\
9.2156862745098	-0.0800173732210846\\
9.41176470588235	-0.0846404701988046\\
9.6078431372549	-0.0892632509701809\\
9.80392156862745	-0.0938835593592331\\
};
\addlegendentry{TGV Solution}

\addplot [color=red,mark=diamond]
  table[row sep=crcr]{%
5.09803921568628	1.92156862745098\\
5.29411764705882	1.76470588235294\\
5.49019607843137	1.6078431372549\\
5.68627450980392	1.45098039215686\\
5.88235294117647	1.29411764705882\\
6.07843137254902	1.13725490196078\\
6.27450980392157	0.980392156862744\\
6.47058823529412	0.823529411764706\\
6.66666666666667	0.666666666666666\\
6.86274509803922	0.509803921568627\\
7.05882352941177	0.352941176470588\\
7.25490196078431	0.196078431372548\\
7.45098039215686	0.0392156862745097\\
7.64705882352941	0\\
7.84313725490196	0\\
8.03921568627451	0\\
8.23529411764706	0\\
8.43137254901961	0\\
8.62745098039216	0\\
8.82352941176471	0\\
9.01960784313725	0\\
9.2156862745098	0\\
9.41176470588235	0\\
9.6078431372549	0\\
9.80392156862745	0\\
};
\addlegendentry{Exact solution}

\end{axis}
\end{tikzpicture}%

%% file: state_vs_observations.tex
\pgfplotsset{compat=1.6}
\definecolor{mycolor1}{rgb}{0.00000,0.44700,0.74100}%
\pgfplotsset{soldot/.style={color=red,only marks,mark=*}}

\begin{tikzpicture}[scale=0.75]

\begin{axis}[%
width=2in,
height=2in,
at={(0in,0.771in)},
scale only axis,
xmin=0,
xmax=10,
xlabel style={font=\color{white!15!black}},
xlabel={$\Omega$},
ymin=-0.2,
ymax=1.8,
axis background/.style={fill=white},
title style={font=\bfseries},
title={t=0.16},
legend style={at={(0.01,0.7)}, anchor=south west, legend cell align=left, align=left, draw=white!15!white, font=\tiny}
]
\addplot[soldot] 
 table[row sep=crcr]{%
%    0.1961         0\\
    2.1569         0\\
    4.1176    0.6066\\
    6.0784    1.3078\\
    8.0392         0\\
};
\addplot [color=mycolor1]
  table[row sep=crcr]{%
0.196078431372549	0.209072177187739\\
0.392156862745098	0.230228936525204\\
0.588235294117647	0.210841601305604\\
0.784313725490196	0.186999323508575\\
0.980392156862745	0.162860237476178\\
1.17647058823529	0.138709584847728\\
1.37254901960784	0.114560299891838\\
1.56862745098039	0.0904136472324811\\
1.76470588235294	0.0662683991791708\\
1.96078431372549	0.0421240835161565\\
2.15686274509804	0.0179819096876187\\
2.35294117647059	-0.00615122892736066\\
2.54901960784314	-0.0280496035829286\\
2.74509803921569	0.0367016883271589\\
2.94117647058824	0.101456375332061\\
3.13725490196078	0.166212405335463\\
3.33333333333333	0.230969441178055\\
3.52941176470588	0.295728191474272\\
3.72549019607843	0.36048686727797\\
3.92156862745098	0.425246199825677\\
4.11764705882353	0.490006850595165\\
4.31372549019608	0.554769466464199\\
4.50980392156863	0.660670943080903\\
4.70588235294118	0.784635535374948\\
4.90196078431373	0.913600837212687\\
5.09803921568628	1.04371402525428\\
5.29411764705882	1.17406545865576\\
5.49019607843137	1.304462263746\\
5.68627450980392	1.43486576891284\\
5.88235294117647	1.48289702091551\\
6.07843137254902	1.41788773104681\\
6.27450980392157	1.27106091145765\\
6.47058823529412	1.08650200884411\\
6.66666666666667	0.890185448440238\\
6.86274509803922	0.691317953484788\\
7.05882352941177	0.49206186654087\\
7.25490196078431	0.29276632893473\\
7.45098039215686	0.0934715759437987\\
7.64705882352941	0.0171432065696621\\
7.84313725490196	-0.0478076497670634\\
8.03921568627451	-0.0524626147697968\\
8.23529411764706	-0.0571137506645621\\
8.43137254901961	-0.061761814434364\\
8.62745098039216	-0.0664065910921939\\
8.82352941176471	-0.0710493472883455\\
9.01960784313725	-0.0756910062042069\\
9.2156862745098	-0.0803329094549188\\
9.41176470588235	-0.084966900747103\\
9.6078431372549	-0.0893274850181374\\
9.80392156862745	-0.0869452400451674\\
};
\addlegendentry{Observations}
\addlegendentry{Optimal state}

\end{axis}

\begin{axis}[%
width=2in,
height=2in,
at={(2.5in,0.771in)},
scale only axis,
xmin=0,
xmax=10,
xlabel style={font=\color{white!15!black}},
xlabel={$\Omega$},
ymin=-0.2,
ymax=1.8,
axis background/.style={fill=white},
title style={font=\bfseries},
title={t=0.3},
legend style={at={(0.01,0.7)}, anchor=south west, legend cell align=left, align=left, draw=white!15!white, font=\tiny}
]
\addplot[soldot] 
table[row sep=crcr]{%
%	0.1961         0\\
    2.1569         0\\
    4.1176    0.5709\\
    6.0784    1.4560\\
    8.0392         0\\
   };
\addplot [color=mycolor1]
  table[row sep=crcr]{%
0.196078431372549	0.177524125142824\\
0.392156862745098	0.22342752927215\\
0.588235294117647	0.213492831035088\\
0.784313725490196	0.190754497889114\\
0.980392156862745	0.166263458417803\\
1.17647058823529	0.141616306080575\\
1.37254901960784	0.116961146449132\\
1.56862745098039	0.0923082525212149\\
1.76470588235294	0.067656957489887\\
1.96078431372549	0.0430066967235785\\
2.15686274509804	0.0183586452810539\\
2.35294117647059	-0.0062623260439875\\
2.54901960784314	-0.0265863292895785\\
2.74509803921569	0.0347870574521484\\
2.94117647058824	0.0961634032105985\\
3.13725490196078	0.157540945807745\\
3.33333333333333	0.218919370237862\\
3.52941176470588	0.280299179285187\\
3.72549019607843	0.341679181053848\\
3.92156862745098	0.403059638841272\\
4.11764705882353	0.464441058861332\\
4.31372549019608	0.525823964016043\\
4.50980392156863	0.610065162301083\\
4.70588235294118	0.712900650917351\\
4.90196078431373	0.824838160826502\\
5.09803921568628	0.940368114944192\\
5.29411764705882	1.05716279437852\\
5.49019607843137	1.17437222825406\\
5.68627450980392	1.29171077254393\\
5.88235294117647	1.38665145459797\\
6.07843137254902	1.42429501203527\\
6.27450980392157	1.37842700359808\\
6.47058823529412	1.24741148900378\\
6.66666666666667	1.05413686153641\\
6.86274509803922	0.828752650219049\\
7.05882352941177	0.591953475060202\\
7.25490196078431	0.352443770095666\\
7.45098039215686	0.112534430783238\\
7.64705882352941	0.0184259708370086\\
7.84313725490196	-0.0479975590559808\\
8.03921568627451	-0.0526708436430029\\
8.23529411764706	-0.0573402894309202\\
8.43137254901961	-0.0620066171821549\\
8.62745098039216	-0.0666696906263299\\
8.82352941176471	-0.0713307761557221\\
9.01960784313725	-0.0759907934631254\\
9.2156862745098	-0.0806493217417366\\
9.41176470588235	-0.0852582507866007\\
9.6078431372549	-0.0889193424142085\\
9.80392156862745	-0.0809618776419906\\
};
\addlegendentry{Observations}
\addlegendentry{Optimal state}
\end{axis}

\begin{axis}[%
width=2in,
height=2in,
at={(5in,0.771in)},
scale only axis,
xmin=0,
xmax=10,
xlabel style={font=\color{white!15!black}},
xlabel={$\Omega$},
ymin=-0.2,
ymax=1.7,
axis background/.style={fill=white},
title style={font=\bfseries},
title={t= 0.5},
legend style={at={(0.01,0.7)}, anchor=south west, legend cell align=left, align=left, draw=white!15!white, font=\tiny}
]
\addplot[soldot] 
table[row sep=crcr]{%
% 	0.1961         0\\
    2.1569         0\\
    4.1176    0.5392\\
    6.0784    1.4783\\
    8.0392         0\\
   };
\addplot [color=mycolor1]
  table[row sep=crcr]{%
0.196078431372549	0.154248690832525\\
0.392156862745098	0.213478230578008\\
0.588235294117647	0.214283926792817\\
0.784313725490196	0.194263926346919\\
0.980392156862745	0.169765503404633\\
1.17647058823529	0.144643999785804\\
1.37254901960784	0.11946462247144\\
1.56862745098039	0.0942839756055751\\
1.76470588235294	0.0691049591407052\\
1.96078431372549	0.0439270900124266\\
2.15686274509804	0.018751505826078\\
2.35294117647059	-0.00636734515466487\\
2.54901960784314	-0.0252681558343476\\
2.74509803921569	0.0330622847232938\\
2.94117647058824	0.091395322321582\\
3.13725490196078	0.149729431182116\\
3.33333333333333	0.20806431954806\\
3.52941176470588	0.266400345609868\\
3.72549019607843	0.32473670123507\\
3.92156862745098	0.383073420728189\\
4.11764705882353	0.441410855454344\\
4.31372549019608	0.499749421204557\\
4.50980392156863	0.571302882646853\\
4.70588235294118	0.658039693748156\\
4.90196078431373	0.75497959578543\\
5.09803921568628	0.857335263506883\\
5.29411764705882	0.962227780076291\\
5.49019607843137	1.06822210892322\\
5.68627450980392	1.17467121743865\\
5.88235294117647	1.2745619111107\\
6.07843137254902	1.35135562438436\\
6.27450980392157	1.3813279153104\\
6.47058823529412	1.33972591260107\\
6.66666666666667	1.21117519492721\\
6.86274509803922	1.0008000428815\\
7.05882352941177	0.73425035572149\\
7.25490196078431	0.441552889786501\\
7.45098039215686	0.141321587407469\\
7.64705882352941	0.0201793457134237\\
7.84313725490196	-0.0481889760714117\\
8.03921568627451	-0.0528807252013638\\
8.23529411764706	-0.0575686240409989\\
8.43137254901961	-0.0622533623001695\\
8.62745098039216	-0.066934879451501\\
8.82352941176471	-0.0716144354353461\\
9.01960784313725	-0.07629271057959\\
9.2156862745098	-0.0809632261314262\\
9.41176470588235	-0.0854873992814035\\
9.6078431372549	-0.0881405210012945\\
9.80392156862745	-0.0757490126896423\\
};
\addlegendentry{Observations}
\addlegendentry{Optimal state}

\end{axis}
\end{tikzpicture}%

%% file: descent_objective.tex
\definecolor{mycolor1}{rgb}{0.00000,0.44700,0.74100}%
\definecolor{mycolor2}{rgb}{0.85000,0.32500,0.09800}%
\definecolor{mycolor3}{rgb}{0.92900,0.69400,0.12500}%
\definecolor{mycolor4}{rgb}{0.49400,0.18400,0.55600}%
\definecolor{mycolor5}{rgb}{0.46600,0.67400,0.18800}%
\begin{tikzpicture}

\begin{axis}[%
xlabel={Iterations},
width=1.5in,
height=1.5in,
at={(0in,0.822in)},
scale only axis,
xmin=4,
xmax=15,
ymin=20,
ymax=60,
axis background/.style={fill=white},
legend style={legend cell align=left, align=left, draw=white!15!white, font=\scriptsize},
title style={font=\bfseries},
title={Experiment 1}]
\addplot [color=mycolor1]
  table[row sep=crcr]{%
1	599.001825136983\\
2	290.299336265704\\
3	268.436296388154\\
4	55.0563125351357\\
5	35.3393727503547\\
6	34.0641434191888\\
7	27.6063426448655\\
8	26.2939417498091\\
9	25.8214580515111\\
10	25.5721678779217\\
11	25.4883879872968\\
12	25.3872254409836\\
13	25.3826056036218\\
};
\addlegendentry{$u^0\equiv 1$}

\addplot [color=mycolor2,mark=x]
  table[row sep=crcr]{%
1	601.601624875417\\
2	290.629070724788\\
3	268.697524469\\
4	55.0961437466675\\
5	34.1263151385787\\
6	33.3241611437591\\
7	28.1357048453114\\
8	26.4475161789633\\
9	26.716315341161\\
10	26.0032862862362\\
11	25.502873429737\\
12	25.3877673367025\\
13	25.3826160466234\\
};
\addlegendentry{$u^0\equiv 0.5$}

\addplot [color=mycolor3,mark=square]
  table[row sep=crcr]{%
1	986.965383076948\\
2	290.079331251316\\
3	271.073321311696\\
4	56.2867524566686\\
5	41.6872726805469\\
6	40.7081539624774\\
7	34.7548214796279\\
8	29.0401666225502\\
9	26.4039176890667\\
10	25.5097567719952\\
11	25.3965763867444\\
12	25.3828594054286\\
};
\addlegendentry{$u^0\equiv 2$}

\addplot [color=mycolor4,mark=*]
  table[row sep=crcr]{%
1	1181.11912700718\\
2	1691.3913526722\\
3	938.598782240268\\
4	72.1615890951217\\
5	40.4814534273788\\
6	30.28798531505\\
7	26.6104087970643\\
8	26.1643396065931\\
9	26.4278139056936\\
10	26.0378959322996\\
11	25.5033577889716\\
12	25.3877174824979\\
13	25.3826158160677\\
};
\addlegendentry{$u^0$ random}

\addplot [color=mycolor5,mark=diamond]
  table[row sep=crcr]{%
1	571.453641146676\\
2	389.918183170552\\
3	305.500920109603\\
4	58.0972558183096\\
5	36.5521533474039\\
6	29.1518000076786\\
7	27.9047801193893\\
8	26.3837532667287\\
9	27.447367252854\\
10	25.8874835193121\\
11	25.4959686913139\\
12	25.387252699631\\
13	25.3826159388096\\
};
\addlegendentry{$u^0$ TV sol.}

\end{axis}
\begin{axis}[%
xlabel={Iterations},
width=1.5in,
height=1.5in,
at={(2in,0.822in)},
scale only axis,
xmin=4,
xmax=14,
ymin=20,
ymax=60,
axis background/.style={fill=white},
legend style={legend cell align=left, align=left, draw=white!15!white, font=\scriptsize},
title style={font=\bfseries},
title={Experiment 2}]
\addplot [color=mycolor1]
  table[row sep=crcr]{%
1	690.973232229496\\
2	295.912295443652\\
3	269.109820591289\\
4	54.5603918759692\\
5	33.1401984767617\\
6	33.0868990047518\\
7	32.1832302339785\\
8	27.6760858741863\\
9	27.6098293284563\\
10	26.5480099302552\\
11	26.3471413645299\\
12	26.3232940453635\\
};
\addlegendentry{$u^0\equiv 1$}

\addplot [color=mycolor2,mark=x]
  table[row sep=crcr]{%
1	592.585328486896\\
2	296.112108301666\\
3	268.355252300249\\
4	54.5528088279234\\
5	33.1450471008864\\
6	33.0867524094159\\
7	31.9000072807735\\
8	27.3760840459489\\
9	27.233756451214\\
10	26.4929377886885\\
11	26.3288776955298\\
12	26.3234307855411\\
13	26.3226704051103\\
};
\addlegendentry{$u^0\equiv 0.5$}

\addplot [color=mycolor3,mark=square]
  table[row sep=crcr]{%
1	1280.77669531549\\
2	295.91222199074\\
3	278.715216502577\\
4	54.2282912709529\\
5	33.1593095117523\\
6	31.8136420317424\\
7	29.2796029740465\\
8	27.1825773742598\\
9	27.2001277582859\\
10	26.5316027281105\\
11	26.3485645026717\\
12	26.3234475855123\\
};
\addlegendentry{$u^0\equiv 2$}

\addplot [color=mycolor4,mark=*]
  table[row sep=crcr]{%
1	1152.93432575295\\
2	1668.33563343729\\
3	924.051964944719\\
4	70.6932105659687\\
5	45.412072686993\\
6	31.5367725510287\\
7	29.910186790793\\
8	27.0277582513395\\
9	26.3475956672324\\
10	26.3228465698942\\
};
\addlegendentry{$u^0$ random}

\addplot [color=mycolor5,mark=diamond]
  table[row sep=crcr]{%
1	573.139251486831\\
2	543.865523870716\\
3	345.080389595103\\
4	60.9894620136634\\
5	33.1090556229721\\
6	32.2512201372898\\
7	31.2511561403839\\
8	28.1097747777558\\
9	27.2979464519839\\
10	26.6206665500061\\
11	26.3555433175205\\
12	26.3238895976919\\
};
\addlegendentry{$u^0$ TV sol.}
\end{axis}
 
\begin{axis}[%
xlabel={Iterations},
width=1.5in,
height=1.5in,
at={(4in,0.822in)},
scale only axis,
xmin=7,
xmax=18,
ymin=45,
ymax=60,
axis background/.style={fill=white},
legend style={legend cell align=left, align=left, draw=white!15!white, font=\scriptsize},
title style={font=\bfseries},
title={Experiment 3},]
\addplot [color=mycolor1]
  table[row sep=crcr]{%
7	52.0316961910515\\
8	47.2831492939565\\
9	45.8583773937424\\
10	45.8591640533777\\
11	45.855069502879\\
12	45.8568617390389\\
13	45.8593107798821\\
14	45.8612909801481\\
15	45.8627728711408\\
16	45.8638711115312\\
17	45.8646005462306\\
};
\addlegendentry{$u^0\equiv 1$}
\addplot [color=mycolor2,mark=x]
  table[row sep=crcr]{%
7	57.7052547561828\\
8	48.694256852619\\
9	48.6155232647761\\
10	46.366169760463\\
11	45.9693848122852\\
12	45.8679546412013\\
};
\addlegendentry{$u^0\equiv 0.5$}

\addplot [color=mycolor3,mark=square]
  table[row sep=crcr]{%
7	52.3971957536433\\
8	47.289349807016\\
9	45.8587458558373\\
10	45.8593711032709\\
11	45.8552253408367\\
12	45.8573060944896\\
13	45.8598813795576\\
14	45.8618331920093\\
15	45.8632217374232\\
16	45.8642027355914\\
17	45.8648152312689\\
};
\addlegendentry{$u^0\equiv 2$}

\addplot [color=mycolor4,mark=*]
  table[row sep=crcr]{%
7	58.6360990318433\\
8	49.2321431920376\\
9	48.2645966314958\\
10	47.6685869838215\\
11	46.9262025224347\\
12	46.0799761348514\\
13	45.8466228477207\\
14	45.8480310521091\\
15	45.8494251554981\\
16	45.8507825239347\\
17	45.852087132515\\
%18	45.8533281328092\\
%19	45.8544986980644\\
%20	45.8555950957437\\
%21	45.8566159455527\\
%22	45.8575616283845\\
%23	45.8584338177162\\
%24	45.8592351100926\\
%25	45.859968735566\\
%26	45.86063833247\\
%27	45.8612477737935\\
};
\addlegendentry{$u^0$ random}

\addplot [color=mycolor5,mark=diamond]
  table[row sep=crcr]{%
1	435.948444519896\\
2	549.106072254828\\
3	259.07108360499\\
4	69.1219313740891\\
5	52.0174579673466\\
6	47.1198948060668\\
7	46.3801728896117\\
8	45.9689714475658\\
9	45.9118658075313\\
10	45.8814457353539\\
11	45.8656735457916\\
12	45.8656158384082\\
13	45.8656121506537\\
14	45.8656350368906\\
};
\addlegendentry{$u^0$ TV sol.}
\end{axis}
\end{tikzpicture}%

%% file: superlineal_convergence.tex
\definecolor{mycolor1}{rgb}{0.00000,0.44700,0.74100}%
\definecolor{mycolor2}{rgb}{0.85000,0.32500,0.09800}%
\definecolor{mycolor3}{rgb}{0.46600,0.67400,0.18800}%
\begin{tikzpicture}[scale=0.6]

\begin{axis}[%
xlabel=Iterations,
yticklabel style = {font=\small},
width=2.5in,
height=2.5in,
at={(0in,0in)},
scale only axis,
xmin=0,
xmax=10,
ymode=log,
ymin=0.01,
ymax=1,
yminorticks=true,
axis background/.style={fill=white},
legend style={at={(0.05,0.01)}, anchor=south west, legend cell align=left, align=left, draw=white!15!white,font=\small},
title style={font=\bfseries},
title={Experiment 1}]
\addplot [color=mycolor1]
 table[row sep=crcr]{%
1	0.174501173221827\\
2	0.931649876416842\\
3	0.47455458922573\\
4	0.491782060091696\\
5	0.779512269883816\\
6	0.797760267148119\\
7	0.91282788156717\\
8	0.490094160498212\\
9	0.120705825554602\\
10	0.0437348175294836\\
11	0\\
};
\addlegendentry{$\alpha=5$, $\beta=0.1$}

\addplot [color=mycolor2,mark=*]
 table[row sep=crcr]{%
1	0.199206421533761\\
2	0.941368109097722\\
3	0.49840187423737\\
4	0.452545500382003\\
5	0.506343016654005\\
6	0.338228462118885\\
7	0.159751414913234\\
8	0.0296077356651341\\
9	0\\
};
\addlegendentry{$\alpha=7.5$, $\beta=0.1125$}

\addplot [color=mycolor3,mark=diamond]
   table[row sep=crcr]{%
1	0.230564496545264\\
2	0.948915738034386\\
3	0.573898693200439\\
4	0.652386699822468\\
5	0.698516424066236\\
6	0.708301309307832\\
7	0.373634364360073\\
8	0.368980011757095\\
9	0.272527343853006\\
10	0.0221011596817142\\
11	0\\
};
\addlegendentry{$\alpha=15$, $\beta=0.30$}
\end{axis}
\begin{axis}[%
xlabel=Iterations,
yticklabel style = {font=\small},
width=2.5in,
height=2.5in,
at={(3in,0in)},
scale only axis,
xmin=0,
xmax=12,
ymode=log,
ymin=0.01,
ymax=1,
yminorticks=true,
axis background/.style={fill=white},
legend style={at={(0.05,0.01)}, anchor=south west,legend cell align=left, align=left, draw=white!15!white,font=\small},
title style={font=\bfseries},
title={Experiment 2}]
\addplot [color=mycolor1]
   table[row sep=crcr]{%
1	0.203555453210146\\
2	0.804249399159003\\
3	0.546566868269708\\
4	0.680509324459661\\
5	0.74292297971891\\
6	0.746584013207724\\
7	0.640862379216379\\
8	0.917395568702667\\
9	0.585140295592688\\
10	0.321795913150692\\
11	0.224497799670839\\
12	0.0645593749761887\\
13	0\\
};
\addlegendentry{$\alpha=5$, $\beta=0.075$}

\addplot [color=mycolor2,mark=*]
   table[row sep=crcr]{%
1	0.160620069255196\\
2	0.618944761244025\\
3	0.489201421977967\\
4	0.624399755489374\\
5	0.421292665548904\\
6	0.389486288702267\\
7	0.211791036922554\\
8	0.0871899434428511\\
9	0.0306502409059147\\
10	0\\
};
\addlegendentry{$\alpha=7.5$, $\beta=0.1125$}

\addplot [color=mycolor3,mark=diamond]
  table[row sep=crcr]{%
1	0.174981260707609\\
2	0.876673809111089\\
3	0.465709214398852\\
4	0.430998462114148\\
5	0.558887814064777\\
6	0.552859124503897\\
7	0.326074105631621\\
8	0.0465168102973589\\
9	0.0195748413506587\\
10	0\\
};
\addlegendentry{$\alpha=17.5$, $\beta=0.35$}

\end{axis}
\begin{axis}[%
xlabel=Iterations,
yticklabel style = {font=\tiny},
width=2.5in,
height=2.5in,
at={(6in,0in)},
scale only axis,
xmin=0,
xmax=15,
ymode=log,
ymin=0.01,
ymax=1,
yminorticks=true,
axis background/.style={fill=white},
legend style={at={(0.05,0.01)}, anchor=south west, legend cell align=left, align=left, draw=white!15!white,font=\small},
title style={font=\bfseries},
title={Experiment 3}]
\addplot [color=mycolor1]
  table[row sep=crcr]{%
1	0.624175519473809\\
2	0.921145657399381\\
3	0.262016244982637\\
4	0.680951866686514\\
5	0.771616286413009\\
6	0.903682683190747\\
7	0.528763330124249\\
8	0.208349378311591\\
9	0.549254844238288\\
10	0.553690329796059\\
11	0.566118122503186\\
12	0.575091549477429\\
13	0.556457532459463\\
14	0.464380869686441\\
15	0.138630572924629\\
16	0\\
};
\addlegendentry{$\alpha=5$, $\beta=0.15$}

\addplot [color=mycolor2,mark=*]
   table[row sep=crcr]{%
1	0.669589225685003\\
2	0.712211167636499\\
3	0.525555562991837\\
4	0.552599365940038\\
5	0.597740336353764\\
6	0.718965712001551\\
7	0.660819042830584\\
8	0.730808192299391\\
9	0.20123665521268\\
10	0.0413662861638191\\
11	0.0128545202671186\\
12	0\\
};
\addlegendentry{$\alpha=17.5$, $\beta=0.2625$}

\addplot [color=mycolor3,mark=diamond]
   table[row sep=crcr]{%
1	0.651064235044329\\
2	0.738835861796462\\
3	0.538068733761443\\
4	0.555762810326979\\
5	0.749557991667042\\
6	0.906727175978125\\
7	0.316640701876493\\
8	0.14255368647612\\
9	0.0115944210063573\\
10	0\\
};
\addlegendentry{$\alpha=17.5$, $\beta=0.35$}

\end{axis}
\end{tikzpicture}%

%% file: Total_variation_burgers_v3.bbl
\begin{thebibliography}{10}

\bibitem{apte2010variational}
A.~Apte, D.~Auroux, and M.~Ramaswamy.
\newblock Variational data assimilation for discrete burgers equation.
\newblock {\em Electronic Journal of Differential Equations}, 19:15--30, 2010.

\bibitem{bredies2013spatially}
K.~Bredies, Y.~Dong, and M.~Hinterm{\"{u}}ller.
\newblock {Spatially dependent regularization parameter selection in total
  generalized variation models for image restoration}.
\newblock {\em International Journal of Computer Mathematics}, 90(1):109--123,
  2013.

\bibitem{bredies2010total}
K.~Bredies, K.~Kunisch, and T.~Pock.
\newblock {Total generalized variation}.
\newblock {\em SIAM Journal on Imaging Sciences}, 3(3):492--526, 2010.

\bibitem{bredies2011inverse}
K.~Bredies and T.~Valkonen.
\newblock {Inverse problems with second-order total generalized variation
  constraints}.
\newblock {\em Proceedings of SampTA}, 201, 2011.

\bibitem{BuddFreitag2011}
C.J. Budd, M.A. Freitag, and N.K. Nichols.
\newblock {Regularization techniques for ill-posed inverse problems in data
  assimilation}.
\newblock {\em Computers {\&} Fluids}, 46(1):168--173, jul 2011.

\bibitem{chan1999nonlinear}
T.~F Chan, G.~H Golub, and P.~Mulet.
\newblock A nonlinear primal-dual method for total variation-based image
  restoration.
\newblock {\em SIAM journal on scientific computing}, 20(6):1964--1977, 1999.

\bibitem{ciarlet2013linear}
P.G. Ciarlet.
\newblock {\em Linear and nonlinear functional analysis with applications},
  volume 130.
\newblock Siam, 2013.

\bibitem{de2015numerical}
J.C. {De los Reyes}.
\newblock {\em {Numerical PDE-constrained optimization}}.
\newblock Springer, 2015.

\bibitem{de2017bilevel}
J.C. {De los Reyes}, C.B. Sch{\"{o}}nlieb, and T.~Valkonen.
\newblock {Bilevel parameter learning for higher-order total variation
  regularisation models}.
\newblock {\em Journal of Mathematical Imaging and Vision}, 57(1):1--25, 2017.

\bibitem{dennis1996numerical}
John~E Dennis~Jr and Robert~B Schnabel.
\newblock {\em Numerical methods for unconstrained optimization and nonlinear
  equations}.
\newblock SIAM, 1996.

\bibitem{freitag2010resolution}
M.A. Freitag, N.K. Nichols, and C.J. Budd.
\newblock {Resolution of sharp fronts in the presence of model error in
  variational data assimilation}.
\newblock {\em Quarterly Journal of the Royal Meteorological Society},
  139(672):742--757, 2010.

\bibitem{HintermuellerStadler2007}
M.~Hinterm{\"{u}}ller and G.~Stadler.
\newblock {An infeasible primal-dual algorithm for total bounded
  variation-based inf-convolution-type image restoration}.
\newblock {\em SIAM J. Sci. Comput.}, 28(1):1----23 (electronic), 2006.

\bibitem{hinze2008optimization}
M.~Hinze, R.~Pinnau, M.~Ulbrich, and S.~Ulbrich.
\newblock {\em Optimization with PDE constraints}, volume~23.
\newblock Springer Science \& Business Media, 2008.

\bibitem{kalnay2003atmospheric}
E.~Kalnay.
\newblock {\em {Atmospheric modeling, data assimilation, and predictability}}.
\newblock Cambridge university press, 2003.

\bibitem{knoll2011second}
F.~Knoll, K.~Bredies, T.~Pock, and R.~Stollberger.
\newblock {Second order total generalized variation (TGV) for MRI}.
\newblock {\em Magnetic resonance in medicine}, 65(2):480--491, 2011.

\bibitem{knuth1997art}
D.E. Knuth.
\newblock {\em The Art of Computer Programming: Volume 1: Fundamental
  Algorithms}.
\newblock Pearson Education, 1997.

\bibitem{lee2013bayesian}
J.~Lee and P.K. Kitanidis.
\newblock {Bayesian inversion with total variation prior for discrete geologic
  structure identification}.
\newblock {\em Water Resources Research}, 49(11):7658--7669, 2013.

\bibitem{lewisdataassim}
J.~M. Lewis, S.~Lakshmivarahan, and S.K. Dhall.
\newblock {\em {Dynamic data assimilation : a least squares approach}}.
\newblock Cambridge University Press, 2006.

\bibitem{siltanenbook}
J.~L. Mueller and S.~Siltanen.
\newblock {\em {Linear and Nonlinear Inverse Problems with Practical
  Applications}}.
\newblock Society for Industrial and Applied Mathematics, Philadelphia, PA, oct
  2012.

\bibitem{pfaff2015optimal}
S.~Pfaff and S.~Ulbrich.
\newblock {Optimal boundary control of nonlinear hyperbolic conservation laws
  with switched boundary data}.
\newblock {\em SIAM Journal on Control and Optimization}, 53(3):1250--1277,
  2015.

\bibitem{quarteroni2010numerical}
A.~Quarteroni, R.~Sacco, and F.~Saleri.
\newblock {\em Numerical mathematics}, volume~37.
\newblock Springer Science \& Business Media, 2010.

\bibitem{ulbrich2012nichtlineare}
M.~Ulbrich and S.~Ulbrich.
\newblock {\em Nichtlineare Optimierung}.
\newblock Springer-Verlag, 2012.

\bibitem{vogelinversebook}
C.R. Vogel.
\newblock {\em {Computational Methods for Inverse Problems}}.
\newblock Society for Industrial and Applied Mathematics, jan 2002.

\bibitem{wang2004image}
Z.~Wang, A.~C Bovik, H.R. Sheikh, and E.P. Simoncelli.
\newblock Image quality assessment: from error visibility to structural
  similarity.
\newblock {\em IEEE transactions on image processing}, 13(4):600--612, 2004.

\end{thebibliography}
